\documentclass[10pt,letterpaper]{amsart}
\usepackage[shortlabels]{enumitem}
\usepackage{amsfonts}
\usepackage{amssymb}
\usepackage{amsthm}
\usepackage{amsmath}
\usepackage{amscd}
\usepackage{t1enc}
\usepackage{lmodern}
\usepackage{graphicx}
\usepackage{graphics}
\usepackage{pict2e}
\usepackage{epic}
\numberwithin{equation}{section}
\usepackage[margin=2.9cm]{geometry}
\usepackage{epstopdf} 
\usepackage{bbm}
\usepackage{tikz-cd}
\usetikzlibrary{tikzmark,quotes}
\usepackage{cite}
\usepackage{mathrsfs}
\usepackage{xcolor}
\usepackage{printlen}
\usepackage{eqparbox}
\usepackage[numbers]{natbib}

\usepackage{microtype}
\usepackage[linktocpage=true,colorlinks=true,pagebackref,hyperindex]{hyperref}

\definecolor{burntumber}{rgb}{0.54, 0.2, 0.14}
\definecolor{coolblack}{rgb}{0.0, 0.18, 0.39}
\definecolor{mygreen}{rgb}{0.0, 0.4, 0.0}
\hypersetup{
	linkcolor=burntumber,
	citecolor=mygreen,
	urlcolor=burntumber
}

\setlist[enumerate]{leftmargin=2pc}
\setlist[itemize]{leftmargin=2pc}

\makeatletter
\def\l@subsection{\@tocline{2}{0pt}{2pc}{5pc}{}}
\makeatother

\theoremstyle{plain}
\newtheorem{thm}{Theorem}[section]
\newtheorem{lemma}[thm]{Lemma}
\newtheorem{cor}[thm]{Corollary}
\newtheorem{prop}[thm]{Proposition}

\newtheorem{introthm}{Theorem}

\newtheorem{introcor}[introthm]{Corollary}

 \theoremstyle{definition}
\newtheorem{Def}[thm]{Definition}

\newtheorem{rmk}[thm]{Remark}
\newtheorem{?}[thm]{Problem}
\newtheorem{ex}[thm]{Example}

\newcommand{\zz}{\mathbb{Z}}

\newcommand{\qq}{\mathbb{Q}}

\newcommand{\im}{\operatorname{im}}

\newcommand{\id}{\textup{id}}
\newcommand{\res}{\big|}

\newcommand{\Gdisp}{G\textup{-Disp}_{\underline{W},\mu}}
\newcommand{\GdispBA}{G\textup{-Disp}_{\underline{W}(B/A),\mu}}

\newcommand{\overbar}[1]{\mkern 1.5mu\overline{\mkern-1.5mu#1\mkern-1.5mu}\mkern 1.5mu}
\newcommand{\uS}{\underline{S}}
\newcommand{\Spec}{\textup{Spec }}
\newcommand{\Spf}{\textup{Spf }}

\newcommand{\icris}{{i_{\text{CRIS}}}}

\newcommand{\Z}{\textup{Z}}
\newcommand{\Fil}{\textup{Fil} \hspace{.15em}}

\newcommand{\marginalfootnote}[1]{
	\footnote{#1}
	\marginpar[{\hfill\sf\thefootnote}]{{\sf\thefootnote}}}
\newcommand{\edit}[1]{\marginalfootnote{#1}}

\LetLtxMacro\Oldedit\edit

\newcommand{\EnableEdits}{%
	\LetLtxMacro\edit\Oldedit%
}

\EnableEdits 

\begin{document}
\begin{abstract}
	Let $G$ be a reductive group scheme over the $p$-adic integers, and let $\mu$ be a minuscule cocharacter for $G$. In the Hodge-type case, we construct a functor from nilpotent $(G,\mu)$-displays over $p$-nilpotent rings $R$ to formal $p$-divisible groups over $R$ equipped with crystalline Tate tensors. When $R/pR$ has a $p$-basis \'etale locally, we show that this defines an equivalence between the two categories. The definition of the functor relies on the construction of a $G$-crystal associated with any adjoint nilpotent $(G,\mu)$-display, which extends the construction of the Dieudonn\'e crystal associated with a nilpotent Zink display.
	As an application, we obtain an explicit comparison between the Rapoport-Zink functors of Hodge type defined by Kim and by B\"ultel and Pappas.
\end{abstract}

\title{$G$-displays of Hodge type and formal $p$-divisible groups}

\author{Patrick Daniels}
\thanks{Research partially supported by NSF DMS-1801352 and DMS-1840234}
\address{Department of Mathematics, University of Michigan, Ann Arbor, MI 48109}

\email{pjdaniel@umich.edu}

\maketitle
\tableofcontents
\section{Introduction}
Fix a prime $p$, and let $G$ be a smooth affine group scheme over $\zz_p$ whose generic fiber is reductive. This paper contributes to the search for what it means to endow a $p$-divisible group with $G$-structure. When $G$ is a classical group coming from a local Shimura datum of EL- or PEL-type, to equip a $p$-divisible group with $G$-structure is to decorate it with additional structures coming from the data which cuts out $G$ inside of some general linear group, such as a polarization or an action by a semisimple algebra. Moduli spaces of $p$-divisible groups with additional structure define Rapoport-Zink formal schemes, whose rigid analytic generic fibers determine local analogs of Shimura varieties. 

Recently, Scholze and Weinstein \cite{SW2020} have developed a general theory of local Shimura varieties. Unlike in the EL- and PEL-type cases, however, the general theory takes place entirely in the generic fiber and leaves open the question of whether there exist formal schemes which act as integral models. One would expect moduli spaces of $p$-divisible groups with $G$-structure to define integral models in all cases, as they do in the EL- and PEL-type cases. However, due to the lack of a natural tensor product on the category of $p$-divisible groups, the traditional methods for defining $G$-structure do not simply carry over to this case. In particular, any kind of Tannakian approach is not straightforward.

In this paper we restrict our focus to the case where the pair $(G,\mu)$ is of Hodge type, i.e., where there is a closed embedding $\eta: G \hookrightarrow \textup{GL}(\Lambda)$ such that the cocharacter $\eta \circ \mu$ is conjugate to a standard minuscule cocharacter for $\textup{GL}(\Lambda)$. In this case, there are two approaches to endowing $p$-divisible groups with $G$-structure which have enjoyed some success in providing functor of points descriptions of Rapoport-Zink formal schemes. In the first approach, one uses the embedding $G \hookrightarrow \textup{GL}(\Lambda)$ to define additional structures on tensor powers of the Dieudonn\'e crystal of the given $p$-divisible group. In the second, one replaces $p$-divisible groups with Zink's displays, which are linear-algebraic objects and therefore more readily equipped with $G$-structure. The main result of this paper is that, at least under certain restrictions on the base ring (see Theorem \ref{thm-1} below), these two approaches are equivalent.

Let us describe the two approaches in more detail. If $(G,\mu)$ is of Hodge type, then $G$ is the element-wise stabilizer in $\textup{GL}(\Lambda)$ of a finite collection $\underline{s} = (s_1, \dots, s_r)$ of elements of the total tensor algebra of $\Lambda \oplus \Lambda^\vee$, which we denote by $\Lambda^\otimes$. We call the tuple $\underline{G} = (G,\mu,\Lambda,\eta, \underline{s})$ a local Hodge embedding datum. If $X$ is a $p$-divisible group over a $p$-nilpotent $\zz_p$-algebra $R$, then a crystalline Tate tensor is a morphism of crystals $t: \mathbbm{1} \to \mathbb{D}(X)^\otimes$ over $\Spec R$ which preserves the Hodge filtrations and which is equivariant for the action of the Frobenius, up to isogeny. Here $\mathbbm{1}$ denotes the unit object in the tensor category of crystals of finite locally free $\mathcal{O}_{\Spec{R}/\zz_p}$-modules, and $\mathbb{D}(X)$ denotes the covariant Dieudonn\'e crystal of $X$ as in \cite{BBM1982}. A $p$-divisible group with $(\underline{s},\mu)$-structure over $R$ is a pair $(X,\underline{t})$ consisting of a $p$-divisible group $X$ over $\Spec R$ whose Hodge filtration is \'etale locally determined by $\mu$, and a collection $\underline{t} = (t_1, \dots, t_r)$  of crystalline Tate tensors which are fppf-locally identified with $\underline{s}$ (see Definition \ref{def-smu}). The main theorem of \cite{Kim2018} (see also \cite{HP2017}) states that, when $G$ is reductive and $(G,\mu)$ is of Hodge type, a Rapoport-Zink formal scheme can be defined which is roughly a moduli space of $p$-divisible groups with $(\underline{s},\mu)$-structure. 

On the other hand, the idea of using group-theoretic analogs of Zink's displays to define Rapoport-Zink spaces originally appears in \cite{BP2017}. There, a theory of $(G,\mu)$-displays is developed for pairs $(G,\mu)$ such that $G$ is reductive and $\mu$ is minuscule, and the theory is used to give a purely group-theoretic definition of Rapoport-Zink formal schemes of Hodge type. Subsequently, Lau generalized the theory of $(G,\mu)$-displays \cite{Lau2018}, and an equivalent Tannakian framework was developed in the author's previous paper \cite{Daniels2019}. Denote by $\textup{Disp}(\underline{W}(R))$ the category of higher displays over the Witt frame for $R$ as in \cite{Lau2018}. In the Tannakian framework, we say a $(G,\mu)$-display over $\underline{W}(R)$ is an exact tensor functor $\textup{Rep}_{\zz_p}G \to \textup{Disp}(\underline{W}(R))$ such that, for every representation, the structure of the corresponding display is fpqc-locally governed by the cocharacter $\mu$ (see Definitions \ref{def-typemu} and \ref{def-gmu}). This formulation is essential to the results of this paper, as it allows for a close connection with Zink's original theory of displays \cite{Zink2002}, and therefore also with $p$-divisible groups.

When $G = \textup{GL}_h$ and $\mu = \mu_{d,h}$ is the cocharacter $t \mapsto (1^{(d)},t^{(h-d)})$, $(G,\mu)$-displays are nothing but Zink displays of height $h$ and dimension $d$. When $(G,\mu)$ is of Hodge type, the embedding $\eta: G \hookrightarrow \textup{GL}(\Lambda)$ induces a functor from $(G,\mu)$-displays to Zink displays, and in this case we say a $(G,\mu)$-display is nilpotent with respect to $\eta$ if the corresponding Zink display is nilpotent in the sense of \cite{Zink2002}. The following is the main result of this paper, see Theorem \ref{mainthm}.

\begin{introthm}\label{thm-1}
	Let $G$ be a reductive group scheme over $\zz_p$, and let $\mu$ be a minuscule cocharacter for $G$. For every $p$-nilpotent $\zz_p$-algebra $R$, there is a functor
	\begin{align*}
	\textup{BT}_{\underline{G},R}: 
	\left( 
	\begin{array}{c}
	\textup{$(G,\mu)$-displays over $\underline{W}(R)$} \\ \textup{which are nilpotent with respect to $\eta$}
	\end{array}
	\right)
	\to \left(
	\begin{array}{c}
	\textup{formal $p$-divisible groups over $R$} \\
	\textup{with $(\underline{s},\mu)$-structure}
	\end{array}
	\right).
	\end{align*}
	If $R/pR$ has a $p$-basis \'etale locally, then $\textup{BT}_{\underline{G},R}$ is an equivalence of categories.
\end{introthm}

In particular, the equivalence in Theorem \ref{thm-1} holds when $R$ is a field of characteristic $p$, or when $R/pR$ is a regular, Noetherian, and $F$-finite (the latter by \cite[Lem. 2.1]{Lau2018b}; see Definition \ref{def-pbasis} for the definition of a $p$-basis). When $G = \textup{GL}_h$ and $\mu = \mu_{d,h}$, the equivalence in question holds for arbitrary $p$-nilpotent rings $R$ by a theorem of Zink and Lau (see \cite{Zink2002} and \cite{Lau2010}). Hence the main result of this paper can be seen as a group-theoretic generalization of the theorem of Zink and Lau (but note that the theorem of Zink and Lau is an invaluable input in the proof). Let us also mention that a similar result is proven in the case where $G$ and $\mu$ come from an EL-type local Shimura datum in \cite[\textsection 5.4]{Daniels2019}.

Given a $(G,\mu)$-display $\mathscr{P}$ which is nilpotent with respect to $\eta$, it is straightforward to obtain a formal $p$-divisible group: one takes the $p$-divisible group $X$ associated with the Zink display induced by the embedding $\eta:G \hookrightarrow \textup{GL}(\Lambda)$. The primary difficulty lies in determining an $(\underline{s},\mu)$-structure on $X$. This is resolved by the main innovation of this paper, which is the association of a $G$-crystal to the $(G,\mu)$-display $\mathscr{P}$. We summarize the properties of this $G$-crystal in the following theorem, which is an amalgamation of the results in \textsection \ref{sub-crystalGdisp} and \textsection \ref{sub-hodge}. Let $\textup{LFCrys}(\Spec R/\zz_p)$ denote the category of crystals in locally free $\mathcal{O}_{\Spec R/\zz_p}$-modules as in \cite{BBM1982}. 

\begin{introthm}\label{thm-Gcrystal}
	Let $R$ be a $p$-nilpotent $\zz_p$-algebra. Suppose $\mathscr{P}$ is a $(G,\mu)$-display over $\underline{W}(R)$ which is nilpotent with respect to $\eta$. Then there exists an exact tensor functor
	\begin{align*}
	\mathbb{D}(\mathscr{P}): \textup{Rep}_{\zz_p} G \to \textup{LFCrys}(\Spec R/\zz_p), \  (V,\pi) \mapsto \mathbb{D}(\mathscr{P})^\pi
	\end{align*}
	such that the following properties hold:
	\begin{enumerate}[$($i$)$]
		\item The association $\mathscr{P} \mapsto \mathbb{D}(\mathscr{P})$ is functorial in $\mathscr{P}$ and compatible with base change.
		\item If $\Z_{\eta}(\mathscr{P})$ is the nilpotent Zink display associated with $\mathscr{P}$ via the embedding $\eta$, then there is a natural isomorphism of crystals
		\begin{align*}
		\mathbb{D}(\mathscr{P})^\eta \cong \mathbb{D}(\Z_\eta(\mathscr{P})),
		\end{align*}
		where $\mathbb{D}(\Z_\eta(\mathscr{P}))$ denotes the crystal associated with $\Z_\eta(\mathscr{P})$ as in \textup{\cite{Zink2002}}.
	\end{enumerate}
\end{introthm}
 
Once such a crystal is constructed, it is not difficult to obtain an $(\underline{s},\mu)$-structure on the $p$-divisible group $X$ associated with $\mathscr{P}$. Indeed, by viewing the tensors $s_i$ as morphisms of representations from the trivial representation to $\Lambda^\otimes$, we can use functoriality of the $G$-crystal in representations and its compatibility with the tensor product to obtain morphisms $t_i:\mathbbm{1} \to (\mathbb{D}(\mathscr{P})^\eta)^\otimes$. By Theorem \ref{thm-Gcrystal}, we can replace $(\mathbb{D}(\mathscr{P})^\eta)^\otimes$ with $\mathbb{D}(\Z_\eta(\mathscr{P}))^\otimes$, which is in turn isomorphic to $\mathbb{D}(X)^\otimes$ by the theory of Zink and Lau (Lemma \ref{lem-zinkdieudonne}). With some work (see Proposition \ref{prop-frobeq}) one can show that the resulting morphisms of crystals $t_i: \mathbbm{1} \to \mathbb{D}(X)^\otimes$ are crystalline Tate tensors.

The proof of Theorem \ref{thm-1} then proceeds in two steps. First, the case where $pR = 0$ is dealt with using the strategy of \cite[Thm. 5.15]{Daniels2019}, see Proposition \ref{prop-pR=0}. The case of general $R$ is then reduced to this one using analogs Grothendieck-Messing theory developed in the settings of $G$-displays and of $p$-divisible groups with $(\underline{s},\mu)$-structure, respectively.

The construction of the crystal in Theorem \ref{thm-Gcrystal} requires a technical result about $(G,\mu)$-displays which may be of independent interest. As a starting point we recall that by \cite[Thm. 3.16]{Daniels2019}, if $R$ is a $p$-nilpotent $\zz_p$-algebra, $(G,\mu)$-displays over the Witt frame $\underline{W}(R)$ in the Tannakian framework are equivalent to $G$-displays of type $\mu$ over $\underline{W}(R)$ defined using the torsor-theoretic framework of \cite{Lau2018}. More generally, if $\underline{S}$ is an \'etale sheaf of frames over $\Spec R$, we can define a category of $G$-displays of type $\mu$ over $\underline{S}(R)$ as in \cite{Lau2018}, and a category of $(G,\mu)$-displays over $\underline{S}(R)$ following the Tannakian formulation of \cite{Daniels2019}. We say $\underline{S}$ satisfies descent for modules if finite projective graded modules over the graded ring $S$ form an \'etale stack over $\Spec{R}$. The following theorem (Theorem \ref{thm-equiv} below), which is essentially a generalization of \cite[Thm. 3.16]{Daniels2019}, is critical to our construction of a $G$-crystal for a $(G,\mu)$-display.

\begin{introthm}\label{thm-2}
	If $\underline{S}$ is an \'etale sheaf of frames on $\Spec{R}$ which satisfies descent for modules, then there is an equivalence of categories
	\begin{align*}
	\big(\textup{$(G,\mu)$-displays over $\underline{S}(R)$}\big) \xrightarrow{\sim} \big(\textup{$G$-displays of type $\mu$ over $\underline{S}(R)$}\big).		
	\end{align*}
\end{introthm} 

In particular, in Appendix \ref{section-appendix}, we prove that the sheaf on $\Spec A$ associated with the relative Witt frame $\underline{W}(B/A)$ for a $p$-adic PD-thickening $B \to A$ satisfies descent for modules. Other sheaves of frames which satisfy descent for modules are those associated with $p$-adic frames as in \cite[Def. 4.2.1]{Lau2018}. Examples of $p$-adic frames include the Zink frame $\mathbb{W}(R)$ for an admissible local ring $R$ \cite[Ex. 2.1.13]{Lau2018} and its relative analog associated with a PD-thickening $B \to R$, as well as the truncated Witt frames $\underline{W}_n(R)$ over an $\mathbb{F}_p$-algebra $R$ \cite[Ex. 2.1.6]{Lau2018}. 

Let us briefly sketch the construction of the $G$-crystal. Given a $(G,\mu)$-display $\mathscr{P}$ over $R$, we obtain a corresponding $G$-display of type $\mu$ over $\underline{W}(R)$ using \cite[Thm. 3.16]{Daniels2019}. Moreover, if $B \to R$ is a PD-thickening, then the work of Lau (see Proposition \ref{prop-lifting}) allows us to lift the $G$-display of type $\mu$ over $\underline{W}(R)$ to a $G$-display of type $\mu$ over the relative Witt frame $\underline{W}(B/R)$. Since the relative Witt frame satisfies descent for modules (see Proposition \ref{prop-descent}), the above theorem applies, and we obtain a $(G,\mu)$-display over $\underline{W}(B/R)$, which is, in particular, an exact tensor functor from $\textup{Rep}_{\zz_p} G$ to the category of displays over $\underline{W}(B/R)$. Given any representation, we can obtain from such an object a $B$-module, which we denote $\mathbb{D}(\mathscr{P})_{B/R}^\rho$. The functor which assigns to $(V,\pi)$ the crystal $(B \to R) \mapsto \mathbb{D}(\mathscr{P})^\pi_{B/R}$ is the desired $G$-crystal.

As a consequence of the Theorem \ref{thm-1} we obtain an explicit relationship between the Rapoport-Zink functors of Hodge type defined in \cite{Kim2018} and in \cite{BP2017}. To be more specific, let $k$ be an algebraic closure of $\mathbb{F}_p$, suppose $(G,\{\mu\},[b])$ is an integral local Shimura datum which is unramified of Hodge type, and let $\underline{G} = (G,\mu, \Lambda, \eta, \underline{s})$ be a local Hodge embedding datum. Given a good choice of $\mu$ and $b$, we can define a $(G,\mu)$-display $\mathscr{P}_0$ which is nilpotent with respect to $\eta$, and we denote by $(X_0, \underline{t_0})$ its associated formal $p$-divisible group with $(\underline{s},\mu)$-structure. Denote by $\textup{Nilp}_{W(k)}^\textup{fsm}$ the category of $p$-nilpotent $W(k)$-algebras which are formally smooth and formally finitely generated over $W(k)/p^mW(k)$ for some $m$. To the datum $(\underline{G},b)$ we can associate two Rapoport-Zink functors on $\textup{Nilp}_{W(k)}^\textup{fsm}$. The first, denoted $\textup{RZ}_{\underline{G},b}^{p\textup{-div,fsm}}$, assigns to a $p$-nilpotent $W(k)$-algebra the set of isomorphism classes of triples $(X,\underline{t},\iota)$, where $(X,\underline{t})$ is a $p$-divisible group with $(\underline{s},\mu)$-structure over $A$, and $\iota$ is a quasi-isogeny over $\Spf A/pA$ between $X$ and $X_0$ which respects the tensors modulo an ideal of definition. The second, denoted $\textup{RZ}_{G,\mu,b}^{\textup{disp}}$, assigns to such rings the set of isomorphism classes of pairs $(\mathscr{P},\rho)$ consisting of a $(G,\mu)$-display $\mathscr{P}$ over $\underline{W}(R)$ and a $G$-quasi-isogeny $\rho$ between $\mathscr{P}$ and $\mathscr{P}_0$ which is defined over $\Spf A/pA$ (see \textsection \ref{sub-rzspaces} for details). 

By \cite[Lem. 2.1]{Lau2018}, if $A$ is an object in $\textup{Nilp}_{W(k)}^\textup{fsm}$, then $A/pA$ has a $p$-basis \'etale locally, so the equivalence of Theorem \ref{thm-1} holds. As a result, we obtain the following corollary (see Theorem \ref{thm-RZfunctors}).

\begin{introcor}\label{introcor1}
	The functors $\textup{RZ}_{\underline{G},b}^{p\textup{-div,fsm}}$ and $\textup{RZ}_{G,\mu,b}^{\textup{disp,fsm}}$ on $\textup{Nilp}_{W(k)}^\textup{fsm}$ are naturally isomorphic.
\end{introcor}

It follows from Corollary \ref{introcor1} that the formal schemes defined by Kim \cite{Kim2018} and B\"ultel and Pappas \cite{BP2017} which represent these functors are isomorphic. This was already known by \cite[Remark 5.2.7]{BP2017}, but the geometric method of proof offered in \textit{loc. cit.} differs from the explicit comparison of functors given here. 

Let us give a brief outline of the paper. In the first section we review the definitions of displays and frames as in \cite{Lau2018}, and the crystalline theory of $p$-divisible groups and displays, following especially \cite{BBM1982}, \cite{Zink2002}, and \cite{Lau2013}. In \textsection \ref{section-Gdisplays} we recall basic notions about $G$-displays of type $\mu$ and $(G,\mu)$-displays, and we prove Theorem \ref{thm-2}. By results in Appendix \ref{section-appendix}, the theorem applies in particular in the case of relative Witt frames, which is in turn crucial for \textsection \ref{section-crystals}, where we construct the $G$-crystal associated with a $(G,\mu)$-display over the Witt frame and prove the collection of results which comprise Theorem \ref{thm-Gcrystal}. In \textsection \ref{section-main} we prove Theorem \ref{thm-1}, and derive consequences for the study of Rapoport-Zink spaces of Hodge type and for the deformation theory of $p$-divisible groups with crystalline Tate tensors.

\subsection{Acknowledgements}
I would like to thank Thomas Haines for his helpful suggestions on an early draft of this paper, and Eike Lau for explaining some aspects of his work to me. Moreover, I'd like to thank the anonymous referee for their extremely careful reading of this manuscript and for the many helpful comments they provided. 

\subsection{Conflict of interest statement} The author states that there is no conflict of interest.

\subsection{Data availability statement} Data sharing is not applicable to this article as no datasets were generated or analyzed during the study.

\subsection{Notation}
\begin{itemize}[leftmargin=*]
	\item Throughout the paper, fix a prime $p$ and a finite field $k_0$ of characteristic $p$ and cardinality $q = p^\ell$.
	 	
	\item A ring or abelian group will be called $p$-adic if it is complete and separated with respect to the $p$-adic topology. 
	
	\item If $f: A \to B$ is a ring homomorphism and $M$ is an $A$-module, we write $f^\ast M$ for $M \otimes_{A,f} B$. If $f$ is understood, we write $M_B = M\otimes_A B$ as well. If $X$ is a $p$-divisible group over $\Spec A$, we often write $X \otimes_A B$ for the base change of $X$ to $B$.
	
	\item If $f: A \to B$ is a ring homomorphism, $M$ is an $A$-module, and $N$ is a $B$-module, we say a map $\alpha: M \to N$ is $f$-linear if $\alpha(a\cdot m) = f(a) \cdot \alpha(m)$ for $a \in A$, $m \in M$. In this case we write $\alpha^\sharp$ for the linearization $f^\ast M \to N$ given by $m \otimes b \mapsto \alpha(m)\cdot b$. We say $\alpha$ is an $f$-linear bijection if $\alpha^\sharp$ is a $B$-module isomorphism. 
	
	\item If $R$ is a commutative ring, denote by $\textup{Mod}(R)$ the category of $R$-modules.
	
	\item For a $\zz_p$-algebra $\mathcal{O}$, denote by \text{Nilp}$_{\mathcal{O}}$ the category of $\mathcal{O}$-algebras in which $p$ is nilpotent. We will refer to such an $\mathcal{O}$-algebra as a $p$-nilpotent $\mathcal{O}$-algebra.
	
	\item Let $\bigoplus S_n$ be a $\zz$-graded ring. For a ring homomorphism $\varphi: \bigoplus S_n \to R$, we write $\varphi_n$ for the restriction of $\varphi$ to $S_n$. 
	
	\item Let $R$ be a ring. Denote by $\text{\'Et}_R$ the category of affine \'etale $R$-schemes. We endow this category with a topology by defining a covering of $\Spec{A} \in \text{\'Et}_R$ to be an \'etale covering $\{U_i \to \Spec{A}\}$ such that each $U_i$ is affine. 
	
	\item If $G$ is a sheaf of groups in a topos, denote by $\text{Tors}_G$ the stack of $G$-torsors.
	
	\item If $S$ is any $\zz$-graded ring, denote by $\text{GrMod}(S)$ the category of graded $S$-modules, and by $\text{PGrMod}(S)$ the full subcategory of finite projective graded $S$-modules. By \cite[Lemma 3.0.1]{Lau2018}, this latter category is equivalent to the full subcategory of finitely generated graded $S$-modules which are projective over $S$. 
	
	\item If $R$ is a $p$-adic ring, denote by \text{pdiv}$(R)$ the category of $p$-divisible groups over $R$, and denote by $\text{fpdiv}(R)$ the full subcategory of formal $p$-divisible groups over $R$.
	
	\item If $M$ is a module over a ring $R$, denote by $M^\vee$ its linear dual. 
	
	\item If $X$ is a $p$-divisible group over a ring $R$, denote by $X^D$ its Serre dual.
	
	\item Let $R$ be a $p$-adic $W(k_0)$-algebra. If $A$ is an $R$-algebra, a $p$-adic PD-thickening of $A$ over $R$ is a surjective ring homomorphism $B \to A$ such that $B$ is a $p$-adic $W(k_0)$-algebra and such that the kernel $J$ of $B \to A$ is equipped with divided powers $\delta$ which are compatible with the canonical divided powers on $pW(k_0)$.
	
	\item A PD-morphism between PD-thickenings $B\to A$ and $B'\to A'$ with divided powers $\delta$ and $\delta'$ on their respective kernels is a pair of homomorphisms $\varphi: B \to B'$ and $\psi: A \to A'$ such that the obvious diagram commutes, and such that $\delta_n'(\varphi(x)) = \varphi(\delta_n(x))$ for all $x \in J$ and all $n$.
\end{itemize}
\section{Preliminaries}
\subsection{Frames, graded modules, and displays}\label{sub-frames}
We review the basic definitions and properties of (higher) frames and displays following \cite{Lau2018} and \cite[\textsection 2]{Daniels2019}. In particular, we recall the definition of the Witt frame over a $p$-nilpotent ring $R$ (Example \ref{ex-wittframe}) and the relative Witt frame associated with a $p$-adic PD-thickening $B \to A$ (Example \ref{ex-relwittframe}). Moreover, we make explicit the connection between the theory of displays presented here and the theory of windows (see Lemma \ref{lem-windows}). 

\begin{Def}\label{def-frame}
	A \textit{frame} $\underline{S} = (S,\sigma,\tau)$ is a triple consisting of a $\zz$-graded ring $S$ and two ring homomorphisms $\sigma, \tau: S \to S_0$ satisfying the following properties.
	\begin{enumerate}[(i)]
		\item The endomorphism $\tau_0$ of $S_0$ is the identity, and $\tau_{-n}: S_{-n} \to S_0$ is a bijection for all $n \ge 1$.
		\item The endomorphism $\sigma_0$ of $S_0$ induces the $p$-power Frobenius $s \mapsto s^p$ on $S_0 / pS_0$, and if $t$ is the unique element in $S_{-1}$ such that $\tau_{-1}(t) = 1,$ then $\sigma_{-1}(t) = p$.
		\item We have $p \in \textup{Rad}(S_0)$. 
	\end{enumerate}
\end{Def}

We say that $\underline{S}$ is a frame for $R = S_0 / \tau(S_1)$. The conditions in the definition imply that $\tau$ acts on $S_1$ as multiplication by $t$, so we will usually write $\tau(S_1) = tS_1$. 

\begin{Def} Let $\underline{S} = (S,\sigma,\tau)$ be a frame for $R$. A \textit{display} over $\underline{S}$ is a pair $\underline{M} = (M,F)$ consisting of a finite projective graded $S$-module $M$ and a $\sigma$-linear bijection $F:M \to \tau^\ast M$.
\end{Def}

\begin{Def}
	A \textit{standard datum} for a display is a pair $(L, \Phi)$ consisting of a finite projective graded $S_0$-module $L$ and a $\sigma$-linear automorphism $\Phi: L \to L$. 
\end{Def}

From a standard datum $(L, \Phi)$ we define a display $(M,F)$ where $M = L \otimes_{S_0} S$ and $F(x \otimes s) = \sigma(s) \Phi(x)$. If every $M$ in $\textup{PGrMod}(S)$ is of the form $M = L \otimes_{S_0} S$ for a finite projective $S_0$-module $L$, then every display is isomorphic to one defined by a standard datum, see \cite[\textsection 3.4]{Lau2018}. In particular, by \cite[Lem. 3.1.4]{Lau2018}, this occurs if every finite projective $R$-module lifts to $S_0$.

Let us denote by $\textup{Disp}(\uS)$ the category of displays over $\uS$. If $\uS \to \uS'$ is a frame homomorphism, then we obtain a base change functor 
\begin{align*}
\textup{Disp}(\uS) \to \textup{Disp}(\uS'), \ \underline{M} \mapsto \underline{M} \otimes_{\uS} \uS'.
\end{align*} 
The category $\textup{Disp}(\uS)$ has a tensor product given by $(M,F) \otimes(M',F') = (M\otimes_S M', F \otimes F')$, which makes it into an exact rigid tensor category with unit object $\underline{S} = (S,\sigma)$. 

For any display $\underline{M}$ over a frame $\underline{S}$, there exists a canonical descending filtration on $\overline{M}:=\tau^\ast M\otimes_{S_0} R$, called the Hodge filtration of $\underline{M}$, see \cite[\textsection 5.2]{Lau2018}. Let us recall the definition. Denote by $\theta_n: M_n \to \tau^\ast M$ the composition $M_n \hookrightarrow M \to \tau^\ast M$, and by $\bar{\theta}_n$ the composition 
\begin{align*}
M_n \xrightarrow{\theta_n} \tau^\ast M \to \tau^\ast M \otimes_{S_0} R.
\end{align*}
The $n$th piece of the \textit{Hodge filtration} is given by $\im(\bar{\theta}_n)$, and is denoted $\Fil^n(\underline{M})$:
\begin{align}\label{eq-hodgedisp}
\textup{Fil}^n(\underline{M}) := \textup{im}(\bar{\theta}_n) \subset \overline{M}.
\end{align} 
Since $\bar{\theta}_n$ factors through $t: M_n \to M_{n-1}$, we have $\textup{Fil}^n(\underline{M}) \subset \textup{Fil}^{n-1}(\underline{M})$, so $(\textup{Fil}^n(\underline{M}))_{n \in \zz}$ defines a descending filtration on $\overline{M}$. If $(L,\Phi)$ is a standard datum for $\underline{M}$, with $L = \bigoplus_i L_i$, then 
\begin{align*}
\textup{Fil}^n(\underline{M}) = \bigoplus_{i \ge n} L_i \otimes_{S_0} R.
\end{align*}

\begin{rmk}\label{rmk-hodge} 
	The Hodge filtration is functorial in $\underline{M}$, meaning that any morphism of displays $\varphi: \underline{M} \to \underline{M'}$ induces a morphism of $R$-modules $\bar{\varphi}:\overline{M} \to \overline{M'}$ which sends $\textup{Fil}^n(\underline{M})$ into $\textup{Fil}^n(\underline{M'})$. Moreover, the Hodge filtration is compatible with tensor products of displays, i.e., we have
	\begin{align}\label{eq-hodgetensor}
	\textup{Fil}^n(\underline{M} \otimes \underline{M'}) = \sum_{j+k = n} \textup{Fil}^j(\underline{M}) \otimes \textup{Fil}^k(\underline{M'}).
	\end{align}
\end{rmk}

Let us briefly review the connection between (higher) displays over frames and windows over 1-frames. Recall the following definition (see \cite[Def. 2.2.1]{Lau2018}).

\begin{Def}
	A $1$-frame $\mathcal{S} = (S_0 \supset I, \sigma_0, \dot{\sigma})$ consists of a ring $S_0$, an ideal $I \subset S_0$, a ring endomorphism $\sigma_0$ of $S_0$, and a $\sigma_0$-linear map $\dot{\sigma}: I \to S_0$ such that
	\begin{enumerate}[(i)]
		\item $\sigma_0: S_0 \to S_0$ is a lift of the Frobenius on $S_0/pS_0$,
		\item $\sigma_0(a) = p\dot{\sigma}(a)$ for $a \in I$, 
		\item $p \in \textup{Rad}(S_0)$.
	\end{enumerate}
	We say that $\mathcal{S}$ is a 1-frame for $R = S_0/I$.
\end{Def}

A frame $\underline{S}$ is said to extend the 1-frame $\mathcal{S}$ if $t:S_1 \to S_0$ is injective, $I = tS_1$, and $\dot{\sigma}(ta) = \sigma_1(a)$ for $a \in S_1$. 

\begin{Def}
	Let $\mathcal{S}$ be a $1$-frame. A window over $\mathcal{S}$ is a quadruple $\underline{P} = (P,\Fil P,F_0,F_1)$ consisting of a finitely generated projective $S_0$-module $P$, an $S_0$-submodule $\Fil P \subseteq P$, and two $\sigma_0$-linear maps $F_0:P \to P$ and $F_1: \Fil P \to P$ such that
	\begin{enumerate}[(i)]
		\item there is a decomposition $P = L_0 \oplus L_1$ with $\Fil P = L_0 \oplus IL_1$,
		\item $F_1(ax) = \dot{\sigma}(a)F_0(x)$ for $a \in I$ and $x \in P$,
		\item $F_0(x) = pF_1(x)$ for $x \in \Fil P$, 
		\item $F_0(P)+F_1(\Fil P)$ generates $P$ as an $S_0$-module.
	\end{enumerate}
\end{Def}

Because there is no surjectivity condition on $\dot{\sigma}$ in the definition of a 1-frame, the definition of windows that appears here differs slightly from others in the literature, see \cite[Rmk. 2.11]{Lau2010}. By \cite[Lem. 2.6]{Lau2010}, if $P = L_0 \oplus L_1$ is a finite projective $S_0$-module and $\Fil P = IL_0 \oplus L_1$, then the set of $\mathcal{S}$-window structures on $P$ and $\Fil P$ is mapped bijectively to the set of $\sigma_0$-linear isomorphisms
\begin{align*}
\Psi: L_0 \oplus L_1 \to P.
\end{align*}
The bijection is determined by $\Psi = F_0 \res_{L_0} \oplus F_1 \res_{L_1}$, and the triple $(L_0, L_1, \Psi)$ is called a \textit{normal representation} for $(P, \Fil P, F_0, F_1)$. We remark for later use that, in terms of $\Psi$, the linearization of $F_0$ can be expressed as
\begin{align}\label{eq-F0sharp}
F_0^\sharp = \Psi^\sharp \circ (\id_{L_0} \oplus p\cdot \id_{L_1}).
\end{align}

To any window $\underline{P} = (P, \Fil P, F_0, F_1)$ over $\mathcal{S}$ we can associate an $S_0$-module homomorphism $V^\sharp: P \to \sigma_0^\ast P$ which is uniquely determined by the identities
\begin{align*}
V^\sharp(\xi \cdot F_0(x)) = p\xi \otimes x \text{ and } V^\sharp(\xi \cdot F_1(y)) = \xi \otimes y 
\end{align*}
for $\xi \in S_0, x \in P, y \in \Fil P$. If $(L_0, L_1, \Psi)$ is a normal representation for $\underline{P}$, then 
\begin{align}\label{eq-Vsharp}
V^\sharp = (p\cdot \id_{L_0} \oplus \id_{L_1}) \circ (\Psi^\sharp)^{-1}.
\end{align}
From (\ref{eq-F0sharp}) and (\ref{eq-Vsharp}) we see that $F_0^\sharp \circ V^\sharp = p\cdot \id_{P}$ and $V^\sharp \circ F_0^\sharp = p\cdot \id_{\sigma_0^\ast P}$.

If $\underline{S}$ is a frame for $R$, denote by $\nu$ the ring homomorphism $S \to R$ which extends the projection $S_0 \to R$ by zero on $S_n$ for $n \ne 0$. If $M$ is a finite projective graded module over $S$, then $M \otimes_{S,\nu} R$ is a finite projective graded $R$-module with graded pieces that we denote by $\overline{L}_i$. Recall the following definition, cf. \cite[Def. 2.16]{Daniels2019}.

\begin{Def}
	We say $\underline{M}$ is a \textit{1-display} over $\underline{S}$ if $\overline{L}_i = 0$ for all $i < 0$ and $i > 1$. 
\end{Def}

In the language of \cite{Daniels2019}, $\underline{M}$ is a 1-display if the depth of $M$ is nonnegative and the altitude of $M$ is less than 1. If $(L, \Phi)$ is a standard datum for $\underline{M}$, then $\underline{M}$ is a 1-display if and only if $L_i = 0$ for all $i < 0$ and $i > 1$, see \cite[Lem. 2.7]{Daniels2019}. 

\begin{lemma}\label{lem-windows}
	Suppose $\underline{S}$ is a frame extending the 1-frame $\mathcal{S}$, and suppose that all finitely projective $R$-modules lift to $S_0$. Then the functor
	\begin{align}\label{eq-displaytowindow}
		P_{\underline{S}}: (1\text{-displays over }\underline{S}) \to (\text{Windows over }\mathcal{S}),
	\end{align}
	defined by assigning to a 1-display $\underline{M} = (M,F)$ the window $(P, \Fil P, F_0, F_1)$ with $P = \tau^\ast M, \Fil P = \theta_1(M_1)$, $F_0 = F\res_{M_0} \circ \theta_0^{-1}: P \to P$, and $F_1 = F\res_{M_1} \circ \theta_1^{-1}: \Fil P \to P$ is an equivalence of categories.
\end{lemma}
\begin{proof}
	The proof follows from a straightforward adaptation of the arguments in \cite[Lem. 2.25]{Daniels2019}. 
\end{proof}

One can also prove the lemma using normal representations: if $(L, \Phi)$ is a standard datum for a 1-display over $\underline{S}$, then $L = L_0 \oplus L_1$, so $(L_0, L_1, \Phi)$ is a normal representation for the associated window over $\mathcal{S}$. For use later, let us denote the quasi-inverse functor to $P_{\underline{S}}$ by
\begin{align}\label{eq-windowtodisplay}
M_{\underline{S}}: (\text{Windows over }\mathcal{S}) \to (1\text{-displays over }\underline{S}).
\end{align}

We close this section by discussing a few example of frames and 1-frames that will be of particular importance in what follows. Recall that to give a frame it suffices to specify a triple $(S_{\ge 0}, \sigma, (t_n)_{n \ge 0})$ consisting of a $\zz_{\ge 0}$-graded ring $S_{\ge 0}$, a ring homomorphism $\sigma: S_{\ge 0} \to S_0$, and a maps $t_n: S_{n+1} \to S_n$, see \cite[\textsection 2.1]{Daniels2019} and \cite[Rmk. 2.0.2]{Lau2018}. 

If $R$ is a $\zz_p$-algebra, let $W(R)$ denote the ring of infinite length Witt vectors over $R$. The ring $W(R)$ comes equipped with a ring endomorphism $f_R$, called the Frobenius, and an additive self-map $v_R$, called the Verschiebung. When the ring $R$ is clear from context we will write simply $f$ and $v$ for these maps. Denote by $I(R)$ the kernel of the canonical map $w_0: W(R) \to R$, so $I(R) = v(W(R))$. 

\begin{ex}[The Witt frame]\label{ex-wittframe}
	Let $R$ be a $p$-adic ring. Define a frame $\underline{W}(R)$ from the Witt ring over $R$ as follows. Define $S_0 = W(R)$, and for $n \ge 1$, let $S_n = I(R)$, viewed as an $S_0$-module. We define a $\zz_{\ge 0}$-graded ring structure on $S_{\ge 0} = S_0 \oplus \bigoplus_{n > 0} S_n$ by endowing it with the multiplication $S_n \times S_m \to S_{n +m}$ determined by $(v(a), v(b)) \mapsto v(ab)$ for $n, m \ge 1$. The map $t_0: S_1 \to S_0$ is given by inclusion $I(R) \hookrightarrow W(R)$, and $t_n$ is multiplication by $p$ for $n \ge 1$. We let $\sigma_0 = f_R$, and for $n \ge 1$ we define $\sigma_n(v(s)) = s$ for all $v(s) \in S_n = I(R)$. We will write $S = W(R)^\oplus$ for the resulting $\zz$-graded ring. The corresponding frame $\underline{W}(R) = (W(R)^\oplus, \sigma, \tau)$ is the \textit{Witt frame} for $R$.
	
	The Witt frame extends the Witt $1$-frame $\mathcal{W}(R) = (W(R) \supset I(R), f, v^{-1})$. Windows over the Witt 1-frame are equivalent to $3n$-displays in the sense of \cite{Zink2002}, which we will hereafter refer to as Zink displays. Then by Lemma \ref{lem-windows}, 1-displays over $\underline{W}(R)$ are equivalent to Zink displays.
\end{ex}

Let $B \to A$ be a $p$-adic PD-thickening with kernel $J$. Using the divided powers on $J$, Zink defines an isomorphism of $W(B)$-modules
\begin{align*}
\log: W(J) \xrightarrow{\sim} \prod_{i \in \mathbb{N}} J,
\end{align*}
see \cite[\textsection 1.4]{Zink2002} for details. Denote the image of $\xi \in W(J)$ by $\log(\xi) = [\xi_0, \xi_1, \dots]$. 

\begin{ex}[The relative Witt frame]\label{ex-relwittframe}
	Let $B \to A$ be a $p$-adic PD-thickening. Define a frame $\underline{W}(B/A)$ associated with $B/A$ as follows. For $S_{\ge 0}$ take the $\zz_{\ge 0}$-graded ring with $S_0 = W(B)$, $S_n = I(B) \oplus J$ with $J$ viewed as a $W(B)$-module by restriction of scalars, and multiplication $S_n \times S_m \to S_{n+m}$ for $n, m \ge 1$ defined by $(v(a),x)\cdot (v(b),y) =  (v(ab), xy)$ for $a, b \in W(B)$, $x,y \in J$. The map $t_0: S_1 = I(B) \oplus J \to W(B) = S_0$ is given by $(v(a), x) \mapsto v(a) + \log^{-1}[x,0,0,\dots]$, and $t_n$ for $n \ge 1$ is given by multiplication by $p$ on the first factor and the identity on the second factor. Finally, let $\sigma_0 = f_B$ and for $n \ge 1$ define $\sigma_n(v(a),x) = a$. Denote the resulting $\zz$-graded ring by $W(B/A)^\oplus$. The corresponding frame $\underline{W}(B/A) = (W(B/A)^\oplus, \sigma, \tau)$ is the \textit{relative Witt frame} for $B/A$.
	
	Let $I(B/A)$ denote the kernel of $W(B) \to A$, and denote by $\tilde{v}^{-1}$ the unique extension of $v^{-1}$ to $I(B/A)$ whose restriction to $W(J) = \ker(W(B) \to W(A))$ is given by $[\xi_0, \xi_1, \dots] \mapsto [\xi_1, \xi_2, \dots]$ in logarithmic coordinates. Then $\mathcal{W}(B/A) = (W(B) \supset I(B/A), f, \tilde{v}^{-1})$ is a $1$-frame (see \cite[\textsection 2.2]{Lau2014}), and $\underline{W}(B/A)$ extends $\mathcal{W}(B/A)$.
\end{ex}
\subsection{Recollections on crystals}\label{sub-reviewcrystals}
We review the definitions of crystals and the crystalline site as in \cite{BBM1982}, and we sketch proofs of some standard lemmas which will be useful in \textsection \ref{sub-hodge} when we are checking Frobenius equivariance of certain morphisms of crystals.

For a $W(k_0)$-scheme $X$ in which $p$ is locally nilpotent, denote by $\textup{CRIS}(X/W(k_0))$ the big fppf crystalline site as in \cite{BBM1982}. This is the site whose underlying category is the category of triples $(U,T,\delta)$ where $U \hookrightarrow T$ is a closed immersion of an $X$-scheme $U$ into a $p$-nilpotent $W(k_0)$-scheme $T$ such that the ideal $\mathscr{I}$ of $\mathcal{O}_T$ defining the embedding is equipped with divided powers compatible with the natural divided powers on $pW(k_0)$. If $X = \Spec R$ is affine, we will write $\textup{CRIS}(R/W(k_0))$ to mean $\textup{CRIS}(\Spec R/W(k_0))$. Recall that to give a sheaf on $\textup{CRIS}(X/W(k_0))$ is equivalent to giving, for every triple $(U,T,\delta)$ in $\textup{CRIS}(X/W(k_0))$, an fppf sheaf $\mathscr{F}_T$ on $T$, and for every morphism $(u,v): (U',T,',\delta') \to (U,T,\delta)$, a morphism of sheaves $v^{-1}\mathscr{F}_T \to \mathscr{F}_{T'}$ which satisfies a cocycle condition (see \cite[1.1.3]{BBM1982}). The crystalline structure sheaf for $X$ over $W(k_0)$, denoted by $\mathcal{O}_{X/W(k_0)}$, is defined by the rule $\Gamma((U,T,\delta),\mathcal{O}_{X/W(k_0)}) = \Gamma(T,\mathcal{O}_T).$ If $\mathscr{F}$ is a sheaf of $\mathcal{O}_{X/W(k_0)}$-modules, then for a morphism $(u,v)$ as above the transition morphism $v^{-1}\mathscr{F}_T\to \mathscr{F}_{T'}$ induces a morphism $v^\ast\mathscr{F}_T \to \mathscr{F}_{T'}$. 

\begin{Def}
	A \textit{crystal of finite locally free (resp. locally free, resp. quasi-coherent) $\mathcal{O}_{X/W(k_0)}$-modules} is an $\mathcal{O}_{X/W(k_0)}$-module $\mathscr{F}$ such that for every $(U,T,\delta)$ in $\textup{CRIS}(X/W(k_0))$ the $\mathcal{O}_T$-module $\mathscr{F}_T$ is finite locally free (resp. locally free, resp. quasi-coherent) and for every morphism  $(u,v):(U',T',\delta') \to (U,T,\delta)$, the transition morphism $v^\ast\mathscr{F}_T \to \mathscr{F}_{T'}$ is an isomorphism. 
\end{Def}

We will denote by $\textup{LFCrys}(X/W(k_0))$ the category of crystals of locally free $\mathcal{O}_{X/W(k_0)}$-modules. The full subcategory of crystals of finite locally free $\mathcal{O}_{X/W(k_0)}$-,modules is a rigid exact tensor category which is a full tensor subcategory of the category of crystals in quasi-coherent $\mathcal{O}_{X/W(k_0)}$-modules. The unit object is the crystal $\mathbbm{1}$ which assigns to any $(U,T,\delta)$ the finite locally free $\mathcal{O}_T$-module $\mathcal{O}_T$. If $X = \Spec R$ is affine, we will write $\textup{LFCrys}(R/W(k_0))$ instead of $\textup{LFCrys}(\Spec R/W(k_0))$.

\begin{rmk}\label{rmk-crystals}
	We will often write just $B \to A$ to denote the PD-thickening $(\Spec A, \Spec B, \delta)$. Because fppf sheaves on a scheme $T$ are uniquely determined by their evaluations on affine $T$-schemes, to give a crystal in quasi-coherent $\mathcal{O}_{X/W(k_0)}$-modules, it is enough to give, for every PD-thickening $B\to A$ of $p$-nilpotent $W(k_0)$-algebras over $X$, a $B$-module $M_{B/A}$, and for every morphism $(B'\to A') \to (B\to A)$ of PD-thickenings, an isomorphism
	\begin{align}\label{eq-crystalproperty}
	M_{B/A} \otimes_B B'\xrightarrow{\sim} M_{B'/A'}.
	\end{align}
	These isomorphisms should satisfy the obvious cocycle condition for compositions. The associated crystal is (finite) locally free if each $B$-module $M_{B/A}$ is (finite) projective. 
\end{rmk}

If $(U,T,\delta)$ is an object in $\textup{CRIS}(X/W(k_0))$, then we can view $(U,T,\delta)$ as an object in $\textup{CRIS}(Y/W(k_0))$, denoted $\psi_!(U,T,\delta)$, by viewing $U$ as a $Y$-scheme via $U \to X \to Y$. If $\mathscr{F}$ is a sheaf on $\textup{CRIS}(Y/W(k_0))$, define $\psi^\ast \mathscr{F}$ by 
\begin{align}\label{eq-basechangecrystals}
\psi^\ast\mathscr{F}(U,T,\delta) := \mathscr{F}(\psi_!(U,T,\delta)).
\end{align}
This determines a pullback functor $\textup{Sh}(\textup{CRIS}(Y/W(k_0))) \to \textup{Sh}(\textup{CRIS}(X/W(k_0))$, which preserves the respective categories of crystals.

\begin{Def}
	The category of \textit{isocrystals over $X$}, denoted $\textup{Isoc}(X)$, is the category whose objects are crystals $\mathbb{D}$ in locally free $\mathcal{O}_{X/W(k_0)}$-modules, and whose morphisms are global sections of the Zariski sheaf $\underline{\textup{Hom}}(\mathbb{D},\mathbb{D}')[1/p]$. We will write $\mathbb{D}[1/p]$ for the object $\mathbb{D}$ viewed as an object in $\textup{Isoc}(X)$. When $X = \Spec R$ is affine, we write $\textup{Isoc}(\Spec R) = \textup{Isoc}(R)$.
\end{Def}

\begin{rmk}
	If $X$ is quasi-compact, then $\underline{\textup{Hom}}(\mathbb{D},\mathbb{D}')[1/p]$ can be identified with $\textup{Hom}(\mathbb{D},\mathbb{D}')[1/p]$, i.e. a morphism $\mathbb{D}[1/p] \to \mathbb{D}'[1/p]$ in $\textup{Isoc}(X)$ is an equivalence class of diagrams
	\begin{align*}
	\mathbb{D} \xleftarrow{p^n} \mathbb{D} \xrightarrow{s} \mathbb{D}',
	\end{align*}
	where $s$ is a morphism of crystals of $\mathcal{O}_{X/W(k_0)}$-modules.
\end{rmk}

If $B \to A$ is a $p$-adic PD-thickening, then $B \to A$ can be written as the projective limit of divided power extensions $B_n = B/p^nB \to A/p^nA = A_n$. If $\mathbb{D}$ is a crystal in $\mathcal{O}_{\Spec R/W(k_0)}$-modules and $B \to A$ is a $p$-adic PD-thickening over $R$, we write $\mathbb{D}_{B/A} := \varprojlim \mathbb{D}_{B_n/A_n}$. This defines an evaluation functor 
\begin{align}\label{eq-evalBA}
(-)_{B/A}: \textup{LFCrys}(R/W(k_0)) \to \textup{Mod}(B), \ \mathbb{D}\mapsto \mathbb{D}_{B/A}.
\end{align}
The functor $(-)_{B/A}$ extends naturally to a functor
\begin{align}\label{eq-evalBAiso}
(-)_{B/A}[1/p]:\textup{Isoc}(R) \to \textup{Mod}(B[1/p]), \ \mathbb{D}[1/p] \mapsto (\mathbb{D}_{B/A})[1/p]
\end{align}
which on morphisms is given by the composition
\begin{align*}
\textup{Hom}(\mathbb{D}_1, \mathbb{D}_2) \to \textup{Hom}_B((\mathbb{D}_1)_{B/A},(\mathbb{D}_2)_{B/A})[1/p] \to \textup{Hom}_{B[1/p]}((\mathbb{D}_1)_{B/A}[1/p],(\mathbb{D}_2)_{B/A}[1/p]).
\end{align*}
\begin{rmk} \label{rmk-lastarrow}
	If $\mathbb{D}_1$ is a crystal of finite locally free $\mathcal{O}_{\Spec R/ W(k_0)}$-modules, then $(\mathbb{D}_1)_{B/A}$ is a finite projective $B$-module, and the last arrow is an isomorphism.
\end{rmk}

\begin{rmk}\label{rmk-perfectequiv}
	If $B \to A$ is a PD-thickening over $R$, then $W(B)$ is $p$-adic by \cite[Prop. 3]{Zink2002}, and $W(B) \to A$ is a $p$-adic PD-thickening, see e.g. \cite[\textsection 1G]{Lau2014}. When $pR=0$ and $R$ is perfect, the evaluation functors $(-)_{W(R)/R}$ and $(-)_{W(R)/R}[1/p]$ are equivalences (see e.g., \cite[Prop. 4.5]{Grothendieck1974}).
\end{rmk}

The following lemmas are no doubt well-known to experts, but we could not find a reference, so we sketch proofs for the sake of completeness. Suppose $R$ is an $\mathbb{F}_p$-algebra, and choose a polynomial algebra $W(k_0)[x_\alpha]_{\alpha \in \mathcal{A}}$ surjecting onto $R$. Let $\gamma$ denote the canonical divided powers on $pW(k_0)$, and denote by $D$ the PD-envelope of $W(k_0)[x_\alpha]$ with respect to $K = \ker(W(k_0)[x_\alpha] \to R)$ relative to $(W(k_0), pW(k_0), \gamma)$. Then the kernel $\bar{J}$ of $D \to R$ is equipped with divided powers compatible with those on $pW(k_0)$, and $D_n:=D/p^nD \to R$ is a PD-thickening over $R$ for every $n$. If we denote by $D^\wedge$ the $p$-adic completion of $D$, then $D^\wedge \to R$ defines a $p$-adic PD-thickening, and we can define functors $(-)_{D^\wedge / R}$ and $(-)_{D^\wedge / R} [1/p]$ as in (\ref{eq-evalBA}) and (\ref{eq-evalBAiso}).

\begin{lemma} \label{lem-faithful}
	The functor $(-)_{D^\wedge / R}$ is faithful. Moreover, if $\mathbb{D}_1$ is a crystal of finite locally free $\mathcal{O}_{\Spec R/W(k_0)}$-modules, then the map 
	\begin{align*}
	\textup{Hom}(\mathbb{D}_1,\mathbb{D}_2)[1/p] \to \textup{Hom}_{D^\wedge[1/p]}((\mathbb{D}_1)_{D^\wedge/R}[1/p], (\mathbb{D}_2)_{D^\wedge/R}[1/p])
	\end{align*}
	induced by $(-)_{D^\wedge / R}[1/p]$ is injective.
\end{lemma}
\begin{proof}
	The first statement follows from the fact that for any PD-thickening $B \to A$ we can find a lift $W(k_0)[x_\alpha] \to B$ of $R \to A$, so by the universal properties of $D$ and $D^\wedge$ we obtain a PD-morphism $(D^\wedge \to R) \to (B \to A)$. The second statement follows from the first using Remark \ref{rmk-lastarrow} and exactness of localization.	
\end{proof}

As $R$ varies in $\textup{Nilp}_{W(k_0)}$ we obtain fibered categories $\textup{LFCrys}$ and $\textup{Isoc}$ whose fibers over $R$ in $\textup{Nilp}_{W(k_0)}$ are the categories $\textup{LFCrys}(R)$ and $\textup{Isoc}(R)$, respectively.

\begin{lemma}\label{lem-etalelocal}
	The fibered categories $\textup{LFCrys}$ and $\textup{Isoc}$ form stacks for the \'etale topology on $\textup{Nilp}_{W(k_0)}$.
\end{lemma}
\begin{proof}
	It is enough to show the result for $\textup{LFCrys}$, where the key point is that if $B \to A$ is a PD-thickening over $R$, and $R \to R'$ is \'etale and faithfully flat, then the homomorphism $A \to A' := A\otimes_R R'$ is also \'etale and faithfully flat, so there exists a unique \'etale faithfully flat lift $B \to B'$ with $B' \to A'$ a PD-thickening over $R'$. The result follows from \'etale descent for modules over rings along with the crystal property (\ref{eq-crystalproperty}).
\end{proof}

We will eventually want to consider $p$-nilpotent $W(k_0)$-algebras which have a $p$-basis \'etale locally. For the convenience of the reader, we recall the definition of a $p$-basis.

\begin{Def}\label{def-pbasis}
	Let $R$ be an $\mathbb{F}_p$-algebra. A \textit{$p$-basis} for $R$ is a subset $\{x_\alpha\}$ of $R$ such that the set of monomials $x^J$ for $J$ running over the multi-indices $J = (i_\alpha), 0 \le i_\alpha < p$, provides a basis for $R$ viewed as an $R$-module over itself via the Frobenius.
\end{Def}

For example, any field of characteristic $p$ or any regular local ring which is essentially of finite type over a field of characteristic $p$ has a $p$-basis (see \cite[Ex. 1.1.2]{BM1990}). We say that an $\mathbb{F}_p$-algebra $R$ has a $p$-basis \'etale locally if there is some faithfully flat \'etale ring homomorphism $R \to R'$ where $R'$ has a $p$-basis. One reason for the usefulness of the existences of a $p$-basis is the following lemma. Recall that the \textit{perfect closure} of an $\mathbb{F}_p$-algebra is the colimit of infinitely many copies of $R$ along the Frobenius morphism $x \mapsto x^p$. 

\begin{lemma}\label{lem-pbasisuseful}
	Let $R$ be an $k_0$-algebra which admits a $p$-basis, and let $R^\textup{perf}$ be the perfect closure of $R$. Then $R \to R^\textup{perf}$ is faithfully flat, and the base change functor $\textup{LFCrys}(R/W(k_0)) \to \textup{LFCrys}(R^\textup{perf}/W(k_0))$ $($see $($\ref{eq-basechangecrystals}$))$ is faithful.
\end{lemma}
\begin{proof}
	If $R$ has a $p$-basis then the Frobenius $\phi: R \to R$ is faithfully flat, since $R$ is free viewed as a module over itself via $\phi$. Thus $R^\textup{perf}$ is faithfully flat, since it is a colimit of faithfully flat $R$-algebras. The second part follows from \cite[Lem. 7.5]{Lau2018b}; we give the argument for completeness. If $(x_i)_{i\in I}$ is a $p$-basis for $R$, then $R^\textup{perf} = R[(x_i^{1/p^\infty})_{i \in I}]$, and for a PD-thickening $B \to A$ over $R$, we have $A\otimes_R R^{\textup{perf}} = A[(a_i^{1/p^\infty})_{i\in I}]$, where $a_i$ is the image of $x_i$ in $A$. If $b_i \in B$ is a lift of $a_i$, then the divided powers extend to $B[(b_i^{1/p^\infty})_{i\in I}] \to A[(a_i^{1/p^\infty})_{i\in I}]$ by flatness. Thus the result follows from faithfully flat descent for modules over rings along with the crystal property (\ref{eq-crystalproperty}).
\end{proof}

The following lemma will be useful in the proof of Theorem \ref{thm-1}.

\begin{lemma}\label{lem-tensorsagree}
	Suppose $pR=0$ and that $R$ has a $p$-basis \'etale locally. Let $\mathbb{D}$ and $\mathbb{D}'$ be crystals in locally free $\mathcal{O}_{\Spec R/ W(k_0)}$-modules, and let 
	\begin{align*}
		t_1, t_2: \mathbb{D} \to \mathbb{D}'
	\end{align*}
	be two morphisms of crystals. Then $t_1 = t_2$ if and only if their evaluations on $W(R) \to R$ agree.
\end{lemma}
\begin{proof}
	One direction holds by definition, so we only need to prove that $t_1 = t_2$ if their evaluations on $W(R) \to R$ agree. The property of agreeing on $W(R) \to R$ is stable under base change so, by Lemma \ref{lem-etalelocal}, it is enough to assume that $R$ has a $p$-basis. In turn we can use Lemma \ref{lem-pbasisuseful} to reduce to the case where $R$ is perfect. There the result follows because evaluation on $W(R) \to R$ is faithful for perfect rings, see Remark \ref{rmk-perfectequiv}.
\end{proof}

Suppose $R$ is a $p$-nilpotent $W(k_0)$-algebra, and let $R_0 = R/pR$. Then the closed embedding $i: \Spec R_0 \hookrightarrow \Spec R$ induces a morphism of topoi $\icris = (\icris_\ast, i_{\text{CRIS}}^\ast)$ between sheaves on $\textup{CRIS}(R_0/W(k_0))$ and sheaves on $\textup{CRIS}(R /W(k_0))$. By \cite[IV, Thm. 1.4.1]{Berthelot1974}, the functors $\icris_\ast$ and $i_{\text{CRIS}}^\ast$ are quasi-inverse to one another, and induce an equivalence of categories
\begin{align}\label{eq-crystalequiv}
\textup{LFCrys}(R_0/W(k_0)) \xrightarrow{\sim} \textup{LFCrys}(R/W(k_0)).
\end{align}
This equivalence extends to an equivalence $\textup{Isoc}(R_0) \xrightarrow{\sim} \textup{Isoc}(R)$. 

Let $R_0$ be an $\mathbb{F}_p$-algebra, and let $\phi_0$ denote the $p$-power Frobenius $r \mapsto r^p$ of $R_0$. If $B \to A$ is a PD-thickening over $R_0$, we write $\phi_!(B/A)$ for the PD-thickening $B \to A$ where $A$ is viewed as an $R_0$-algebra via restriction of scalars along $\phi_0$. For any crystal $\mathbb{D}$ in $\mathcal{O}_{\Spec R_0/W(k_0)}$-modules, we define the value of the Frobenius pullback $\phi_0^\ast \mathbb{D}$ on a $p$-adic PD-thickening $B \to A$ over $R_0$ by 
\begin{align}\label{eq-shriek}
(\phi_0^\ast \mathbb{D})_{B/A} := \mathbb{D}_{{\phi_0}_!(B/A)}.
\end{align}
If $\sigma: B \to B$ is a lift of the Frobenius of $A$ which preserves the divided powers, then $\sigma$ induces a PD-morphism $(B \to A) \to {\phi_0}_!(B \to A)$, so by the crystal property we obtain
\begin{align}\label{eq-frobcrystal}
(\phi_0^\ast \mathbb{D})_{B/A} \xrightarrow{\sim} \sigma^\ast( \mathbb{D}_{B/A}).
\end{align}

More generally, if $R$ is a $p$-nilpotent $W(k_0)$-algebra and $\mathbb{D}$ is a crystal in locally free $\mathcal{O}_{\Spec R/W(k_0)}$-modules, then we can use the equivalence (\ref{eq-crystalequiv}) to define the Frobenius pullback $\phi^\ast\mathbb{D}$ of $\mathbb{D}$. Explicitly,
\begin{align*}
\phi^\ast \mathbb{D} := \icris_\ast(\phi_0^\ast i_{\text{CRIS}}^\ast \mathbb{D}).
\end{align*}
\subsection{The crystals associated with $p$-divisible groups and displays}\label{sub-zinkcrystals}
We recall the crystals associated with $p$-divisible groups and to nilpotent Zink displays, and we discuss the connection between the two. Our main reference for the crystals associated with $p$-divisible groups is \cite{BBM1982}. For more information on the crystals associated with nilpotent Zink displays, we refer the reader to \cite[\textsection 2.2]{Zink2002} and \cite[\textsection 2.4]{Lau2018}.

If $X$ is a $p$-divisible group over a $p$-nilpotent $W(k_0)$-algebra $R$, denote by $\mathbb{D}(X)$ the \textit{covariant} Dieudonn\'e crystal of $X$ as in \cite{BBM1982}. In fact, the crystal associated with $X$ as defined in \textit{loc. cit.} is contravariant, so to obtain a covariant crystal we define $\mathbb{D}(X)$ to be the contravariant Dieudonn\'e crystal associated with $X^D$. Equivalently, by the crystalline duality theorem \cite[\textsection 5.3]{BBM1982}, $\mathbb{D}(X)$ is the dual of the contravariant Dieudonn\'e crystal associated with $X$. 

The Dieudonn\'e crystal $\mathbb{D}(X)$ is a crystal of finite locally free $\mathcal{O}_{\Spec R/W(k_0)}$-modules, and the sections of $\mathbb{D}(X)$ over the trivial PD-thickening $\id_R: R \to R$ are equipped with a filtration by finite projective $R$-modules
\begin{align}\label{eq-hodgeX}
\text{Fil}^0(\mathbb{D}(X)) = \mathbb{D}(X)_{R/R}\supset \text{Fil}^1(\mathbb{D}(X)) = \text{Lie}(X^D)^\vee \supset \textup{Fil}^2(\mathbb{D}(X)) = 0,
\end{align}
called the Hodge filtration of $X$, which makes the following sequence exact
\begin{align}\label{eq-hodgeseqX}
0 \to \textup{Fil}^1(\mathbb{D}(X)) \to \mathbb{D}(X)_{R/R} \to \textup{Lie}(X) \to 0.
\end{align}
\begin{Def}
	A \textit{Dieudonn\'e crystal} over $R$ is a triple $(\mathbb{D},\mathbb{F},\mathbb{V})$, where $\mathbb{D}$ is a crystal of finite locally free $\mathcal{O}_{\Spec R/W(k_0)}$-modules, and
	\begin{align*}
	\mathbb{F}: \phi^\ast \mathbb{D} \to \mathbb{D} \text{ and } \mathbb{V}: \mathbb{D} \to \phi^\ast \mathbb{D}
	\end{align*}
	are morphisms of crystals such that $\mathbb{F} \circ \mathbb{V} = p\cdot \id_{\mathbb{D}}$ and $\mathbb{V} \circ \mathbb{F} = p \cdot \id_{\phi^\ast \mathbb{D}}$. 
\end{Def}
If $X$ is a $p$-divisible group over an $\mathbb{F}_p$-algebra $R_0$, denote by $X^{(p)}$ the $p$-divisible group $X \otimes_{R_0, \phi_0} R_0$ obtained by base change along $\phi_0$. We obtain a Dieudonn\'e crystal structure on $\mathbb{D}(X)$ by taking $\mathbb{F}$ and $\mathbb{V}$ to be induced from the Verschiebung and Frobenius
\begin{align*}
V_{X}: X^{(p)} \to X, \ F_X: X \to X^{(p)},
\end{align*}
respectively. Let us emphasize that since we are using the covariant Dieudonn\'e crystal, $X \mapsto \mathbb{D}(X)$ sends the Frobenius of $X$ to the Verschiebung of $\mathbb{D}(X)$ and the Verschiebung of $X$ to the Frobenius of $\mathbb{D}(X)$. More generally, if $R$ is a $p$-nilpotent $W(k_0)$-algebra, then we obtain $\mathbb{F}$ and $\mathbb{V}$ on $\mathbb{D}(X)$ by taking the unique maps lifting the Frobenius and Verschiebung for $i_{\textup{CRIS}}^\ast \mathbb{D}(X)$ along the equivalence (\ref{eq-crystalequiv}). 

The unit object $\mathbbm{1}$ in the rigid tensor category of finite locally free crystals in $\mathcal{O}_{\Spec R/ W(k_0)}$-modules is given by the crystal $\mathbb{D}(\mu_{p^\infty})$ associated with the multiplicative $p$-divisible group $\mu_{p^\infty}$ over $R$ (here we use $\mathbb{D}(\mu_{p^\infty})$ because we normalize Dieudonn\'e theory covariantly; in this way we have the same unit object as in \cite{HP2017} and \cite{Kim2018}). It follows that $\mathbbm{1}$ is endowed with the structure of a Dieudonn\'e crystal. Explicitly, we have a canonical isomorphism $\phi^\ast\mathbbm{1} \cong \mathbbm{1}$, and with respect to this isomorphism we take $\mathbb{F} = \id_{\mathbbm{1}}$ and $\mathbb{V} = p \cdot \id_{\mathbbm{1}}$. We will also endow the sections of $\mathbbm{1}$ over $\id_R: R \to R$ with the filtration
\begin{align}\label{eq-hodge1}
\textup{Fil}^0(\mathbbm{1}) = R \supset \textup{Fil}^1(\mathbbm{1}) = 0.
\end{align}
We will refer to this as the Hodge filtration for $\mathbbm{1}$.

Let $R$ be a $p$-nilpotent $\zz_p$-algebra, and denote by $\textup{Zink}(R)$ the category of Zink displays over $R$, which is equivalent to the category of windows over $\mathcal{W}(R)$ and to the category of 1-displays over $\underline{W}(R)$ by Lemma \ref{lem-windows}. Denote by $\textup{nZink}(R)$ the full subcategory of nilpotent Zink displays (see \cite[Def. 11]{Zink2002}). If a Zink display $\underline{P}$ over $R$ is nilpotent, then we can associate to $\underline{P}$ a formal $p$-divisible group $\textup{BT}_R(\underline{P})$. By the main theorems of \cite{Zink2002} and \cite{Lau2008}, $\textup{BT}_R$ defines an equivalence of categories between nilpotent Zink displays and formal $p$-divisible groups over $R$. When the ring $R$ is clear from context, we will sometimes omit the subscript from $\textup{BT}_R$.

An explicit quasi-inverse functor $\Phi_R$ for $\textup{BT}_R$ is defined in \cite[Prop. 2.1]{Lau2013}. Let us briefly review its definition. As a first step one defines a functor from $p$-divisible groups over $R$ to the category of filtered $F$-$V$-modules over $R$. Here a \textit{filtered $F$-$V$-module} over $R$ is a quadruple $(P, \Fil P, F^\sharp, V^\sharp)$, where $P$ is a finite projective $W(R)$-module with a filtration $I(R)P \subseteq \Fil P \subseteq P$ such that $P / \Fil P$ is projective over $R$, and where $F^\sharp: f^\ast P \to P \text{ and } V^\sharp: P \to f^\ast P $ are $W(R)$-module homomorphisms such that $F^\sharp \circ V^\sharp = p \cdot\id_{P}$ and $V^\sharp \circ F^\sharp = p \cdot \id_{f^\ast P}$.

Let $R_0 = R/pR$, so the kernel of $W(R) \to R_0$ is naturally equipped with divided powers, making $W(R) \to R_0$ into a $p$-adic PD-thickening over $R_0$ (see Remark \ref{rmk-perfectequiv}). If $X$ is a $p$-divisible group over $R$, and $X_0 = X \otimes_R R_0$, then since the Frobenius for $W(R)$ is compatible with the PD-structure on the kernel of $W(R) \to R_0$ (see \cite[\textsection 1G]{Lau2014}), by (\ref{eq-frobcrystal}) we have
\begin{align*}
(\phi_0^\ast \mathbb{D}(X_0))_{W(R)/R_0} \cong f^\ast(\mathbb{D}(X_0)_{W(R)/R_0}).
\end{align*} 
Hence if we take $P = \mathbb{D}(X_0)_{W(R)/R_0}$, the evaluation of $\mathbb{F}$ and $\mathbb{V}$ for $\mathbb{D}(X_0)$ on $W(R) \to R_0$ induce homomorphisms $F^\sharp$ and $V^\sharp$ as in the above definition. Moreover, the natural identification $\mathbb{D}(X)_{W(R)/R} \cong \mathbb{D}(X_0)_{W(R)/R_0}$ provides us with a map $P \to \textup{Lie}(X)$ via the composition
\begin{align*}
	P \xrightarrow{\sim} \mathbb{D}(X)_{W(R)/R} \twoheadrightarrow \mathbb{D}(X)_{R/R} \twoheadrightarrow \textup{Lie}(X).
\end{align*}
It follows that we can define a filtered $F$-$V$-module associated with $X$ by $(P, \Fil P, F^\sharp, V^\sharp)$, with $\Fil P = \ker(P \to \textup{Lie}(X))$. As in \cite{Lau2013} we write $\Theta_R$ for the functor which assigns a filtered $F$-$V$-module to a $p$-divisible group. 

We also have a faithful functor $\Upsilon_R$ from Zink displays over $R$ to filtered $F$-$V$-modules over $R$, defined by assigning to the Zink display $\underline{P} = (P, \Fil P, F_0, F_1)$ the filtered $F$-$V$-module $(P, \Fil P, F_0^\sharp, V^\sharp)$, where $F_0^\sharp: f^\ast P \to P$ is the linearization of $F_0$ and $V^\sharp$ is the homomorphism $P \to f^\ast P$ associated with $\underline{P}$ by \cite[Lem. 10]{Zink2002} (see \textsection \ref{sub-frames}). By \cite[Prop. 2.1]{Lau2013}, there is a unique functor 
\begin{align}\label{eq-laufunctor}
\Phi_R : \textup{pdiv}(R) \to \textup{Zink}(R)
\end{align}
which is compatible with base change and for which there is a natural isomorphism of functors $\Theta_R \cong \Upsilon_R \circ \Phi_R$. The restriction of $\Phi_R$ to formal $p$-divisible groups is an equivalence by \cite[Thm. 5.1]{Lau2013}, and $\Phi_R$ provides a quasi-inverse to $\textup{BT}_R$ by \cite[Lem. 8.1]{Lau2013}. 

If $\underline{P}$ is a Zink display over an $\mathbb{F}_p$-algebra $R_0$, define $\underline{P}^{(p)} = (P^{(p)}, \Fil P^{(p)}, F_0^{(p)}, F_1^{(p)})$ to be the base change of $\underline{P}$ along the $p$-power Frobenius $\phi_0: R_0 \to R_0$. By definition of base change for displays, we have $P^{(p)} = f^\ast P$. By \cite[Ex. 23]{Zink2002}, $F_0^\sharp$ and $V^\sharp$ induce functorial morphisms of Zink displays 
\begin{align}\label{eq-frver}
\textup{Ver}_{\underline{P}}: \underline{P}^{(p)} \to \underline{P} \text{ and } \textup{Fr}_{\underline{P}}: \underline{P} \to \underline{P}^{(p)}, 
\end{align}
respectively. If $X$ is a $p$-divisible group over $R_0$, then one sees from the definition of $\Theta_{R_0}$ and the faithfulness of $\Upsilon_{R_0}$ that
\begin{align}\label{eq-fv}
\Phi_{R_0}(V_{X}) = \textup{Ver}_{\Phi_{R_0}(X)} \text{ and }\Phi_{R_0}(F_{X}) = \textup{Fr}_{\Phi_{R_0}(X)}.
\end{align}

Let us now recall the definition of the crystal associated with a nilpotent Zink display. Let $B \to A$ be a $p$-adic PD-thickening over $R$. Then the natural morphism of 1-frames $\mathcal{W}(B/A) \to \mathcal{W}(A)$ induces an equivalence of categories between nilpotent windows over $\mathcal{W}(B/A)$ and nilpotent Zink displays over $A$ by \cite[Thm. 44]{Zink2002} (see also \cite[Prop. 10.4]{Lau2010}, and for the definition of nilpotence in this generality see \cite[\textsection 10.3]{Lau2010}). It follows that if $\underline{P}$ is a nilpotent Zink display over $R$, and  $\underline{P}_A$ is the base change of $\underline{P}$ to $\mathcal{W}(A)$, then for any $p$-adic PD-thickening $B \to A$ over $R$ there is a unique (up to unique isomorphism which lifts the identity) lift of $\underline{P}_A$ to $\mathcal{W}(B/A)$. Denote this lift by $\underline{\tilde{P}} = (\tilde{P}, \Fil \tilde{P}, \tilde{F}_0, \tilde{F}_1)$. The evaluation of the Dieudonn\'e crystal $\mathbb{D}(\underline{P})$ associated with $\underline{P}$ on $B\to A$ is
\begin{align*}
\mathbb{D}(\underline{P})_{B/A} := \tilde{P} / I(B) \tilde{P}.
\end{align*}
In particular, if $\underline{P} = (P, \Fil P, F_0, F_1)$, then $\mathbb{D}(\underline{P})_{R/R} = P / I(R) P$. We refer to the filtration
\begin{align*}
\textup{Fil}^0(\mathbb{D}(\underline{P})) = P / I(R) P \supset \text{Fil}^1(\mathbb{D}(\underline{P})) = \Fil P / I(R) P \supset \textup{Fil}^2(\mathbb{D}(\underline{P})) = 0
\end{align*}
as the Hodge filtration of $\mathbb{D}(\underline{P})$ (or of $\underline{P})$, and we observe that the following sequence is exact
\begin{align*}
0 \to \textup{Fil}^1(\mathbb{D}(\underline{P})) \to \mathbb{D}(\underline{P})_{R/R} \to P / \Fil P \to 0.
\end{align*}
We will sometimes denote $P / \Fil P$ by $\text{Lie}(\underline{P})$. If $\underline{P} = \Phi_R(X)$ for a formal $p$-divisible group $X$ over $R$, then by definition of $\Phi_R$ we have $\textup{Lie}(\underline{P}) = P / \Fil P \cong \textup{Lie}(X)$. If $\underline{P}$ is the Zink display associated with a higher display $\underline{M}$ by Lemma \ref{lem-windows}, then 
\begin{align}\label{eq-samehodge}
\textup{Fil}^i(\underline{M}) = \textup{Fil}^i(\mathbb{D}(\underline{P})),
\end{align}
for $0 \le i \le 2$, where $\textup{Fil}^i(\underline{M})$ is the Hodge filtration of $\underline{M}$ (see (\ref{eq-hodgedisp})).

The assignment $\underline{P} \mapsto \mathbb{D}(\underline{P})$ is functorial in $\underline{P}$, so if $\underline{P}$ is a Zink display over an $\mathbb{F}_p$-algebra $R_0$, then the maps (\ref{eq-frver}) induce morphisms of crystals
\begin{align}\label{eq-frobdef}
\mathbb{F}: \mathbb{D}(\underline{P}^{(p)}) \to \mathbb{D}(\underline{P}) \text{ and } \mathbb{V}: \mathbb{D}(\underline{P}) \to \mathbb{D}(\underline{P}^{(p)}).
\end{align}
Moreover, as a consequence of the definition of $\mathbb{D}(\underline{P})$ we obtain a canonical isomorphism
\begin{align}\label{eq-crystalbc}
\phi_0^\ast \mathbb{D}(\underline{P}) \cong \mathbb{D}(\underline{P}^{(p)}).
\end{align} 
Hence $\mathbb{D}(\underline{P})$ is canonically endowed with the structure of a Dieudonn\'e crystal. As in the case of $p$-divisible groups, we can use (\ref{eq-crystalequiv}) to lift this structure in the case of $p$-nilpotent $\zz_p$-algebras $R$.

\begin{lemma}\label{lem-zinkdieudonne}
	The functors $\underline{P} \mapsto \mathbb{D}(\underline{P})$ and $\underline{P} \mapsto \mathbb{D}(\textup{BT}_R(\underline{P}))$ from nilpotent Zink displays to crystals in finite locally free $\mathcal{O}_{\Spec R / \zz_p}$-modules are naturally isomorphic. Moreover, the isomorphism is compatible with the Frobenius and Verschiebung maps, and it preserves the Hodge filtration.
\end{lemma}
\begin{proof}
	The first statement proven in \cite[Thm. 94]{Zink2002} for the restriction of these crystals to the nilpotent crystalline site, and in \cite[Cor. 97]{Zink2002} for the restriction to PD-thickenings $B \to A$ which have nilpotent kernel. In general, it follows from the results of \cite{Lau2013}. Indeed, it is enough to show that the functors $X \mapsto \mathbb{D}(X)$ and $X \mapsto \mathbb{D}(\Phi_R(X))$ from infinitesimal $p$-divisible groups to $\textup{LFCrys}(\Spec R/\zz_p)$ are naturally isomorphic. For any given PD-thickening $B \to A$ over $R$, there is an isomorphism of $B$-modules
	\begin{align}\label{eq-crystalisom}
	\mathbb{D}(X)_{B/A} \cong \mathbb{D}(\Phi_R(X))_{B/A}
	\end{align}
	by \cite[Cor. 2.7]{Lau2013}. Explicitly, by the results of \cite{Lau2013}, if $\underline{\tilde{P}}$ is the unique lift of $\Phi_R(X)$ to $\mathcal{W}(B/A)$, then we can identify $\tilde{P} = \mathbb{D}(X)_{W(B)/A}$, so (\ref{eq-crystalisom}) is obtained from the crystal property applied to the morphism of PD-thickenings $(W(B) \to A) \to (B \to A)$. That (\ref{eq-crystalisom}) is compatible with the transition isomorphisms follows from the cocycle condition for $\mathbb{D}(X)$ and uniqueness of liftings along $\mathcal{W}(B/A) \to \mathcal{W}(A)$. Functoriality in $X$ follows from functoriality of $X \mapsto \mathbb{D}(X)$, and if $\Phi_R(X) = (P, \Fil P, F_0, F_1)$, then $\Fil P = \ker(\mathbb{D}(X)_{W(R)/R} \to \textup{Lie}(X))$, so it follows from (\ref{eq-hodgeseqX}) that (\ref{eq-crystalisom}) preserves the Hodge filtrations.
	
	Finally to prove compatibility with the Frobenius and Verschiebung one reduces to the case where $R$ is an $\mathbb{F}_p$-algebra, in which case we have $\mathbb{F}_{\mathbb{D}(X)} = \mathbb{D}(V_X)$ and $\mathbb{V}_{\mathbb{D}(X)} = \mathbb{D}(F_X)$. Then the result follows from functoriality of the isomorphism $\mathbb{D}(X) \cong \mathbb{D}(\Phi_R(X))$ along with (\ref{eq-fv}) and compatibility of $\Phi_R$ with base change.
\end{proof}
\section{$G$-displays} \label{section-Gdisplays}
Let $G = \Spec \mathcal{O}_G$ be a flat affine group scheme of finite type over $\zz_p$, and let $\mu: \mathbb{G}_{m,W(k_0)} \to G_{W(k_0)}$ be a cocharacter of $G_{W(k_0)}$. In section \textsection \ref{sub-gdisplau} we define the stack of $G$-displays of type $\mu$ over an \'etale sheaf of frames, following \cite{Lau2018}, and in \textsection \ref{sub-gdispdaniels} we develop Tannakian analogs of these objects. If $R$ is in $\textup{Nilp}_{\zz_p}$, and $\underline{S}$ is an \'etale sheaf of frames on $\Spec R$ which satisfies descent for displays (see Definition \ref{def-descent}), we prove (Theorem \ref{thm-equiv}) that our Tannakian framework is equivalent to Lau's torsor-theoretic framework. This is closely analogous to \cite[Thm. 3.16]{Daniels2019}, and throughout we provide references to \cite{Daniels2019} in lieu of proofs whenever the arguments mimic those in \textit{loc. cit.}.

In \textsection \ref{sub-hodgefilt} we define the Hodge filtration for Tannakian $(G,\mu)$-displays, compare it to the Hodge filtration for $G$-displays of type $\mu$, and explain (following \cite{Lau2018}) how lifts of the Hodge filtration relate to lifts of a Tannakian $(G,\mu)$-display. In \textsection \ref{sub-lifting}, under the additional assumptions that $G$ is reductive and $\mu$ is minuscule, we recall Lau's unique lifting lemma (Proposition \ref{prop-lifting}) for adjoint nilpotent $(G,\mu)$-displays. The unique lifting lemma is a crucial component of the construction of the crystal associated with an adjoint nilpotent $(G,\mu)$-display in \textsection \ref{sub-crystalGdisp}.
\subsection{$G$-displays of type $\mu$}\label{sub-gdisplau}
Recall \cite[\textsection 5]{Lau2018} a frame $\underline{S}$ is a frame over $W(k_0)$ if $S$ is a graded $W(k_0)$-algebra and $\sigma: S \to S_0$ extends the Frobenius of $W(k_0)$. In particular, if $R$ is in $\textup{Nilp}_{W(k_0)}$ and $B \to A$ is a PD-thickening over $R$, then the frames $\underline{W}(R)$ and $\underline{W}(B/A)$  are $W(k_0)$-frames. See \cite[Ex. 5.0.2]{Lau2018} for details.

If $X = \text{Spec }A$ is an affine $W(k_0)$-scheme, then an action of $\mathbb{G}_m$ on $X$ is equivalent to a $\zz$-grading on $A$ (see \cite[\textsection 5.1]{Lau2018} and \cite[\textsection 3.1]{Daniels2019} for details). If $\mathbb{G}_m$ acts on $X$, and $S$ is a $\zz$-graded $W(k_0)$-algebra, denote by $X(S)^0 \subseteq X(S)$ the set of $\mathbb{G}_m$-equivariant sections $\text{Spec }S \to X$ over $W(k_0)$. In other words, $X(S)^0$ is the set $\textup{Hom}_{W(k_0)}^0(A,S)$ of homomorphisms $A \to S$ of graded $W(k_0)$-algebras. 

Suppose $\underline{S}$ is an \'etale sheaf of frames on $\text{Spec }R$. If $R \to R'$ is \'etale, write 
\begin{align*}
\underline{S}(R') = (S(R'), \sigma(R'), \tau(R')),
\end{align*} 
so $S(R')$ is a $\zz$-graded ring, and $\sigma(R')$ and $\tau(R')$ are ring homomorphisms $S(R') \to S(R')_0$ as in Definition \ref{def-frame}.
To $X$ and $\underline{S}$ we associate two functors on \'etale $R$-algebras:
\begin{align*}
X(\underline{S})^0: R' \mapsto X(S(R'))^0, \text{ and }
X(\underline{S}_0): R' \mapsto X(S(R')_0).
\end{align*}
\begin{lemma}\label{lem-sheaves}
	Let $\underline{S}$ be an \'etale sheaf of frames on $\textup{Spec }R$, and let $X$ be an affine scheme of finite type over $W(k_0)$. Then the functors $X(\underline{S})^0$ and $X(\underline{S}_0)$ are \'etale sheaves on $\textup{Spec }R$.
\end{lemma}
\begin{proof}
	The proof is formally the same as that of \cite[Lem. 5.3.1]{Lau2018}.
\end{proof}

Let us recall the definition of the display group associated with $G$ and $\mu$ with values in a $\zz$-graded ring $S$. For details we refer the reader to \cite[\textsection 3.1]{Daniels2019} and \cite[\textsection 5.1]{Lau2018}. The cocharacter $\mu$ defines a right action of $\mathbb{G}_{m,W(k_0)}$ on $G_{W(k_0)}$ by
\begin{align*}
g \cdot \lambda := \mu(\lambda)^{-1} g \mu(\lambda) 
\end{align*}
for any $W(k_0)$-algebra $R$, $g \in G_{W(k_0)}(R)$ and $\lambda \in \mathbb{G}_{m,W(k_0)}(R)$. If $S$ is a $\zz$-graded ring, define  
\begin{align*}
G(S)_\mu := G(S)^0,
\end{align*}
i.e., $G(S)_\mu$ is the subset of $G_{W(k_0)}(S) = \text{Hom}_{W(k_0)}(\mathcal{O}_G,S)$ consisting of $W(k_0)$-algebra homomorphisms which preserve the respective gradings. Similarly, if $S$ is an \'etale sheaf of frames on $\text{Spec }R$, define
\begin{align*}
G(\underline{S})_\mu := G(\underline{S})^0,
\end{align*}	
so $G(\underline{S})_\mu$ is an \'etale sheaf of groups on $\Spec{R}$.

Suppose $\underline{S} = (S,\sigma,\tau)$ is a $W(k_0)$-frame. Then the $\zz_p$-algebra homomorphisms $\sigma, \tau: S \to S_0$ induce group homomorphisms
\begin{align*}
\sigma, \tau: G(S)_\mu \to G(S_0)
\end{align*}
as follows: if $g \in G(S)_\mu$, then $\sigma(g)$ (resp. $\tau(g)$) is defined by post-composing $g \in \text{Hom}_{W(k_0)}(\mathcal{O}_G,S)$ with $\sigma: S \to S_0$ (resp. $\tau:S \to S_0$). Using $\sigma$ and $\tau$, we define an action of $G(S)_\mu$ on $G(S_0)$:
\begin{align}\label{eq-action}
G(S_0) \times G(S)_\mu \to G(S_0), \ (x, g) \mapsto \tau(g)^{-1} x \sigma(g).
\end{align}
If $\underline{S}$ is an \'etale sheaf of $W(k_0)$-frames on $\text{Spec }R$, this action sheafifies to provide an action of $G(\underline{S})_\mu$ on $G(\underline{S}_0)$. 

\begin{Def}\label{def-gdisp}
	Let $R$ be a $p$-nilpotent $W(k_0)$-algebra, and suppose $\underline{S}$ is an \'etale sheaf of $W(k_0)$-frames on $\text{Spec }R$. The \textit{stack of $G$-displays of type $\mu$ over $\underline{S}$} is the \'etale quotient stack
	\begin{align*}
	G\text{-}\textup{Disp}_{{\underline{S}},\mu} := [G(\underline{S}_0) / G(\underline{S})_{\mu}]
	\end{align*}
	over $\textup{\'Et}_R$, where $G(\underline{S})_{\mu}$ acts on $G(\underline{S}_0)$ via the action (\ref{eq-action}).
\end{Def}

Explicitly, for an \'etale $R$-algebra $R'$, $G\text{-}\textup{Disp}_{\underline{S},\mu}(R')$ is the groupoid of pairs $(Q,\alpha)$, where $Q$ is an \'etale locally trivial $G(\uS)_\mu$-torsor over $\text{Spec }R'$ and $\alpha: Q \to G(\uS_0)$ is a $G(\uS)_\mu$-equivariant morphism for the action (\ref{eq-action}).

Let us point out the case which will be of particular interest to us. Suppose $B \to A$ is a PD-thickening of $p$-nilpotent $W(k_0)$-algebras. If $A \to A'$ is \'etale, let $B(A')$ be the unique \'etale $B$-algebra with $B(A') \otimes_B A = A'$ (see e.g., \cite[\href{https://stacks.math.columbia.edu/tag/039R}{Tag 039R}]{stacks-project}). If $J = \ker(B \to A)$, then $\ker(B(A') \to A') = JB'$, and the divided powers on $B \to A$ extend to $B(A') \to A'$ by flatness of $B \to B(A')$, see \cite[\href{https://stacks.math.columbia.edu/tag/07H1}{Tag 07H1}]{stacks-project}. Denote by $\underline{W}_{B/A}$ the \'etale sheaf of frames defined by
\begin{align}\label{eq-W_{B/A}}
	\underline{W}_{B/A}(A') = \underline{W}(B(A')/A')
\end{align}
for $A \to A'$ \'etale (see Lemma \ref{lem-sheafqf}). By taking $\uS = \underline{W}_{B/A}$ in Definition \ref{def-gdisp} we obtain the stack of $G$-displays of type $\mu$ for $B\to A$
\begin{align*}
\GdispBA. 
\end{align*}

Following \cite[\textsection 7.4]{Lau2018}, we have a notion of a Hodge filtration for $G$-displays of type $\mu$, which will be useful later on for understanding deformations of $G$-displays along nilpotent thickenings. Let us recall this notion.

Let $\underline{S}$ be a frame for $R$, and let $(Q,\alpha)$ be a $G$-display of type $\mu$ over $\underline{S}$. Let $\bar{\tau}$ be the composition of $\tau: S \to S_0$ with the quotient $S_0 \to R$. Then $\bar{\tau}$ defines a morphism of \'etale sheaves on $\Spec R$
\begin{align}\label{eq-bartau}
	\bar{\tau}: G(\underline{S})_\mu \to G_R.
\end{align}
We write $Q_R$ for the $G_R$-torsor induced from $Q$ by $\tau$. 

Let $P_\mu \subset G$ be the subgroup scheme defined by $\mu$, that is
\begin{align}\label{eq-parabolic}
	 P_\mu(R) = \{h \in H(R) \mid \lim_{t \to 0} \mu(t)h\mu(t)^{-1} \text{ exists }\},
\end{align}
see \cite[Thm. 4.1.17]{Conrad2014}. By \cite[Prop. 6.2.2]{Lau2018}, the morphism (\ref{eq-bartau}) has image inside of $P_\mu$; write $\bar{\tau}_0$ for the resulting morphism $G(\underline{S})_\mu \to P_{\mu,R}$. 

\begin{Def}\label{def-hodgefiltG}
	Let $(Q,\alpha)$ be a $G$-display of type $\mu$ over $\underline{S}$. The \textit{Hodge filtration for $(Q,\alpha)$} is the $P_{\mu,R}$-torsor $Q_\mu \subset Q_R$ induced from $Q$ by $\bar{\tau}_0$. 
\end{Def}

We close this section by recalling the stack of $G$-displays of type $\mu$ over the Witt frame. Let $\underline{W}$ be the fpqc sheaf in frames on $\textup{Nilp}_{\zz_p}$ given by $R\mapsto \underline{W}(R).$ associated with $G$, $\mu$, and $W$ we have two group-valued functors on $\textup{Nilp}_{\zz_p}$:
\begin{align*}
L^+G := G(\underline{W}_0), \text{ and } L^+_\mu G := G(\underline{W})^0.
\end{align*}
By \cite[Lem. 5.4.1]{Lau2018} these are representable functors.

\begin{Def}\label{def-GDispW}
	The stack of \textit{$G$-displays of type $\mu$ over $\underline{W}$} is the \'etale quotient stack
	\begin{align*}
	\Gdisp := [L^+G / L^+_\mu G]
	\end{align*}
	over $\textup{Nilp}_{W(k_0)}$, where $L^+_\mu G$ acts on $L^+G$ via the action (\ref{eq-action}).
\end{Def}

\begin{rmk}
	One could also take the quotient stack with respect to the fpqc topology, which is the perspective used in \cite{Daniels2019}. The point is that the \'etale stack given by Definition \ref{def-GDispW} is an fpqc stack by \cite[Lem. 5.4.2]{Lau2018}.
\end{rmk}

\subsection{Tannakian $G$-displays}\label{sub-gdispdaniels}
Continuing the notation of the previous section, let $G$ be a flat affine group scheme of finite type over $\zz_p$, and let $\mu: \mathbb{G}_{m,W(k_0)} \to G_{W(k_0)}$ be a cocharacter for $G_{W(k_0)}$. Let us recall some definitions from \cite{Daniels2019}. If $(V,\pi)$ is any representation of $G$, then $V_{W(k_0)} = V\otimes_{\zz_p} W(k_0)$ is graded by the action of the cocharacter $\mu$, and for any $W(k_0)$-algebra $R$ we obtain an exact tensor functor $\mathscr{C}(\underline{W})_{\mu,R}$ (denoted $\mathscr{C}_{\mu,R}$ in \cite{Daniels2019}), given by
\begin{align*}
\mathscr{C}(\underline{W})_{\mu,R}: \textup{Rep}_{\zz_p}G \to \textup{PGrMod}(W(R)^\oplus), \ (V,\pi) \mapsto V_{W(k_0)} \otimes_{W(k_0)} W(R)^\oplus.
\end{align*}
We refer to an exact tensor functor $\mathscr{F}: \textup{Rep}_{\zz_p}G \to \textup{PGrMod}(W(R)^\oplus)$ as a \textit{graded fiber functor} over $W(R)^\oplus$, and we say $\mathscr{F}$ is \textit{of type $\mu$} if $\mathscr{F}$ is \'etale locally isomorphic to $\mathscr{C}(\underline{W})_{\mu,R}$. Let $\upsilon_R$ denote the forgetful functor $\textup{Disp}(\underline{W}(R)) \to \textup{PGrMod}(W(R)^\oplus)$. Recall the following definition (see \cite[Def. 3.14]{Daniels2019}).
\begin{Def}
	A \textit{Tannakian $(G,\mu)$-display over $\underline{W}(R)$} is an exact tensor functor
	\begin{align*}
	\mathscr{P}: \textup{Rep}_{\zz_p}G \to \textup{Disp}(\underline{W}(R))
	\end{align*}
	such that $\upsilon_R \circ \mathscr{P}$ is a graded fiber functor of type $\mu$.
\end{Def}
Denote the stack of Tannakian $(G,\mu)$-displays on $\textup{Nilp}_{W(k_0)}$ by $ G$-\textup{Disp}$_{\underline{W},\mu}^\otimes$. From a Tannakian $(G,\mu)$-display over $\underline{W}(R)$ we obtain a $G$-display $\mathscr{P}$ of type $\mu$ by taking $Q_\mathscr{P}$ to be the $L^+_\mu G$-torsor of trivializations of the underlying fiber functor of type $\mu$ of $\mathscr{P}$ and $\alpha_\mathscr{P}$ the morphism $Q \to L^+G$ coming from the Frobenius for $\mathscr{P}$, see \cite[Cons. 3.15]{Daniels2019} for details. The following is a consequence of the main theorem of \cite[\textsection 3]{Daniels2019}:
\begin{thm}\label{thm-thesis}
	The morphism 
	\begin{align*}
		G\textup{-Disp}_{\underline{W},\mu}^\otimes \to \Gdisp, \ \mathscr{P} \mapsto (Q_\mathscr{P},\alpha_\mathscr{P})
	\end{align*}
	is an equivalence of \'etale stacks on $\textup{Nilp}_{W(k_0)}$. 
\end{thm}
\begin{proof}
	This is proved in \cite[Thm. 2.16]{Daniels2019} in the case where $\Gdisp$ is given as the quotient for the fpqc topology and graded fiber functors of type $\mu$ are defined to be fpqc-locally isomorphic to $\mathscr{C}(\underline{W})_{\mu}$. The result follows in general because any fpqc-locally trivial $L^+_\mu G$-torsor is \'etale locally trivial (hence any graded fiber functor which is fpqc-locally trivial is \'etale locally trivial) by \cite[Lem. 5.4.2]{Lau2018}.
\end{proof}

In this section we prove a theorem analogous to Theorem \ref{thm-thesis} for $G$-displays of type $\mu$ over \'etale sheaves of frames with good descent properties. Let $R$ be a ring, and let $S$ be an \'etale sheaf of $\zz$-graded rings over $\Spec{R}$. We will denote by $\textup{PGrMod}_{S}$ the fibered category over $\textup{\'Et}_R$ whose fiber over an \'etale $R$-algebra $R'$ is $\textup{PGrMod}(S(R'))$. Further, if $\uS$ is a sheaf of frames, let $\textup{Disp}_{\uS}$ denote the fibered category of displays over $\uS$. 

\begin{Def}\label{def-descent}
	We say:
	\begin{itemize}
		\item An \'etale sheaf of $\zz$-graded rings $S$ on $\text{Spec }R$ \textit{satisfies descent for modules} if $\textup{PGrMod}_{S}$ is an \'etale stack over $\textup{\'Et}_R$. 
		\item An \'etale sheaf of frames $\underline{S}$ on $\Spec{R}$ \textit{satisfies descent for displays} if $\textup{Disp}_{\uS}$ is an \'etale stack over $\textup{\'Et}_R$.
	\end{itemize}
\end{Def}

\begin{lemma}\label{lem-descentmoddisp}
	Let $\underline{S}$ be an \'etale sheaf of frames on $\Spec{R}$ such that $\uS(R')$ is a frame for $R'$ for all \'etale $R$-algebras $R'$. If the underlying sheaf of $\mathbb{Z}$-graded rings $S$ satisfies descent for modules, then $\underline{S}$ satisfies descent for displays.
\end{lemma}
\begin{proof}
	That morphisms descend follows from Lemma \ref{lem-localproperties} \ref{lem-dispmorphism} and the fact that $S$ satisfies descent for modules. To prove that objects descend we need only to show that isomorphisms $\sigma^\ast M \xrightarrow{\sim} \tau^\ast M$ form an \'etale sheaf. But since $\underline{S}$ is an \'etale sheaf of frames, the functor $S_0: R'\mapsto S(R')_0$ is an \'etale sheaf of rings on $\Spec R$, and so for any finite projective $S(R)_0$-module $N$ the following sequence is exact:
	\begin{align*}
	0 \to N \to N \otimes_{S(R)_0} S(R')_0 \rightrightarrows N \otimes_{S(R)_0} S(R'\otimes_R R')_0,
	\end{align*}
	and the result follows.
\end{proof}

\begin{rmk}
	The frame of interest for the purposes of this paper is the relative Witt frame $\underline{W}(B/A)$ associated with a $p$-adic PD-thickening $B \to A$ (Example \ref{ex-relwittframe}). The \'etale sheaf of frames $A' \mapsto \underline{W}(B'/A')$ associated with $\underline{W}(B/A)$ (see \textsection \ref{sub-descent}) satisfies descent for modules (hence for displays as well, by Lemma \ref{lem-descentmoddisp}) by Proposition \ref{prop-descent}. The other primary example of a sheaf of frames which satisfies descent for modules is the \'etale sheaf of frames on $\Spec R$ associated with a $p$-adic frame $\underline{S}$ over $R$ (see \cite[Lem. 4.3.1]{Lau2018}). The Zink frame $\mathbb{W}(R)$ over an admissible ring $R$ \cite[Ex. 2.1.13]{Lau2018} and its relative analog \cite[Ex. 2.1.14]{Lau2018} for a PD-thickening $B \to A$ of admissible rings, as well as the truncated Witt frames over $\mathbb{F}_p$-algebras \cite[Ex. 2.1.6]{Lau2018} and their relative analogs are all examples of $p$-adic frames. The relative Witt frame $\underline{W}(B/A)$ for $B \to A$ is also a $p$-adic frame, but the \'etale sheaf of frames associated with it by \cite[Lem. 4.2.3]{Lau2018} using the $p$-adic topology differs from the one we consider here, which uses the natural topology for the Witt vectors (see \cite[Ex. 4.2.7]{Lau2018}).
\end{rmk}

\begin{Def}
	Let $S$ be a $\zz$-graded $W(k_0)$-algebra. A \textit{graded fiber functor} over $S$ is an exact tensor functor
	\begin{align*}
	\mathscr{F}: \textup{Rep}_{\zz_p}G \to \textup{PGrMod}(S).
	\end{align*}
\end{Def}

Denote by $\textup{GFF}(S)$ the category of graded fiber functors over $S$. Suppose $S$ is an \'etale sheaf of $\zz$-graded rings on $\text{Spec }R$. If $R \to R'$ is a homomorphism of \'etale $R$-algebras, the natural base change $M \mapsto M\otimes_{S(R)} S(R')$ induces a base change functor $\textup{GFF}(S(R)) \to \textup{GFF}(S(R'))$. In this way we obtain a fibered category $\textup{GFF}_{S}$ over $\textup{\'Et}_R$. 

\begin{lemma}\label{lem-gffstack}
	Let $\uS$ be an \'etale sheaf of frames such that the underlying sheaf of graded rings $S$ satisfies descent for modules. Then the fibered category $\textup{\textup{GFF}}_{S}$ is an \'etale stack over $\textup{\textup{\'Et}}_R$.
\end{lemma}
\begin{proof}
	The proof is the same as that of \cite[Lem. 3.5]{Daniels2019}, with Lemma \ref{lem-localproperties} \ref{lem-descentseq} replacing \cite[Lem. 2.12]{Daniels2019}.
\end{proof}

Suppose $R$ is a $W(k_0)$-algebra, and that $\uS$ is an \'etale sheaf of $W(k_0)$-frames over $\Spec{R}$ which satisfies descent for modules. For any cocharacter $\mu$ of $G$ defined over $W(k_0)$ and any \'etale $R$-algebra $R'$, we define a distinguished graded fiber functor over $S(R')$. Given a representation $(V,\pi)$ in $\textup{Rep}_{\zz_p}G$, let 
\begin{align*}
V_{W(k_0)}^i = \{v \in V_{W(k_0)} \mid (\pi \circ \mu)(z) \cdot v = z^i v \text{ for all } z \in \mathbb{G}_m(W(k_0))\}.
\end{align*} 
Then $\mu$ induces a canonical weight decomposition 
\begin{align}
V_{W(k_0)} = \bigoplus_{i \in \zz} V_{W(k_0)}^i.
\end{align}
Since any morphism of representations preserves the grading induced by $\mu$, we obtain an exact tensor functor
\begin{align}\label{eq-fiberfunctor}
\mathscr{C}(\uS)_{\mu,R'}: \textup{Rep}_{\zz_p}G \to \textup{PGrMod}_{S}(R'), \ V \mapsto V_{W(k_0)} \otimes_{W(k_0)} S(R').
\end{align}
If $R'$ is an \'etale $R$-algebra, then $\mathscr{C}(\uS)_{\mu,R'}$ is given by the composition of functors
\begin{align*}
\textup{Rep}_{\zz_p}G \xrightarrow{\mathscr{C}(\uS)_{\mu,R}} \textup{PGrMod}_{S}(R) \to \textup{PGrMod}_{S}(R'),
\end{align*}
where the second functor is the canonical base change. If $R$ is understood, we will suppress it in the notation and write $\mathscr{C}(\uS)_\mu$ for $\mathscr{C}(\uS)_{\mu,R}$.

\begin{Def}\label{def-typemu}
	A graded fiber functor $\mathscr{F}$ over $S(R)$ is \textit{of type $\mu$} if for some faithfully flat \'etale extension $R \to R'$ there is an isomorphism $\mathscr{F}_{R'}\cong \mathscr{C}(\uS)_{\mu,R'}$. 
\end{Def}

Let $\textup{GFF}_{S,\mu}$ denote the fibered category of graded fiber functors of type $\mu$. Since the property of being type $\mu$ is \'etale-local, $\textup{GFF}_{S,\mu}$ forms a substack of $\textup{GFF}_{S}$. If $\mathscr{F}_1$ and $\mathscr{F}_2$ are two graded fiber functors over $\uS$, denote by $\textup{\underline{Isom}}^\otimes(\mathscr{F}_1, \mathscr{F}_2)$ the \'etale sheaf of isomorphisms of tensor functors $\mathscr{F}_1 \xrightarrow{\sim} \mathscr{F}_2$. Let $\textup{\underline{Aut}}^\otimes(\mathscr{F}) = \textup{\underline{Isom}}^\otimes(\mathscr{F},\mathscr{F})$. The following is the analog of the main theorems of \cite[\textsection 3.2]{Daniels2019}. 

\begin{thm}\label{thm-isom}
	Let $\uS$ be an \'etale sheaf of $W(k_0)$-frames which satisfies descent for modules. The assignment $g \mapsto (\pi(g))_{(V,\pi)}$ defines an isomorphism of \'etale sheaves on $\Spec{R}$
	\begin{align*}
	G(\uS)_\mu \xrightarrow{\sim} \textup{\underline{Aut}}^\otimes(\mathscr{C}(\uS)_{\mu}),
	\end{align*}
	which, in turn, induces an equivalence of stacks
	\begin{align*}
	\textup{\textup{GFF}}_{\uS,\mu} \xrightarrow{\sim} \textup{\textup{Tors}}_{G(\uS)_\mu}, \ \mathscr{F} \mapsto \textup{\underline{Isom}}^\otimes(\mathscr{C}(S)_\mu, \mathscr{F}).
	\end{align*}
\end{thm}
\begin{proof}
	The arguments of \cite[\textsection 3.2]{Daniels2019} go through nearly verbatim, after replacing the Witt frame with $\uS$, and the fpqc topology with the \'etale topology.
\end{proof}

For any \'etale $R$-algebra $R'$ we have a forgetful functor
\begin{align}\label{eq-forgetful}
\upsilon_{S(R')}: \textup{Disp}_{\uS}(R') \to \textup{PGrMod}_{S}(R'), \ (M,F) \mapsto M.
\end{align}

\begin{Def}\label{def-gmu}
	Let $R$ be a $p$-nilpotent $W(k_0)$-algebra.
	\begin{itemize}
		\item A \textit{Tannakian $G$-display over $\uS(R)$} is an exact tensor functor
		\begin{align*}
		\mathscr{P}: \textup{Rep}_{\zz_p}G \to \textup{Disp}_{\uS}(R).
		\end{align*}
		\item A \textit{Tannakian $(G,\mu)$-display} over $\uS(R)$ is a Tannakian $G$-display $\mathscr{P}$ over $\uS(R')$ such that $\upsilon_{S(R)} \circ \mathscr{P}$ is a graded fiber functor of type $\mu$.
	\end{itemize}	
\end{Def}

If $R \to R'$ is \'etale, denote by $G$-\textup{Disp}$^\otimes(\uS(R'))$, resp. $G$-\textup{Disp}$_\mu^\otimes(\uS(R'))$ the category of Tannakian  $G$-displays, resp. the full subcategory of Tannakian $(G,\mu)$-displays over $\uS(R')$. By an analog of Lemma \ref{lem-gffstack} we see that Tannakian  $G$-displays form an \'etale stack $G$-\textup{Disp}$_{\uS}^\otimes$ over $\textup{\'Et}_R$, and Tannakian $(G,\mu)$-displays define a substack $G$-\textup{Disp}$_{\uS,\mu}^\otimes$. 

There are a number of useful functorialities between categories of Tannakian $G$-displays. If $\mathscr{P}$ is a Tannakian $G$-display over $\uS(R)$ and $\psi: R \to R'$ is homomorphism of $p$-nilpotent $W(k_0)$-algebras, we denote by $\psi^\ast \mathscr{P}$ or $\mathscr{P}_{\underline{S}(R')}$ the base change of $\mathscr{P}$, which is given by
\begin{align*}
\textup{Rep}_{\zz_p}G \xrightarrow{\mathscr{P}} \textup{Disp}_{\uS}(R) \to \textup{Disp}_{\uS}(R').
\end{align*}
Similarly, if $\alpha: \uS \to \underline{S'}$ is a morphism of \'etale sheaves of frames, we obtain a base change functor 
\begin{align}\label{eq-changeofframe}
\alpha^\ast: G\textup{-Disp}_{\uS}^\otimes \to G\textup{-Disp}_{\underline{S'}}^\otimes
\end{align}	
given by post-composition with $\textup{Disp}(\underline{S}(R)) \to \textup{Disp}(\underline{S'}(R))$. Finally, if $\gamma:G \to G'$ is a homomorphism of $\zz_p$-group schemes, and $\mathscr{P}$ is a Tannakian $G$-display over $\underline{S}(R)$, we denote by $\gamma( \mathscr{P})$ the  $G'$-display
\begin{align*}
\textup{Rep}_{\zz_p}G' \xrightarrow{\text{res}} \textup{Rep}_{\zz_p}G \xrightarrow{\mathscr{P}} \textup{Disp}_{\underline{S}}(R).
\end{align*}
If $\mathscr{P}$ is a Tannakian $(G,\mu)$-display, then $\gamma(\mathscr{P})$ is a Tannakian $(G', \gamma\circ\mu)$-display. 

To any Tannakian $(G,\mu)$-display we can associate a $G$-display of type $\mu$. Let us summarize the construction (see \cite[Constr. 3.15]{Daniels2019} for details). Let $\mathscr{P}$ be a Tannakian $(G,\mu)$-display over $\uS(R)$. By Theorem \ref{thm-isom}, 
\begin{align*}
Q_{\mathscr{P}}:= \underline{\textup{Isom}}^\otimes(\mathscr{C}(\uS)_{\mu,R},\upsilon_{S(R)}\circ\mathscr{P})
\end{align*}
is a $G(\uS)_\mu$-torsor over $R$. If $R'$ is an \'etale $R$-algebra, write $\mathscr{P}_{R'}(V,\pi) = (M(\pi)',F(\pi)')$ for any $(V,\pi)$ in $\textup{Rep}_{\zz_p}G$. Given an isomorphism of tensor functors $\lambda: \mathscr{C}(\uS)_{\mu,R'} \xrightarrow{\sim} \upsilon_{S(R')} \circ \mathscr{P}_{R'}$, we obtain an automorphism 
\begin{align*}
\alpha_{\mathscr{P}}(\lambda)^\pi := \tau^\ast(\lambda^\pi) \circ (F(\pi)')^\sharp \circ \sigma^\ast(\lambda^{\pi})
\end{align*}
of $V \otimes_{\zz_p} S(R')_0$ for every $(V,\pi)$ in $\textup{Rep}_{\zz_p}G$. If $\omega_{S(R')_0}$ denotes the canoncial fiber functor $(V,\pi) \mapsto V \otimes_{\zz_p} S(R')_0$, then the collection $(\alpha_{\mathscr{P}}(\lambda)^\pi)_{(V,\pi)}$ constitutes an element of $\text{Aut}^\otimes(\omega_{S(R')_0})$. By Tannakian duality \cite[Thm. 44]{Cornut2014}, the map $g \mapsto (\pi(g))_{(V,\pi)}$ determines an isomorphism
\begin{align*}
G(S(R')_0) \cong \text{Aut}^\otimes(\omega_{S(R')_0}),
\end{align*}
so there is some $\alpha_{\mathscr{P}}(\lambda) \in G(S(R')_0) = G(\uS_0)(R')$ such that $\pi(\alpha_{\mathscr{P}}(\lambda)) = \alpha_{\mathscr{P}}(\lambda)^\pi$ for every $(V,\pi)$. Altogether the assignment $\lambda \mapsto \alpha_{\mathscr{P}}(\lambda)$ defines a morphism of \'etale sheaves
\begin{align}\label{eq-alpha}
\alpha_{\mathscr{P}}: Q_{\mathscr{P}} \to G(\uS_0). 
\end{align}
As in \cite[Constr. 3.15]{Daniels2019} one checks that the association $\mathscr{P} \mapsto (Q_{\mathscr{P}},\alpha_{\mathscr{P}})$ is functorial in $\mathscr{P}$ and compatible with base change, so we obtain a morphism of stacks
\begin{align}\label{eq-morphism}
G\text{-}\textup{Disp}_{\uS,\mu}^\otimes \to G\text{-}\textup{Disp}_{\uS,\mu}, \ \mathscr{P} \mapsto (Q_{\mathscr{P}}, \alpha_{\mathscr{P}}).
\end{align}
The following is the analog of \cite[Thm. 3.16]{Daniels2019}.
\begin{thm}\label{thm-equiv}
	If $\uS$ satisfies descent for modules, the morphism $($\ref{eq-morphism}$)$ is an equivalence of \'etale stacks over $\textup{\textup{\'Et}}_R$.
\end{thm}
\begin{proof}
	The proof of \cite[Thm. 3.16]{Daniels2019} goes through here as well, after replacing the Witt frame by the frame $\uS$, and the fpqc topology by the \'etale topology. Let us sketch the argument.
	
	By the first part of Theorem \ref{thm-isom}, the functor is faithful. If $\mathscr{P}_1$ and $\mathscr{P}_2$ are Tannakian $(G,\mu)$-displays over $R$, and $\eta: (Q_{\mathscr{P}_1},\alpha_{\mathscr{P}_1}) \to (Q_{\mathscr{P}_2},\alpha_{\mathscr{P}_2})$ is a morphism, then the second part of Theorem \ref{thm-isom} provides us with a morphism $\psi: \upsilon_R \circ \mathscr{P}_1 \to \upsilon_R \circ \mathscr{P}_2$ which induces $Q_{\mathscr{P}_1} \to Q_{\mathscr{P}_2}$. It remains only to check this morphism is compatible with the respective Frobeneius morphisms, but by Lemma \ref{lem-localproperties} \ref{lem-dispmorphism} it is enough to check this after some faithfully flat \'etale extension $R \to R'$. By choosing an extension such that $Q_{\mathscr{P}_1}(R')$ is nonempty, the result follows from the definitions of the $\alpha_{\mathscr{P}_i}$. Finally, to complete the proof it is enough to show that every $G$-display of type $\mu$ over $\uS(R)$ is \'etale locally in the essential image of (\ref{eq-morphism}), which is done using Theorem \ref{thm-isom}. 
\end{proof}

\begin{cor}
	Let $B\to A$ be a PD-thickening of $p$-nilpotent $W(k_0)$-algebras. Then $(\ref{eq-morphism})$ induces an equivalence
	\begin{align*}
	G\text{-}\textup{\textup{Disp}}_{\underline{W}(B/A),\mu}^\otimes \xrightarrow{\sim} G\text{-}\textup{\textup{Disp}}_{\underline{W}(B/A),\mu}.
	\end{align*}	
\end{cor}
\begin{proof}
	Combine Theorem \ref{thm-equiv} with Proposition \ref{prop-descent}.
\end{proof}

\begin{rmk}\label{rmk-GLn}
	Let $\Lambda$ be a finite free $\zz_p$-module, and let $\mu$ be a cocharacter of $\textup{GL}(\Lambda)$. Let us say a display $\underline{M}=(M,F)$ over $\underline{S}(R)$ is of type $\mu$ if, \'etale locally, there is an isomorphism $M \cong \Lambda \otimes_{\zz_p} S(R)$ of graded $S(R)$-modules, where $\Lambda$ is graded by the weight space decomposition of the cocharacter $\mu$. Denote by $\textup{Disp}_\mu(\underline{S}(R))$ the category of displays over $\underline{S}(R)$ which are of type $\mu$. Then one checks (as in \cite[Thm. 5.15]{Daniels2019}, for example) that the functor 
	\begin{align*}
	\textup{GL}(\Lambda)\text{-}\textup{Disp}^\otimes_{\underline{S},\mu}(R) \to \textup{Disp}_\mu(\underline{S})
	\end{align*}
	induced by evaluation on the standard representation is an equivalence of categories. If $I = (i_1, i_2, \dots, i_n) \in \zz^n$ with $i_1 \le i_2 \le \cdots \le i_n$, and $\mu_I$ is the cocharacter $t \mapsto \text{diag}(t^{i_1}, t^{i_2}, \dots, t^{i_n})$ for some choice of basis of $\Lambda$, then this is compatible with the equivalence between $\text{GL}_n\text{-}\textup{Disp}_{\underline{S},\mu_I}$ and the stack of displays of type $I$ over $\underline{S}$ described in \cite[Ex. 5.3.5]{Lau2018} (see also \cite[Rmk. 3.2]{Daniels2019}). 
	
	Suppose now $\underline{S}(R)$ extends some $1$-frame $\mathcal{S}$. We say that a window over $\mathcal{S}$ is of type $\mu$ if the corresponding $1$-display is of type $\mu$. If $\mu$ is minuscule, the functor described above is valued in $1$-displays, and therefore induces an equivalence between  Tannakian $(\textup{GL}(\Lambda),\mu)$-displays over $\underline{S}(R)$ and windows over $\mathcal{S}$ of type $\mu$. In particular, if $I = (0^{(d)}, 1^{(h-d)})$ for some $d$, and $\mu = \mu_I$, then $\textup{GL}_h\text{-}\textup{Disp}^\otimes_{\underline{S},\mu_I}(R)$ is equivalent to the category of windows $(P_0, P_1, F_0, F_1)$ over $\mathcal{S}$ with $\text{rk}_{S_0} P_0 =h$ and $\text{rk}_{R}(P_0/P_1) = d$. 
\end{rmk}

Let us now summarize the local description of the stack $G$-\textup{Disp}$_{\uS,\mu}^\otimes$. Let us again assume that $R$ is a $W(k_0)$-algebra, and that $\uS$ is an \'etale sheaf of $W(k_0)$-frames over $\Spec R$ which satisfies descent for modules. 

\begin{Def}\label{def-banal}
	A  Tannakian $(G,\mu)$-display $\mathscr{P}$ over $\uS(R)$ is \textit{banal} if there is an isomorphism $\upsilon_{S(R)} \circ \mathscr{P} \cong \mathscr{C}(\uS)_{\mu,R}$.
\end{Def}

If $\mathscr{P}$ is a  Tannakian $(G,\mu)$-display over $R$, then $\mathscr{P}$ is banal locally for the \'etale topology on $R$. Given any $U \in G(S(R)_0)$ we can define a banal  Tannakian $(G,\mu)$-display $\mathscr{P}_U$ on $\uS(R)$ as follows: to the representation $(V,\pi)$ we associate the display over $\uS(R)$ defined from the standard datum 
\begin{align*}
(V \otimes_{\zz_p} S(R)_0, \pi(U) \circ (\id \otimes \sigma_0)),
\end{align*}
where $V \otimes_{\zz_p} S(R)_0 = V_{W(k_0)} \otimes_{W(k_0)} S(R)_0$ is endowed with the grading induced by the cocharacter $\mu$. 

\begin{prop}\label{prop-banal} \hspace{2cm}
	\begin{enumerate}[\textup{(}i\textup{)}]
		\item Every banal  Tannakian $(G,\mu)$-display $\mathscr{P}$ over $R$ is isomorphic to $\mathscr{P}_U$ for some $U \in G(S(R)_0)$.
		\item The category of banal  Tannakian $(G,\mu)$-displays over $R$ is equivalent to the category whose objects are $U \in G(S(R)_0)$ and whose morphisms are given by
		\begin{align*}
		\textup{Hom}(U,U') = \{h \in G(\uS)_\mu(R) \mid \tau(h)^{-1} U' \sigma(h) = U\}.
		\end{align*}
	\end{enumerate}
\end{prop}
\begin{proof}
	The proof follows from the arguments at the end of \cite[\textsection 3.3]{Daniels2019}. 
\end{proof}

\begin{rmk}
	If $(Q,\alpha)$ is the $G$-display of type $\mu$ corresponding to a Tannakian $(G,\mu)$-display $\mathscr{P}$, then $\mathscr{P}$ is banal if and only if $Q$ is a trivial torsor, and if $\beta: G(\uS)_\mu \xrightarrow{\sim} Q$ is a trivialization, then giving $U$ as in Proposition \ref{prop-banal} is equivalent to giving $\alpha(\beta(1))$. 
\end{rmk}

\subsection{The Hodge filtration for Tannakian $G$-displays}\label{sub-hodgefilt}

Let $G$ be a flat affine group scheme of finite type over $\zz_p$ and let $\mu$ be a cocharacter for $G_{W(k_0)}$. Suppose $R$ is in $\textup{Nilp}_{W(k_0)}$, and let $\underline{S}$ be an \'etale sheaf of frames on $\Spec R$. In this section we define the Hodge filtration for Tannakian $G$-displays over $\underline{S}(R)$ and compare it to the Hodge filtration for $G$-displays of type $\mu$ as defined in \textsection \ref{sub-gdisplau}. 

For any ring $R$, let us denote by $\textup{Fil}(R)$ the category of finite projective $R$-modules $M$ equipped with a descending filtration by direct summands $(\textup{Fil}^n(M))_{n\in\zz}$. This is an exact tensor category with tensor product on the filtrations defined as in (\ref{eq-hodgetensor}). To any graded fiber functor $\mathscr{F}$ over $S(R)$ we can attach an exact tensor functor
\begin{align}\label{eq-hodgefunctor}
	\textup{Fil}_\mathscr{F}: \textup{Rep}_{\zz_p} G \to \textup{Fil}(R)
\end{align}
by assigning to any $(V,\rho)$ the filtered $R$-module obtained by tensoring the evaluation of $\mathscr{F}$ on $(V,\rho)$ along $\bar{\tau}: S \xrightarrow{\tau} S_0 \to R$. In particular, if $\mathscr{F} = \mathscr{C}(\underline{S})_{\mu}$ for some cocharacter $\mu$ of $G$, then for any representation $(V,\rho)$, the $i$th filtered piece of $\textup{Fil}_{\mathscr{F}}(V,\rho)$ is given by 
\begin{align*}
	\textup{Fil}^i = \bigoplus_{i \ge n} V^i \otimes_{W(k_0)} R.
\end{align*}
Hence in this case $\textup{Fil}_\mathscr{\mathscr{C}(\underline{S})_\mu}$ is the canonical functor 
\begin{align}\label{eq-filmu}
	\textup{Fil}_\mu: \textup{Rep}_{\zz_p}(G) \to \textup{Fil}(R)
\end{align}
associated to the cocharacter $\mu$.

\begin{Def}
	A \textit{fiber functor} for $R$ is an exact tensor functor 
\begin{align*}
	\omega: \textup{Rep}_{\zz_p} G \to \textup{Mod}(R)
\end{align*}
such that $\omega_{\mathscr{F}}$ is \'etale locally (on $\Spec R$) isomorphic to the functor
\begin{align}\label{eq-usualfiberfunctor}
	\omega_R: \textup{Rep}_{\zz_p} G \to \textup{Mod}(R), \ (V,\rho) \mapsto V\otimes_{\zz_p} R.
\end{align}
If $\omega$ is a fiber functor for $R$, a functor $\textup{Fil}: \textup{Rep}_{\zz_p} G \to \textup{Fil}(R)$ is a \textit{filtration} of $\omega$ if $\omega$ factors into the composition
	\begin{align*}
		\textup{Rep}_{\zz_p} G \xrightarrow{\textup{Fil}} 	\textup{Fil}(R) \to \textup{Mod}(R).
	\end{align*}
\end{Def}

Given a graded fiber functor $\mathscr{F}$ of type $\mu$ over $S(R)$, define the exact tensor functor
\begin{align*}
	\omega_{\mathscr{F}}: \textup{Rep}_{\zz_p} G \to \textup{Mod}(R)
\end{align*}
by postcomposing $\textup{Fil}_{\mathscr{F}}$ with the forgetful functor $\textup{Fil}(R) \to \textup{Mod}(R)$. Since $\mathscr{F}$ is \'etale locally isomorphic to $\mathscr{C}(\underline{S})_\mu$, $\omega_\mathscr{F}$ is a fiber functor. Moreover, $\textup{Fil}_\mathscr{F}$ is obviously a filtration of $\omega_{\mathscr{F}}$. 

By Tannakian duality, there is a natural isomorphism $\textup{Aut}^\otimes(\omega_R) \xrightarrow{\sim} G_R$ (see e.g. \cite[Thm. 44]{Cornut2014}). It follows that \begin{align}\label{eq-GRtors}
	Q_{\omega_{\mathscr{F}}} := \textup{Isom}^\otimes(\omega_R,\omega_{\mathscr{F}})
\end{align}
is a $G_R$-torsor on $\Spec R$. If $Q_\mathscr{F}$ is the $G(\underline{S})_\mu$-torsor associated to $\mathscr{F}$ by Theorem \ref{thm-isom}, then $Q_{\omega_\mathscr{F}}$ is isomorphic to the $G_R$-torsor $Q_{\mathscr{F},R}$ induced from $Q_\mathscr{F}$ by $\bar{\tau}: G(\underline{S})_\mu \to G_R$ (see (\ref{eq-bartau})). Indeed, base change along $\bar{\tau}: S \to R$ induces a $G_R$-equivariant morphism
\begin{align}\label{eq-GRtorsors}
	Q_{\mathscr{F}} = \textup{Isom}^\otimes(\mathscr{C}(\underline{S})_\mu, \mathscr{F}) \to \textup{Isom}^\otimes(\omega_R, \omega_{\mathscr{F}}) = Q_{\omega_\mathscr{F}},
\end{align} 
which is necessarily an isomorphism of $G_R$-torsors.

Let $\mathscr{P}$ be a Tannakian $G$-display over $\underline{S}(R)$. Since $\upsilon_{S(R)} \circ \mathscr{P}$ is a graded fiber functor, we can define from $\mathscr{P}$ functors $\omega_\mathscr{P}$ and $\textup{Fil}_\mathscr{P}$ by
\begin{align}\label{eq-hodgefiltdef}
	\omega_\mathscr{P} := \omega_{\upsilon_{S(R)} \circ \mathscr{P}} \text{  and  } \textup{Fil}_\mathscr{P} := \textup{Fil}_{\upsilon_{S(R)} \circ \mathscr{P}}.
\end{align}

\begin{Def}\label{def-hodgefunctor}
	Let $\mathscr{P}$ be a Tannakian $G$-display over $\underline{S}(R)$. The \textit{Hodge filtration} for $\mathscr{P}$ is the exact tensor functor $\textup{Fil}_\mathscr{P}$ defined in (\ref{eq-hodgefiltdef}). 
\end{Def}

We can equivalently define $\textup{Fil}_\mathscr{P}$ as the functor that assigns to every $(V,\rho)$ the Hodge filtration of its corresponding display over $\underline{S}(R)$ as in (\ref{eq-hodgedisp}). 

Suppose now $G$ is reductive and $\mu$ is a cocharacter for $G$. Let $\mathscr{P}$ be a Tannakian $(G,\mu)$-display over $\underline{S}(R)$. Let us compare the Hodge filtration of $\mathscr{P}$ to that of its associated $G$-display of type $\mu$, $(Q_\mathscr{P}, \alpha_\mathscr{P})$, see Definition \ref{def-hodgefiltG}. By \cite[Thm. 60 and Rmk. 54]{Cornut2014}, the \'etale sheaf on $\Spec R$ of automorphisms of the tensor functor $\textup{Fil}_\mu$ (see (\ref{eq-filmu})) is isomorphic to the sheaf associated to the parabolic subgroup $P_{\mu,R}$ of $G_R$ (see (\ref{eq-parabolic})). Thus the subsheaf \begin{align}\label{eq-Pmutors}
	\textup{Isom}^\otimes(\textup{Fil}_\mu, \textup{Fil}_\mathscr{P}) \subset Q_{\omega_\mathscr{F}}
\end{align}
of tensor-isomorphisms $\textup{Fil}_\mu \xrightarrow{\sim} \textup{Fil}_\mathscr{P}$ is an \'etale $P_{\mu,R}$-torsor on $\Spec R$, which we denote by $Q_{\textup{Fil}_\mathscr{P}}$. If $Q_{\mathscr{P},\mu} \subset Q_{\mathscr{P},R}$ denotes the Hodge filtration of the associated $G$-display of type $\mu$, then the natural map
\begin{align*}
	Q_\mathscr{P} = \textup{Isom}^\otimes(\mathscr{C}(\underline{S}_{\mu,R}, \upsilon_{S(R)}\circ \mathscr{P}) \to \textup{Isom}^\otimes(\textup{Fil}_{\mu}, \textup{Fil}_\mathscr{P})=Q_{\textup{Fil}_{\mathscr{P}}}
\end{align*} 
given by base change along $\bar{\tau}: S \to R$ induces $P_{\mu,R}$-equivariant map
\begin{align} \label{eq-hodgecompare}
	Q_{\mathscr{P},\mu} \to Q_{\textup{Fil}_{\mathscr{P}}}
\end{align}
which is therefore necessarily an isomorphism of $P_{\mu, R}$-torsors. Moreover, (\ref{eq-hodgecompare}) is compatible with the inclusions $Q_{\mathscr{P},\mu} \subset Q_{\mathscr{P},R}$ and $Q_{\textup{Fil}_\mathscr{P}} \subset Q_{\omega_{\mathscr{P}}}$ and the isomorphism (\ref{eq-GRtorsors}). 

We close this section by explaining the way that the Hodge filtration controls lifts along certain homomorphisms of \'etale sheaves of frames, following \cite[\textsection 7.4]{Lau2018}. Let $B \to A$ be a homomorphism of $p$-nilpotent $W(k_0)$-algebras such that $J = \ker(B \to A)$ is locally nilpotent; i.e., such that $x^n = 0$ for some $n$ for all $ x \in J$ (for example, $B \to A$ could be a PD-thickening). Let $\underline{S}'$ be an \'etale sheaf of frames on $\Spec B$, and let $\underline{S}$ be an \'etale sheaf of frames on $\Spec A$. Since $J$ is locally nilpotent, for every \'etale $A$-algebra $A'$ there exists a unique \'etale $B$-algebra $B(A')$ lifting $A'$, so we can consider $\underline{S}'$ as an \'etale sheaves of frames on $\Spec A$. We assume that for all \'etale $A$-algebras $A'$, $\underline{S}(A')$ is a frame for $A'$, and $\underline{S}'(B(A'))$ is a frame for $B(A')$. 

Suppose now that we have a morphism of \'etale sheaves of frames $\beta:\underline{S}' \to \underline{S}$ such that, for all \'etale $A$-algebras $A'$, $\beta_0: \underline{S}'(B(A'))_0 \to \underline{S}(A')_0$ is bijective. The morphism $\beta$ determines a morphism of stacks on $\Spec A$
\begin{align}\label{eq-reduction2}
	G\textup{-Disp}^\otimes_{\underline{S}',\mu} \to G\textup{-Disp}^\otimes_{\underline{S},\mu}.
\end{align}

Let us continue to assume that $G$ is reductive over $\zz_p$ and let us suppose now that $\mu$ is minuscule. Let $\mathscr{P}'$ be a Tannakian $(G,\mu)$-display over $\underline{S}'(B)$. Applying the morphism (\ref{eq-reduction2}) to $\mathscr{P}'$, we obtain a Tannakian $(G,\mu)$-display $\mathscr{P}$ over $\underline{S}(A)$. Associated with $\mathscr{P}'$ is the Hodge filtration $\textup{Fil}_{\mathscr{P}'}$ of $\mathscr{P}'$ from Definition \ref{def-hodgefunctor}, which is a filtration of the functor $\omega_{\mathscr{P}'}: \textup{Rep}_{\zz_p} G \to \textup{Mod}(B)$, see (\ref{eq-hodgefiltdef}). Moreover, we have the Hodge filtration $\textup{Fil}_{\mathscr{P}}$ of $\mathscr{P}$. Then $\textup{Fil}_{\mathscr{P}'}$ is a lift of $\textup{Fil}_\mathscr{P}$ along $B \to A$.

Since $\beta_0$ is bijective, we have a morphism $\tilde{\tau}: S(A) \to S(A)_0 \xrightarrow{\sim} S'(B)_0 \to B$, with the property that the composition of $\tilde{\tau}$ with the natural map $S'(B) \to S(A)$ is the map $\bar{\tau}_0: S(B) \to B$. Therefore base change along the map $\tilde{\tau}$ induces a fiber functor
\begin{align*}
	\omega_{\mathscr{P},B}: \textup{Rep}_{\zz_p} G \to \textup{Mod}(B),
\end{align*}
with the property that $\omega_{\mathscr{P},B}$ composed with the base change $\textup{Mod}(B) \to \textup{Mod}(A)$ is the fiber functor $\omega_{\mathscr{P}}$ associated to $\mathscr{P}$. Since the composition $S(B) \to S(A) \xrightarrow{\tilde{\tau}} B \to B$ is $\bar{\tau}: S(B) \to B$, we see that in fact $\omega_{\mathscr{P},B} = \omega_{\mathscr{P}'}$ is the fiber functor associated to $\mathscr{P}'$ in the case where $\mathscr{P}$ is obtained from base change from $\mathscr{P}'$ along $\underline{S}' \to \underline{S}$.

\begin{prop}\label{prop-Gdisplift}
	Suppose $\beta: \underline{S}' \to \underline{S}$ is a morphism of frames as above such that $\beta_1: \underline{S}'(B(A'))_1 \to \underline{S}(A')_1$ is injective for all \'etale $A$-algebras $A'$. Then the assignment
	\begin{align*}
		\mathscr{P}' \mapsto (\mathscr{P}, \textup{Fil}_{\mathscr{P}'})
	\end{align*}
	described above determines an equivalence of categories between Tannakian $(G,\mu)$-displays over $\underline{S}'(B)$ and Tannakian $(G,\mu)$-displays over $\underline{S}(A)$ together with a filtration $\textup{Fil}$ of the fiber functor $\omega_{\mathscr{P},B}$.
\end{prop}
\begin{proof}
	By Theorem \ref{thm-thesis} and Theorem \ref{thm-equiv} along with the comparison of the respective Hodge filtrations (\ref{eq-hodgecompare}), it is enough to show the result for $G$-displays of type $\mu$. This follows from the arguments of \cite[\textsection 7.4]{Lau2018} with the following remarks. In \textit{loc. cit.} this is shown for any morphism of $p$-adic frames $\underline{S}' \to \underline{S}$ for $R'$ and $R$, respectively, over $W(k_0)$ with $S_0' = S_0$, such that the property 
	\begin{align}\label{eq-property}
		S_1' \to S_1 \text{ is injective, and } S_1/S_1' = \ker(R' \to R)
	\end{align}
	is satisfied (recall that we are assuming $G$ is reductive and $\mu$ is minuscule). In the case of \cite{Lau2018}, it follows from the fact that $\underline{S}'$ and $\underline{S}$ are $p$-adic frames that the property (\ref{eq-property}) is preserved after \'etale base change, and therefore the result follows from \cite[Lem. 7.4.2]{Lau2018}. In our case, it is preserved by assumption. Thus once again the result follows from \cite[Lem. 7.4.2]{Lau2018}.
\end{proof}

In particular, this result applies when $B\to A$ is a PD-thickening, $\underline{S}' = \underline{W}(B)$ is the Witt frame for $B$, and $\underline{S} = \underline{W}(B/A)$ is the relative Witt frame for $B \to A$. Indeed, $W(B)^\oplus_0 = W(B/A)^\oplus_0 = W(B)$, and $W(B)^\oplus_1 \to W(B/A)^\oplus_1$ is the inclusion $I(B) \hookrightarrow I(B/A)$, see Example \ref{ex-relwittframe}. Moreover, these properties clearly hold for all \'etale $A$-algebras $A'$. 

\subsection{Adjoint nilpotence and liftings}\label{sub-lifting}
In this section we assume that $G$ is a reductive group scheme over $\zz_p$ and that $\mu: \mathbb{G}_{m,W(k_0)} \to G_{W(k_0)}$ is a minuscule cocharacter of $G_{W(k_0)}$. Let $B \to A$ be a PD-thickening of $p$-nilpotent $W(k_0)$-algebras. We first fit the adjoint nilpotence condition of \cite[\textsection 3.4]{BP2017} into the present context, and state Lau's unique lifting lemma for adjoint nilpotent Tannakian $(G,\mu)$-displays along $\underline{W}(B/A) \to \underline{W}(A)$ (Proposition \ref{prop-lifting}). We then explain (in our context) Lau's classification of lifts of Tannakian $(G,\mu)$-displays along $\underline{W}(B) \to \underline{W}(B/A)$ by lifts of the Hodge filtration.

Recall that $G$-\textup{Disp}$_{\underline{W},\mu}^\otimes$ (equiv. $G$-\textup{Disp}$_{\underline{W},\mu}$) is a stack for the \'etale topology on $\textup{Nilp}_{W(k_0)}$. For a $p$-nilpotent $W(k_0)$-algebra $A$, we can restrict the stack $ G$-\textup{Disp}$_{\underline{W},\mu}^\otimes$ (resp. $G$-\textup{Disp}$_{\underline{W},\mu}$) to obtain an \'etale stack on $\Spec{A}$, which we will denote by  $G$-\textup{Disp}$_{\underline{W}(A),\mu}^\otimes$ (resp. $G$-\textup{Disp}$_{\underline{W}(A),\mu}$). Alternatively this is $G$-\textup{Disp}$_{\underline{W}_A,\mu}$, where $\underline{W}_A$ is the \'etale sheaf on $\Spec A$ defined by
\begin{align}\label{eq-W_A}
	\underline{W}_A(A') = W(A')
\end{align}
for all \'etale $A$-algebras $A'$.

Let $k$ be a perfect field of characteristic $p$, and let $K = W(k)[1/p]$. The Frobenius $\sigma$ of $W(k)$ naturally extends to $K$. Denote by $\textup{$F$-Isoc}(k)$ the category $F$-isocrystals over $k$, i.e., the category of pairs $(M,\varphi)$ consisting of a finite-dimensional $K$-vector space $M$ and an isomorphism of $K$-vector spaces $\varphi: \sigma^\ast M \xrightarrow{\sim} M$. When $k$ is algebraically closed, $\textup{$F$-Isoc}(k)$ is a semi-simple category with simple objects parametrized by $\lambda \in \qq$ (see e.g., \cite{Demazure1972}). In that case, for $\lambda \in \qq$, we write $M_\lambda$ for the $\lambda$-isotypic component of $M$, and if $M_\lambda$ is nonzero we will say $\lambda$ is a \textit{slope} of $(M,\varphi)$. 

Let $R$ be a $k_0$-algebra, and let $\mathscr{P}$ be a  Tannakian $(G,\mu)$-display over $\underline{W}(R)$. For every point $x \in \Spec{R}$, choose an algebraic closure $k(x)$ of the residue field of $x$. The base change $\mathscr{P}_{k(x)}$ of $\mathscr{P}$ to $k(x)$ is banal, since the $L^+_\mu G$-torsor $\underline{\textup{Isom}}^\otimes(\mathscr{C}(\underline{W})_{\mu,k(x)},\upsilon_{k(x)} \circ \mathscr{P}_{k(x)})$ over $k(x)$ is trivial. Hence by Proposition \ref{prop-banal} there is some $u(x) \in L^+G(k(x)) = G(W(k(x))$ such that $u(x)$ determines $\mathscr{P}_{k(x)}$. Let $K(x) = W(k(x))[1/p]$, and define $b(x) = u(x)\mu^\sigma(p) \in G(K(x))$. To $b(x)$ we can associate an exact tensor functor
\begin{align*}
N_{b(x)}: \textup{Rep}_{\qq_p}(G) \to \textup{$F$-Isoc}(k(x)), \ (V,\pi) \mapsto (V\otimes_{\qq_p} K(x),\pi(b(x)) \circ (\id_V \otimes \sigma)).
\end{align*}
Let us denote by $(\mathfrak{g},\text{Ad}^G)$ the adjoint representation of $G$. 
\begin{Def}\label{def-adjnilp}
	Let $R$ be a $p$-nilpotent $W(k_0)$-algebra. A  Tannakian $(G,\mu)$-display $\mathscr{P}$ over $\underline{W}(R)$ is \textit{adjoint nilpotent} if for all $x \in \Spec R/pR$ all slopes of the isocrystal $N_{b(x)}(\mathfrak{g},\textup{Ad}^G)$ are greater than $-1$. 
\end{Def}
We will likewise say that $U \in L^+G(R)$ is adjoint nilpotent over $\underline{W}(R)$ if the associated  Tannakian $(G,\mu)$-display $\mathscr{P}_U$ is adjoint nilpotent. See \cite[\textsection 3.4]{BP2017} for a discussion of this condition. 

Let $\Lambda$ be a finite free $\zz_p$-module. Let us briefly recall the relationship between adjoint nilpotence and Zink's nilpotence condition in the case where $G = \textup{GL}(\Lambda)$ (cf. \cite[Rmk. 3.4.5]{BP2017}). If $R$ is a $p$-nilpotent $\zz_p$-algebra, then by Remark \ref{rmk-GLn}, evaluation on the standard representation $(\Lambda, \iota)$ defines an equivalence of categories between  Tannakian $(\textup{GL}(\Lambda),\mu)$-displays and 1-displays of type $\mu$ over $\underline{W}(R)$. We will say that a 1-display is nilpotent if its corresponding Zink display (under the equivalence in Lemma \ref{lem-windows}) satisfies Zink's nilpotence condition (see \cite[Def. 11]{Zink2002}). 

\begin{lemma}\label{lem-nilpotent}
	Suppose $\mathscr{P}$ is a  $(\textup{GL}(\Lambda), \mu)$-display over $\underline{W}(R)$ such that $\mathscr{P}(\Lambda, \iota)$ is a nilpotent 1-display. Then $\mathscr{P}$ is adjoint nilpotent over $\underline{W}(R)$.
\end{lemma}
\begin{proof}
	This follows from the arguments in \cite[Rmk. 3.4.5]{BP2017}.
\end{proof}

We extend this definition to the relative Witt frame as follows. Let $B \to A$ be a PD-thickening of $p$-nilpotent $W(k_0)$-algebras. The $W(k_0)$-algebra homomorphism $W(B) \to W(A)$ induces a morphism of frames $\alpha: \underline{W}(B/A) \to \underline{W}(A)$, and base change along $\alpha$ (see (\ref{eq-changeofframe})) determines a morphism 
\begin{align}\label{eq-reduction}
G\text{-}\textup{Disp}_{\underline{W}(B/A),\mu}^\otimes \to G\text{-}\textup{Disp}_{\underline{W}(A),\mu}^\otimes
\end{align}
of \'etale stacks on $\Spec A$. If $\mathscr{P}$ is a  Tannakian $(G,\mu)$-display over $\underline{W}(B/A)$, we denote its base change to $\underline{W}(A)$ by $\alpha^\ast\mathscr{P}$ or $\mathscr{P}_{\underline{W}(A)}$.

\begin{Def}
	Let $B \to A$ be a PD thickening of $p$-nilpotent $W(k_0)$-algebras. A   Tannakian $(G,\mu)$-display $\mathscr{P}$ over $\underline{W}(B/A)$ is \textit{adjoint nilpotent} if $\mathscr{P}_{\underline{W}(A)}$ is adjoint nilpotent in the sense of Definition \ref{def-adjnilp}. 
\end{Def}

Likewise, an element $U \in G(W(B))$ is said to be adjoint nilpotent over $\underline{W}(B/A)$ if the associated  Tannakian $(G,\mu)$-display $\mathscr{P}_U$ over $\underline{W}(B/A)$ is adjoint nilpotent.

\begin{rmk}\label{rmk-adjnilp}
	If $U \in G(W(B))$, we obtain banal  Tannakian $(G,\mu)$-displays $\mathscr{P}_U$ over $\underline{W}(B)$ and $\mathscr{P}_U'$ over $\underline{W}(B/A)$ corresponding to $U$. Since $B\to A$ induces a homeomorphism $\Spec{A} \to \Spec{B}$, $\mathscr{P}_U$ is adjoint nilpotent over $\underline{W}(B)$ if and only if $\mathscr{P}_U'$ is adjoint nilpotent over $\underline{W}(B/A)$. Hence there is no ambiguity in the statement ``$U \in G(W(B))$ is adjoint nilpotent''.
\end{rmk}

We will denote by $G\text{-}\textup{Disp}_{\underline{W}(B/A),\mu}^{\otimes,\text{ad}}$, resp. $G\text{-}\textup{Disp}_{\underline{W}(A),\mu}^{\otimes,\text{ad}}$ the substack of adjoint nilpotent objects in $G\text{-}\textup{Disp}_{\underline{W}(B/A),\mu}^\otimes$, resp. $G\text{-}\textup{Disp}_{\underline{W}(A),\mu}^\otimes$. The morphism (\ref{eq-reduction}) induces a morphism
\begin{align}\label{eq-lifting}
G\text{-}\textup{Disp}_{\underline{W}(B/A),\mu}^{\otimes,\text{ad}} \to G\text{-}\textup{Disp}_{\underline{W}(A),\mu}^{\otimes,\text{ad}}.
\end{align}

\begin{prop}\label{prop-lifting}
	The morphism $($\ref{eq-lifting}$)$ is an equivalence of \'etale stacks on $\Spec{A}$.
\end{prop}
\begin{proof}
	By Theorem \ref{thm-thesis} and Theorem \ref{thm-equiv}, it is enough to show the result for the respective stacks of $G$-displays of type $\mu$. Hence the proposition follows from \cite[Rmk. 7.1.8]{Lau2018}.	
\end{proof}

\begin{rmk}
	In the case where $J = \ker(B\to A)$ is a nilpotent ideal the proposition follows from \cite[Thm. 3.5.4]{BP2017}.
\end{rmk}

\section{Crystals and $G$-displays}\label{section-crystals}
Let $R$ be a $p$-nilpotent $\zz_p$-algebra, let $G$ be a reductive group scheme over $\zz_p$, and let $\mu$ be a minuscule cocharacter for $G_{W(k_0)}$. In \textsection \ref{sub-crystalGdisp}, we construct and study the functorial properties of a $G$-crystal associated with any adjoint nilpotent Tannakian $(G,\mu)$-display over $\underline{W}(R)$. If $\Lambda$ is a finite free $\zz_p$-module, and $G = \textup{GL}(\Lambda)$, this construction recovers the crystal associated with a nilpotent Zink display as in \textsection \ref{sub-zinkcrystals}, see Lemma \ref{lem-comparecrystals}. In \textsection \ref{sub-hodge} we narrow our focus to the case where $(G,\mu)$ is a Hodge type pair (see Definition \ref{def-hodgetype}).
\subsection{The crystals associated with $G$-displays} \label{sub-crystalGdisp}

Let $G$ be a reductive group scheme over $\zz_p$, and let $\mu: \mathbb{G}_{m,W(k_0)} \to G_{W(k_0)}$ be a minuscule cocharacter of $G_{W(k_0)}$. Let $R$ be a $p$-nilpotent $W(k_0)$-algebra, and suppose $\mathscr{P}$ is an adjoint nilpotent Tannakian $(G,\mu)$-display over $\underline{W}(R)$. If $A$ is an $R$-algebra, denote by $\mathscr{P}_{\underline{W}(A)}$ the base change of $\mathscr{P}$ to $\underline{W}(A)$. 

Let $B \to A$ be a PD-thickening over $R$. By Proposition \ref{prop-lifting}, there exists a lift $\mathscr{P}_{B/A}$ of $\mathscr{P}_{\underline{W}(A)}$ to a  Tannakian $(G,\mu)$-display over $\underline{W}(B/A)$, and $\mathscr{P}_{B/A}$ is unique up to a unique isomorphism which lifts $\id_{\mathscr{P}_{\underline{W}(A)}}$. For every representation $(V,\pi)$ of $G$, write 
\begin{align}\label{eq-liftnotation}
\underline{M}_{B/A}^\pi = (M_{B/A}^\pi, F_{B/A}^\pi)
\end{align}
for the evaluation of $\mathscr{P}_{B/A}$ on $(V,\pi)$. By base change along the composition $W(B/A)^\oplus \xrightarrow{\tau} W(B) \xrightarrow{w_0} B$, we obtain a finite projective $B$-module
\begin{align*}
\mathbb{D}(\mathscr{P})_{B/A}^\pi:= (\tau^\ast M_{B/A}^\pi) \otimes_{W(B)} B.
\end{align*}
We claim that the assignment $(B\to A) \mapsto \mathbb{D}(\mathscr{P})_{B/A}^\pi$ defines a crystal of finite locally free $\mathcal{O}_{\Spec{R}/W(k_0)}$-modules for every representation $(V,\pi)$. Indeed, we need to show that if $(B \to A) \to (B' \to A')$ is a morphism of PD-thickenings, then there is an isomorphism of $B'$-modules
\begin{align}\label{eq-transition}
\mathbb{D}(\mathscr{P})^\pi_{B/A} \otimes_B B' \xrightarrow{\sim} \mathbb{D}(\mathscr{P})^\pi_{B'/A'},
\end{align}
and that these isomorphisms satisfy the cocycle condition with respect to compositions. But to obtain an isomorphism (\ref{eq-transition}) it is enough to exhibit an isomorphism
\begin{align*}
(\mathscr{P}_{B/A})_{\underline{W}(B'/A')} \xrightarrow{\sim} \mathscr{P}_{B'/A'}
\end{align*}
of Tannakian $(G,\mu)$-displays over $\underline{W}(B'/A')$. Such an isomorphism is readily found using uniqueness of lifts, since both $(\mathscr{P}_{B/A})_{\underline{W}(B'/A')}$ and $\mathscr{P}_{B'/A'}$ lift $\mathscr{P}_{\underline{W}(A')}$.  
It is straightforward to check that compositions of the transition isomorphisms obtained in this way satisfy the cocycle condition, so by Remark \ref{rmk-crystals} we obtain a crystal of finite locally free $\mathcal{O}_{\Spec{R}/W(k_0)}$-modules $\mathbb{D}(\mathscr{P})^\pi$ for every $(V,\pi)$. 

\begin{lemma}\label{lem-crystal}
	The association
	\begin{align*}
	\mathbb{D}(\mathscr{P}): \textup{Rep}_{\zz_p}(G) \to \textup{LFCrys}({R}/W(k_0)), \ (V,\pi) \mapsto \mathbb{D}(\mathscr{P})^\pi,
	\end{align*}
	defines an exact tensor functor. 
\end{lemma}
\begin{proof}
	A $G$-equivariant morphism $(V_1,\pi_1) \to (V_2,\pi_2)$ induces a morphism of finite projective graded $W(B/A)^\oplus$-modules $M_{B/A}^{\pi_1} \to M_{B/A}^{\pi_2}$, and by base change to $B'$ we obtain $\mathbb{D}(\mathscr{P})^{\pi_1} \to \mathbb{D}(\mathscr{P})^{\pi_2}$. If $(B'\to A') \to (B \to A)$ is a PD-morphism, then the transition map $\mathbb{D}(\mathscr{P})^\pi_{B/A} \otimes_B B'\xrightarrow{\sim} \mathbb{D}(\mathscr{P})^\pi_{B'/A'}$ is induced from the natural transformation of functors $(\mathscr{P}_{B/A})_{\underline{W}(B'/A')} \xrightarrow{\sim} \mathscr{P}_{B'/A'}$, which is compatible with the induced morphisms of representations. It follows that $\mathbb{D}(\mathscr{P})^{\pi_1} \to \mathbb{D}(\mathscr{P})^{\pi_2}$ is a morphism of crystals. Compatibility with tensor products follows from the definition of $\mathbb{D}(\mathscr{P})$ and the compatibility of $\mathscr{P}$ with tensor products. Exactness follows similarly, using that all modules are projective and hence all exact sequences in question split.
\end{proof}
\begin{Def}
	If $\mathscr{P}$ is an adjoint nilpotent  Tannakian $(G,\mu)$-display over $\underline{W}(R)$ for some $p$-nilpotent $W(k_0)$-algebra $R$, then the functor $\mathbb{D}(\mathscr{P})$ defined in Lemma \ref{lem-crystal} is the \textit{$G$-crystal associated with $\mathscr{P}$}. 
\end{Def}

\begin{lemma}\label{lem-crystalfunctorial}
	The assignment $\mathscr{P} \mapsto \mathbb{D}(\mathscr{P})$ is functorial in $\mathscr{P}$ and compatible with base change.
\end{lemma}
\begin{proof}
	Suppose $\psi: \mathscr{P} \to \mathscr{P}'$ is a morphism of Tannakian $(G,\mu)$-displays. If $B\to A$ is a PD-thickening over $R$, denote by $\mathscr{P}_{B/A}$ the lift of $\mathscr{P}_{\underline{W}(A)}$ to $\underline{W}(B/A)$, and by $\mathscr{P}'_{B/A}$ the lift of $\mathscr{P}'_{\underline{W}(A)}$. By Theorem \ref{prop-lifting}, $\psi_{\underline{W}(A)}$ lifts uniquely to a morphism of Tannakian $(G,\mu)$-displays $\mathscr{P}_{B/A} \to \mathscr{P}'_{B/A}$ over $\underline{W}(B/A)$. In particular, for every $(V,\pi)$ we have a morphism
	\begin{align*}
	\psi_{B/A}^\pi: M_{B/A}^\pi \to (M_{B/A}')^\pi,
	\end{align*}
	where here we use notation as in (\ref{eq-liftnotation}). Tensoring this along $W(B/A)^\oplus \xrightarrow{\tau} W(B) \to B$ gives us a morphism
	\begin{align*}
	\mathbb{D}(\psi)^\pi_{B/A} : \mathbb{D}(\mathscr{P})^\pi_{B/A} \to \mathbb{D}(\mathscr{P}')^\pi_{B/A}
	\end{align*}
	for every $B\to A$ and every $(V,\pi)$. That this determines a morphism of crystals
	$\mathbb{D}(\psi)^\pi: \mathbb{D}(\mathscr{P})^\pi \to \mathbb{D}(\mathscr{P}')^\pi$ follows from the definition of the transition morphisms and Proposition \ref{prop-lifting}.
	Moreover, that the resulting morphism $\mathbb{D}(\psi): \mathbb{D}(\mathscr{P}) \to \mathbb{D}(\mathscr{P}')$ is a natural transformation and is compatible with tensor products both follow from the corresponding properties of the morphism $\psi_{B/A}$.
	
	If $\alpha: R \to R'$ is a $W(k_0)$-algebra homomorphism, write $\alpha^\ast \mathbb{D}(\mathscr{P})$ for the base change of $\mathbb{D}(\mathscr{P})$ to $R'$. Explicitly, for any PD-thickening $B \to A$ over $R'$ and representation $(V,\pi)$,
	\begin{align*}
	\alpha^\ast \mathbb{D}(\mathscr{P})^\pi_{B/A} = \mathbb{D}(\mathscr{P})^\pi_{\alpha_!(B/A)},
	\end{align*}
	where we write $\alpha_!(B/A)$ for the PD-thickening $B \to A$ over $R$ given by viewing $A$ as an $R$-algebra via restriction of scalars. Compatibility with base change follows, since by definition $\mathbb{D}(\mathscr{P}_{\underline{W}(R')})_{B/A}^\pi$ is also given by $\mathbb{D}(\mathscr{P})^\pi_{\alpha_!(B/A)}$.
\end{proof}

\begin{rmk}\label{lem-banalcrystal}
	Suppose that $\mathscr{P}$ is a banal Tannakian $(G,\mu)$-display over $\underline{W}(R)$ (see Definition \ref{def-banal}), so there exists an isomorphism $\psi: \mathscr{P}_U \xrightarrow{\sim} \mathscr{P}$ for $U \in L^+G(R)$ by Proposition \ref{prop-banal}. Fix a PD-thickening $B \to A$ over $R$, and denote by $U_A$ the image of $U$ under $G(W(R)) \to G(W(A))$. Any choice of lift $U_B$ of $U_A$ to $G(W(B))$ determines a Tannakian $(G,\mu)$-display $\mathscr{P}_{U_B}$ over $\underline{W}(B/A)$ which lifts $\mathscr{P}_U$. Hence by Proposition \ref{prop-lifting}, there exists a unique isomorphism $\psi_{U_B}: \mathscr{P}_{U_B} \xrightarrow{\sim} \mathscr{P}_{B/A}$ lifting $\psi_A$, where $\mathscr{P}_{B/A}$ is the unique lift of $\mathscr{P}_{\underline{W}(A)}$. From the definitions of $\mathscr{P}_{U_B}$ and of $\mathbb{D}(\mathscr{P})_{B/A}$, we obtain from $\psi_{U_B}$ an isomorphism of tensor functors $\overline{\psi}_{U_B}: \omega_B \xrightarrow{\sim} \mathbb{D}(\mathscr{P})_{B/A}$, where $\omega_B$ is the usual fiber functor (\ref{eq-usualfiberfunctor}).	
\end{rmk}

Suppose $\Lambda$ is a finite free $\zz_p$-module, $G = \textup{GL}(\Lambda)$ and $\mu$ is a minuscule cocharacter for $G$ whose weights are contained in $\{0,1\}$. Then by Remark \ref{rmk-GLn}, the category of  $(\text{GL}(\Lambda), \mu)$-displays over a $p$-adic $W(k_0)$-algebra $R$ is equivalent to the category of Zink displays of type $\mu$ over $R$. In \textsection \ref{sub-zinkcrystals} we recalled the definition of the crystal $\mathbb{D}(\underline{P})$ associated with a nilpotent Zink display $\underline{P}$.

Denote by $\Z_R$ the functor which gives the equivalence between  $(\text{GL}(\Lambda), \mu)$-displays over $\underline{W}(R)$ and Zink displays of type $\mu$ over $R$. By Lemma \ref{lem-nilpotent}, if $\mathscr{P}$ is a  $(\textup{GL}(\Lambda), \mu)$-display over $\underline{W}(R)$ such that $\Z_R(\mathscr{P})$ is nilpotent, then $\mathscr{P}$ is adjoint nilpotent. The following lemma describes the relationship between the $G$-crystal associated with $\mathscr{P}$ and the crystal associated with $\Z_R(\mathscr{P})$.

\begin{lemma}\label{lem-comparecrystals}
	Let $\mathscr{P}$ be a  $(\textup{GL}(\Lambda), \mu)$-display over $\underline{W}(R)$ such that the associated Zink display $\Z_R(\mathscr{P})$ is nilpotent, and denote by $(\Lambda, \iota)$ the standard representation of $\textup{GL}(\Lambda)$. Then there is a natural isomorphism of crystals
	\begin{align*}
	\mathbb{D}(\mathscr{P})^{\iota} \cong \mathbb{D}(\Z_R(\mathscr{P})).
	\end{align*}
\end{lemma}
\begin{proof}
	Let $B \to A$ be a PD-thickening over $R$, and let $\mathscr{P}_{B/A}$ be the unique lift of $\mathscr{P}_{\underline{W}(A)}$ to a  $(\textup{GL}(\Lambda),\mu)$-display over $\underline{W}(B/A)$. Then $\mathscr{P}_{B/A}(\Lambda,\iota)$ corresponds to a window over $\mathcal{W}(B/A)$ which lifts the Zink display corresponding to $\mathscr{P}_{\underline{W}(A)}(\Lambda,\iota)$. But $\underline{\tilde{P}}$ is the unique window over $\mathcal{W}(B/A)$ with this property, so it is isomorphic to the window associated with $\mathscr{P}_{B/A}(\Lambda,\iota)$. In particular, we obtain an isomorphism $\tau^\ast M^\iota_{B/A} \cong \tilde{P}_0$. The result follows.
\end{proof}
\subsection{$G$-displays of Hodge type}\label{sub-hodge}
Let us continue to assume that $G$ is a reductive group scheme over $\zz_p$ and that $\mu:\mathbb{G}_{m,W(k_0)} \to G_{W(k_0)}$ is a minuscule cocharacter for $G_{W(k_0)}$. 
\begin{Def}\label{def-hodgetype}
	We say the pair $(G,\mu)$ is \textit{of Hodge type} if there exists a closed embedding of $\zz_p$-group schemes $\eta: G \hookrightarrow \textup{GL}(\Lambda)$ for a finite free $\zz_p$-module $\Lambda$, such that after a choice of basis $\Lambda_{W(k_0)} \xrightarrow{\sim} W(k_0)^h$, the composition $\eta \circ \mu$ is the minuscule cocharacter $a \mapsto \textup{diag}(1^{(d)}, a^{(h-d)})$ of $\textup{GL}_h$ for some $d$. In this case, the representation $(\Lambda, \eta)$ is called a \textit{Hodge embedding} for $(G,\mu)$.
\end{Def}

If $(G,\mu)$ is of Hodge type, and $\mathscr{P}$ is a  Tannakian $(G,\mu)$-display over $\underline{W}(R)$, then $\mathscr{P}(\Lambda, \eta)$ is a 1-display over $\underline{W}(R)$. Let $\textup{Z}_{\eta,R}(\mathscr{P})$ denote the Zink display associated with this 1-display via Lemma \ref{lem-windows}. If the ring $R$ is clear from context, we will write simply $\Z_\eta(\mathscr{P})$.

\begin{Def}
	We say a Tannakian $(G,\mu)$-display $\mathscr{P}$ over $\underline{W}(R)$ is \textit{nilpotent with respect to $\eta$} if $\Z_\eta(\mathscr{P})$ is a nilpotent Zink display.
\end{Def}

This condition is local for the fpqc topology, and we denote by $G\textup{-}\textup{Disp}_{\underline{W},\mu}^{\otimes,\eta}$ the stack of Tannakian $(G,\mu)$-displays which are nilpotent with respect to $\eta$. 

\begin{lemma}\label{lem-wrteta}
	Suppose $(G,\mu)$ is of Hodge type, and let $\mathscr{P}$ be a Tannakian $(G,\mu)$-display over $R$. If $\mathscr{P}$ is nilpotent with respect to $\eta$, then it is adjoint nilpotent.
\end{lemma}
\begin{proof}
	Notice $(\eta(\mathscr{P}))(\Lambda,\iota) = \mathscr{P}(\Lambda,\eta)$, so since $\Z_\eta(\mathscr{P})$ is nilpotent, it follows from Lemma \ref{lem-nilpotent} that $\eta(\mathscr{P})$ is an adjoint nilpotent $(\textup{GL}(\Lambda), \eta\circ\mu)$-display over $R$. Then $\mathscr{P}$ is an adjoint nilpotent Tannakian $(G,\mu)$-display over $\underline{W}(R)$ (cf. \cite[3.7.1]{BP2017}).
\end{proof}

In the remainder of this section, we assume $(G,\mu)$ is of Hodge type with Hodge embedding $(\Lambda, \eta)$. Let $\mathscr{P}$ be a Tannakian $(G,\mu)$-display over $\underline{W}(R)$ which is nilpotent with respect to $\eta$, so in particular $\mathscr{P}$ is adjoint nilpotent by Lemma \ref{lem-wrteta}, and we can associate a $G$-crystal $\mathbb{D}(\mathscr{P})$ to $\mathscr{P}$ as in the previous section. It is easy to see $\mathbb{D}(\mathscr{P})^\eta \cong \mathbb{D}(\eta_\ast \mathscr{P})^\iota$, so by Lemma \ref{lem-comparecrystals} we have a canonical isomorphism
\begin{align}\label{eq-eta}
\mathbb{D}(\mathscr{P})^\eta \cong \mathbb{D}(\Z_\eta(\mathscr{P})).
\end{align}
As a result we can endow $\mathbb{D}(\mathscr{P})^\eta$ with the structure of a Dieudonn\'e crystal using the Dieudonn\'e crystal structure on $\mathbb{D}(\Z_\eta(\mathscr{P}))$ as in Section \ref{sub-zinkcrystals}. Denote by 
\begin{align*}
\mathbb{F}: \phi^\ast\mathbb{D}(\mathscr{P})^\eta \to \mathbb{D}(\mathscr{P})^\eta \text{ and } \mathbb{V}:\mathbb{D}(\mathscr{P})^\eta \to \phi^\ast \mathbb{D}(\mathscr{P})^\eta
\end{align*}
the Frobenius and Verschiebung for $\mathbb{D}(\mathscr{P})^\eta$.

Suppose $pR=0$ and that $\mathscr{P}$ is banal, so there exists an isomorphism $\psi: \mathscr{P}_U \xrightarrow{\sim} \mathscr{P}$ for $U \in L^+G(R)$ by Proposition \ref{prop-banal}. As in Remark \ref{lem-banalcrystal}, for any PD-thickening $B \to A$ over $R$ and choice of lift $U_B$ to $G(W(B))$ of the image $U_A$ of $U$ in $G(W(A))$, we obtain an isomorphism of tensor functors $\overbar{\psi}_{U_B}: \omega_B \xrightarrow{\sim} \mathbb{D}(\mathscr{P})_{B/A}$. In particular, by evaluating $\overbar{\psi}_{U_B}$ on $(\Lambda, \eta)$, we have an isomorphism
\begin{align}\label{eq-etacrystal}
	\Lambda \otimes_{\zz_p} B \xrightarrow{\sim} \mathbb{D}(\mathscr{P})^\eta_{B/A}.
\end{align}

Let $\mathscr{P}^{(p)}$ denote the base change of $\mathscr{P}$ along $\phi: R \to R$. From the trivialization $\psi$ we obtain an isomorphism $\psi^{(p)}: (\mathscr{P}_U)^{(p)} \xrightarrow{\sim} \mathscr{P}^{(p)}$. Moreover, if $f(U)$ is the image of $U$ in $L^+G(R)$ under the Witt vector Frobenius $f$, then we have an isomorphism $\varepsilon: \mathscr{P}_{f(U)} \xrightarrow{\sim} (\mathscr{P}_U)^{(p)}$ given by
\begin{align}\label{eq-epsilon}
	V \otimes_{\zz_p} W(R)^\oplus \xrightarrow{\sim} V \otimes_{\zz_p} W(R)^\oplus \otimes_{W(R)^\oplus, W(\phi)^\oplus} W(R)^\oplus, \ x \otimes \xi \mapsto x \otimes 1 \otimes \xi
\end{align}
for every representation $(V,\pi)$. Hence $\mathscr{P}^{(p)}$ is banal, with trivialization $\psi^{(p)} \circ \varepsilon$. By functoriality, $f(U_B)$ is a lift of $f(U)$, so by Remark \ref{lem-banalcrystal}, we obtain an isomorphism $\overbar{\psi}^{(p)}_{f(U_B)}: \omega_B \xrightarrow{\sim} \mathbb{D}(\mathscr{P})_{B/A}$, which evaluates on $(\Lambda,\eta)$ to give a trivialization of $\mathbb{D}(\mathscr{P}^{(p)})^\eta_{B/A}$:
\begin{align}\label{eq-crystalp}
	\Lambda \otimes_{\zz_p} B \xrightarrow{\sim} \mathbb{D}(\mathscr{P}^{(p)})_{B/A}^\eta.
\end{align}
Moreover, there is a natural identification 
\begin{align}\label{eq-phistar}
	\mathbb{D}(\mathscr{P}^{(p)}) \xrightarrow{\sim} \phi^\ast\mathbb{D}(\mathscr{P}).
\end{align}
Indeed, this follows essentially from the definitions: since $\phi^\ast\mathbb{D}(\mathscr{P})_{B/A} = \mathbb{D}(\mathscr{P})_{\phi_!(B/A)}$, (see (\ref{eq-shriek})), in order to evaluate $\phi^\ast\mathbb{D}(\mathscr{P})$ on $B \to A$, we first base change $\mathscr{P}$ to $\underline{W}(A)$ along $R \xrightarrow{\phi} R \to A$. But this is exactly how we evaluate $\mathbb{D}(\mathscr{P}^{(p)})$ on $B \to A$. Thus combining (\ref{eq-crystalp}) with the identification (\ref{eq-phistar}), we obtain a trivialization of $\phi^\ast\mathbb{D}(\mathscr{P})^\eta_{B/A}$:
\begin{align}\label{eq-crystalptriv}
	\Lambda \otimes_{\zz_p} B \xrightarrow{\sim} \phi^\ast \mathbb{D}(\mathscr{P})^\eta_{B/A}.
\end{align}
Write $\bar{U}_{B}$ for the image of $U_B$ in $G(B)$ under $w_0: W(B) \to B$. 

\begin{lemma}\label{lem-explicitfrob}
	Let $B \to A$ be a PD-thickening over $R$. Then with respect to the trivializations $($\ref{eq-etacrystal}$)$ and $($\ref{eq-crystalptriv}$)$, the Frobenius and Verschiebung for $\mathbb{D}(\mathscr{P})^\eta_{B/A}$ are given by 
	\begin{align*}
	\eta(\bar{U}_{B}) \circ (\id_{\Lambda^0_{B}} \oplus p\cdot \id_{\Lambda^1_{B}}) \text{ and } (p\cdot \id_{\Lambda^0_{B}} \oplus \id_{\Lambda^1_{B}}) \circ \eta(\bar{U}_{B})^{-1},
	\end{align*}
	respectively. 
\end{lemma}
\begin{proof}	
	If $\mathscr{P}$ is a Tannakian $(G,\mu)$-display over $\underline{W}(R)$ (resp. over  $\underline{W}(B/A)$ for some PD-thickening $B \to A$) which is nilpotent with respect to $\eta$, then we will denote by $\Z_\eta(\mathscr{P})$ the associated Zink display (resp. window over $\mathcal{W}(B/A)$). If $\underline{P} = \Z_\eta(\mathscr{P})$ is the Zink display associated with $\mathscr{P}$, then by compatibility of $\Z_\eta$ with base change we have $\Z_\eta(\mathscr{P}^{(p)}) \cong \underline{P}^{(p)}$. For any Zink display $\underline{P}$ over $R$ and any PD-thickening $B \to A$, denote by $\underline{P}_{B/A}$ the unique lift of of $\underline{P}_A$ to a window over $\mathcal{W}(B/A)$. Recall from the proof of Lemma \ref{lem-comparecrystals} that if $\underline{P} = \Z_\eta(\mathscr{P})$, then $\underline{P}_{B/A} = \mathscr{P}_{B/A}(\Lambda, \eta)$, where $\mathscr{P}_{B/A}$ is the unique lift of $\mathscr{P}_{\underline{W}(A)}$ to $\underline{W}(B/A)$. 
	
	Let $\mathscr{P}$ be a banal Tannakian $(G,\mu)$-display over $R$ with trivialization isomorphism $\psi: \mathscr{P}_U \xrightarrow{\sim} \mathscr{P}$. By replacing $\mathscr{P}$ by $\mathscr{P}_{\underline{W}(A)}$, we may assume $A = R$. Let us start by proving the lemma for the Frobenius. We want an explicit description of the map
	\begin{align}\label{eq-lemmafrob}
	\Lambda \otimes_{\zz_p} B \xrightarrow{\sim} \phi^\ast\mathbb{D}(\mathscr{P})^\eta_{B/R} \xrightarrow{\mathbb{F}_{B/R}} \mathbb{D}(\mathscr{P})^\eta_{B/R} \xrightarrow{\sim} \Lambda \otimes_{\zz_p} B,
	\end{align}
	where the first arrow is (\ref{eq-crystalptriv}) and the last is (\ref{eq-etacrystal}). The unique lift $\psi_{U_B}: \mathscr{P}_{U_B} \xrightarrow{\sim} \mathscr{P}_{B/R}$ of $\psi$ induces an isomorphism of $\mathcal{W}(B/R)$ windows
	\begin{align}\label{eq-Zetapsi} 
		\Z_\eta(\mathscr{P}_U)_{B/R} \xrightarrow{\sim} \Z_\eta(\mathscr{P})_{B/R}.
	\end{align}
	Similarly, the trivialization $\psi^{(p)}\circ \varepsilon$ (see (\ref{eq-epsilon})) induces an isomorphism \begin{align}\label{eq-Zetapsip}
		\Z_\eta(\mathscr{P}_{f(U)})_{B/R} \xrightarrow{\sim} \Z_\eta(\mathscr{P}^{(p)})_{B/R}.
	\end{align}
	Denote by $\tilde{F}_0^\sharp$ the unique lift of the display Verschiebung $\textup{Ver}_{\Z_\eta(\mathscr{P})}:\Z_\eta(\mathscr{P}^{(p)}) \to \Z_\eta(\mathscr{P})$ to a morphism of $\mathcal{W}(B/R)$-windows. Then (\ref{eq-lemmafrob}) is the reduction modulo $I_B$ of the following composition:
	\begin{align}\label{eq-frobcomp1}
	\Z_\eta(\mathscr{P}_{f(U)})_{B/R} \xrightarrow{\sim}  \Z_\eta(\mathscr{P}^{(p)})_{B/R} \xrightarrow{\tilde{F}_0^\sharp} \Z_\eta(\mathscr{P})_{B/R} \xleftarrow{\sim} \Z_\eta(\mathscr{P}_U)_{B/R},
	\end{align}
	where the first arrow is (\ref{eq-Zetapsip}), and the last is the inverse of (\ref{eq-Zetapsi}).	In turn, (\ref{eq-frobcomp1}) is the unique lift of the composition 
	\begin{align}\label{eq-frobcomp2}
	\Z_\eta(\mathscr{P}_{f(U)}) \xrightarrow{\Z_\eta(\varepsilon)} \Z_\eta(\mathscr{P}_U)^{(p)} \xrightarrow{\Z_\eta(\psi)^{(p)}} \Z_\eta(\mathscr{P})^{(p)} \xrightarrow{\textup{Ver}_{\Z_\eta(\mathscr{P})}} \Z_\eta(\mathscr{P}) \xrightarrow{\Z_\eta(\psi^{-1})} \Z_\eta(\mathscr{P}_U).
	\end{align}
	Hence to prove the lemma for the Frobenius it is enough to show (\ref{eq-frobcomp2}) is given by $\eta(U) \circ (\id \oplus p \cdot \id)$. By functoriality of $\textup{Ver}$, we can rewrite (\ref{eq-frobcomp2}) as the composition $\textup{Ver}_{\Z_\eta(\mathscr{P}_U)} \circ \Z_\eta(\varepsilon)$, and we have an explicit description of $\textup{Ver}_{\Z_\eta(\mathscr{P}_U)}$ (see (\ref{eq-F0sharp})):
	\begin{align*}
	\textup{Ver}_{\Z_\eta(\mathscr{P}_U)} = ((\eta(U)\circ (\id_{\Lambda} \otimes f))^\sharp \circ (\id_{f^\ast \Lambda^0_{W(R)}} \oplus p \cdot \id_{f^\ast \Lambda^1_{W(R)}}).
	\end{align*}
	The result for the Frobenius follows because $(\eta(U)\circ (\id_\Lambda \otimes f))^\sharp = \eta(U) \circ Z(\varepsilon)^{-1}$, and $\Z_\eta(\varepsilon)^{-1}$ commutes with $\id \oplus p\cdot \id$. The computation is nearly identical for the Verschiebung, using the explicit description (\ref{eq-Vsharp}) of the Frobenius for $\Z_\eta(\mathscr{P}_U)$.
\end{proof}

For any finite free $\zz_p$-module $\Lambda$, let $\Lambda^\otimes = \bigoplus_{m,n} \Lambda^{\otimes m} \otimes_{\zz_p} (\Lambda^\vee)^{\otimes n}$ denote the total tensor algebra of $\Lambda \oplus \Lambda^\vee$. For any element $s \in \Lambda^\otimes$ and $\zz_p$-algebra $R$, write $(s \otimes 1)_R$ for the map $R \to \Lambda^\otimes \otimes_{\zz_p} R$ given by $1 \mapsto s \otimes 1$. If the pair $(G,\mu)$ is of Hodge type, then by \cite[Prop. 1.3.2]{Kisin2010} and \cite{Deligne2011}, there exists a finite collection of tensors $\underline{s} = (s_1, \dots, s_r)$ with $s_i \in \Lambda^\otimes$ such that, for all $\zz_p$-algebras $R$, 
\begin{align*}
G(R) = \{ g \in \textup{GL}(\Lambda\otimes_{\zz_p} R) \mid g (s_i \otimes 1)_R = (s_i \otimes 1)_R \text{ for all }i\}.
\end{align*}
We say the collection of tensors $\underline{s}$ defines the group $G$ inside $\textup{GL}(\Lambda)$. Without loss of generality, we may assume that, for each $i$, we have  
\begin{align*}
s_i \in \Lambda^{\otimes m_i} \otimes (\Lambda^\vee)^{\otimes n_i}
\end{align*}
for some $m_i$ and $n_i$. Let $\Lambda(i) = \Lambda^{\otimes m_i} \otimes (\Lambda^\vee)^{\otimes n_i}$. This is a $G$-stable submodule of $\Lambda^\otimes$, and we will denote by $(\Lambda(i), \eta(i))$ the corresponding representation. For every $i$, $s_i$ defines a morphism of representations
\begin{align}\label{eq-si}
s_i: \zz_p \to \Lambda(i), \ 1 \mapsto s_i,
\end{align}
where $\zz_p$ denotes the trivial representation. Each $\Lambda(i)$ is canonically graded by the action of the cocharacter $\mu$, and since $s_i$ is $G$-invariant, we see $s_i \in (\Lambda(i))^0$. 

\begin{Def}\label{def-lochodge}
	A \textit{local Hodge embedding datum} is a tuple $\underline{G} = (G,\mu,\Lambda, \eta, \underline{s})$, where
	\begin{itemize}
		\item $(G,\mu)$ is a pair consisting of a reductive $\zz_p$-group scheme and a minuscule cocharacter $\mu$ of $G_{W(k_0)}$ such that $(G,\mu)$ is of Hodge type,
		\item $\eta: G \hookrightarrow \textup{GL}(\Lambda)$ is a Hodge embedding for $(G,\mu)$, and
		\item $\underline{s} = (s_1, \dots, s_r)$ is a collection of tensors which define $G$ inside of $\textup{GL}(\Lambda)$.
	\end{itemize}
\end{Def}

If $\mathscr{P}$ is an adjoint nilpotent Tannakian $(G,\mu)$-display over $R$, we may apply $\mathbb{D}(\mathscr{P})$ to $s_i$ to obtain a morphism of crystals
\begin{align*}
t_i: \mathbbm{1} \to \mathbb{D}(\mathscr{P})^{\eta(i)}.
\end{align*}
Notice that $\phi^\ast \mathbbm{1}$ is canonically identified with $\mathbbm{1}$, so we likewise obtain a morphism $t_i: \mathbbm{1} \to \phi^\ast \mathbb{D}(\mathscr{P})^{\eta(i)}$.

If $\mathscr{P}$ is nilpotent with respect to $\eta$, we have an identification $\mathbb{D}(\mathscr{P})^\eta \cong \mathbb{D}(\underline{P})$, where $\underline{P} = \Z_\eta(\mathscr{P})$. Since $\mathbb{D}(\mathscr{P})$ is compatible with tensor products, we see
\begin{align*}
\mathbb{D}(\mathscr{P})^{\eta(i)} = \mathbb{D}(\underline{P})^{\otimes m_i} \otimes (\mathbb{D}(\underline{P})^\vee)^{\otimes n_i}.
\end{align*}
The Frobenius $\mathbb{F}$ on $\mathbb{D}(\Z_\eta(\mathscr{P}))$ (see (\ref{eq-frobdef})) extends to tensor products, and it extends to (linear) duals after we pass to the associated isocrystal. By the relation $\mathbb{FV} = p$, we see that the resulting extension of $\mathbb{F}$ to $\mathbb{D}(\underline{P})^\vee[1/p]$ is given by $\mathbb{F} = p^{-1} \mathbb{V}^t$. Hence $\mathbb{F}$ extends to a morphism of isocrystals
\begin{align}\label{eq-tensorfrob}
\mathbb{F}_i: \phi^\ast\mathbb{D}(\mathscr{P})^{\eta(i)}[1/p] \to \mathbb{D}(\mathscr{P})^{\eta(i)}[1/p]
\end{align}

\begin{prop}\label{prop-frobeq}
	For each $i$, $t_i$ is Frobenius equivariant, i.e., $\mathbb{F}_i \circ t_i = t_i$. 
\end{prop}
\begin{proof}
	By the equivalence (\ref{eq-crystalequiv}) between $\textup{Isoc}(R)$ and $\textup{Isoc}(R/pR)$, we may assume $pR = 0$. Moreover, by Lemma \ref{lem-faithful}, it is enough to show the result after applying the functor $(-)_{D^\wedge/R}[1/p]: \textup{Isoc}(R) \to \textup{Mod}(D^\wedge[1/p])$, so it is enough to show $\mathbb{F}_i$ fixes 
	\begin{align*}
		(t_i)_{D^\wedge/R} : D^\wedge \to \mathbb{D}(\mathscr{P})^{\eta(i)}_{D^\wedge/R}[1/p].
	\end{align*}
	Notice $\mathbb{F}_i$ is given by $p^{-n_i}\mathbb{F}_i'$, where $\mathbb{F}_i' = \mathbb{F}^{\otimes m_i} \otimes (\mathbb{V}^t)^{\otimes n_i}$ is a morphism of crystals $\phi^\ast \mathbb{D}(\mathscr{P})^{\eta(i)} \to \mathbb{D}(\mathscr{P})^{\eta(i)}$. As in \textsection \ref{sub-reviewcrystals}, let write $D_n = D^\wedge / p^n D^\wedge$. If we denote by $\mathbb{F}'_{i,n}$ the evaluation of $\mathbb{F}_i'$ on $\phi^\ast\mathbb{D}(\mathscr{P})^{\eta(i)}_{D_n/R}$, then 
	\begin{align*}
	\mathbb{F}_i (t_i) = p^{-n_i}\cdot(\mathbb{F}'_{i,n}(t_i))_{k\in \zz_{>0}} \in \left( \varprojlim \mathbb{D}(\mathscr{P})^{\eta(i)}_{D_n/R}\right) [1/p].
	\end{align*}
	Hence we see it is enough to show $\mathbb{F}'_{i,n}(t_i) = p^{n_i}t_i$ for all $n \in \mathbb{Z}_{>0}$. 
	
	By Lemma \ref{lem-etalelocal} we may replace $R$ by an \'etale faithfully flat extension, and since every Tannakian $(G,\mu)$-display is \'etale locally banal (see the proof of \cite[Lem. 5.4.2]{Lau2018}), we may assume $\mathscr{P}$ is banal. Let $\psi: \mathscr{P}_U \xrightarrow{\sim} \mathscr{P}$ be a trivialization of $\mathscr{P}$ for some $U \in L^+G(R)$. We obtain also a trivialization of $\mathscr{P}^{(p)}$ by $\psi^{(p)} \circ \varepsilon: \mathscr{P}_{f(U)} \xrightarrow{\sim} \mathscr{P}^{(p)}$, where $\varepsilon$ is defined in (\ref{eq-epsilon}). If we choose a lift $U_{D_n}$ of $U$ to $G(W(D_n))$, then these trivializations induce isomorphisms
	\begin{align}\label{eq-trivDk}
		\Lambda \otimes_{\zz_p} D_n \xrightarrow{\sim} \mathbb{D}(\mathscr{P})^\eta_{D_n/R} \text{  and  } \Lambda \otimes_{\zz_p} D_n \xrightarrow{\sim} \phi^\ast \mathbb{D}(\mathscr{P})^\eta_{D_n/R}
	\end{align}
	as in (\ref{eq-etacrystal}) and (\ref{eq-crystalptriv}). Under these identifications, the morphisms $t_i$ correspond to 
	\begin{align*}
		D_n \to \Lambda(i) \otimes_{\zz_p} D_n, \ 1 \mapsto s_i \otimes 1,
	\end{align*}
	since the isomorphisms (\ref{eq-trivDk}) are induced by isomorphisms of tensor functors.
	
	Denote by $\mathbb{F}_{D_n}$ and $\mathbb{V}_{D_n}$ the evaluations of $\mathbb{F}$ and $\mathbb{V}$ respectively on $D_n \to R$. By Lemma \ref{lem-explicitfrob}, with respect to the trivializations (\ref{eq-trivDk}) we have
	\begin{align}\label{eq-explicit}
		\mathbb{F}_{D_n/R} = \eta(\bar{U}_{D_n}) \circ (\id_{\Lambda_{D_n}^0} \oplus p \cdot \id_{\Lambda_{D_n}^1}) \text{ and }\mathbb{V}_{D_n/R} = (p \cdot \id_{\Lambda_{D_n}^0} \oplus \id_{\Lambda_{D_n}^1})\circ \eta(\bar{U}_{D_n})^{-1},
	\end{align}
	where $\bar{U}_{D_n}$ is the image of $U_{D_n}$ under $w_0: W(D_n) \to D_n$. 
	
	Since $\Lambda_{W(k_0)}$ decomposes as $\Lambda_{W(k_0)} = \Lambda^0 \oplus \Lambda^1$, we see that $\Lambda(i)\otimes_{\zz_p} D_n$ can be written as a direct sum of terms of the form
	\begin{align*}
	(\Lambda^0_{D_n})^{\otimes j} \otimes_{D_n} (\Lambda^1_{D_n})^{\otimes m_i - j} \otimes_{D_n} ((\Lambda^0_{D_n})^\vee)^{\otimes k} \otimes_{D_n} ((\Lambda^1_{D_n})^\vee)^{\otimes n_i - k}.
	\end{align*}
	Moreover, since $s_i \in (\Lambda(i))^0$, each $s_i \otimes 1$ is contained in a direct sum of terms which satisfy $m_i-j = n_i - k$. By (\ref{eq-explicit}), $\mathbb{F}'_{i,n}$ acts on such a term by 
	\begin{align}\label{eq-etai}
	\eta(\bar{U}_{D_n})^{\otimes j} \otimes p^{m_i - j} \eta(\bar{U}_{D_n})^{\otimes m_i - j} \otimes p^k \eta^\vee(\bar{U}_{D_n})^{\otimes k} \otimes \eta^\vee(\bar{U}_{D_n})^{\otimes n_i - k},
	\end{align}
	where $\eta^\vee$ denotes the contragradient representation. Since $m_i-j = n_i - k$, (\ref{eq-etai}) is equal to $p^{n_i}\cdot\eta(i)(\bar{U}_{D_n/R})$, so
	\begin{align*}
	\mathbb{F}_{i,n}'(s_i \otimes 1) = p^{n_i} \cdot \eta(i)(\bar{U}_{D_n/R})(s_i \otimes 1) = p^{n_i} \cdot (s_i \otimes 1),
	\end{align*}
	with the last equality following because $\eta(i)(\bar{U}_{D_n/R})$ fixes $s_i\otimes 1$ for every $n$. Thus $\mathbb{F}_i\circ t_i = t_i$.
\end{proof}
\section{$G$-displays and formal $p$-divisible groups}\label{section-main}
Let $G$ be a reductive $\zz_p$-group scheme and let $\mu$ be a minuscule cocharacter for $G_{W(k_0)}$. Moreover, assume that the pair $(G,\mu)$ is of Hodge type, and that $\underline{G} = (G,\mu,\Lambda, \eta, \underline{s})$ is a local Hodge embedding datum. In \textsection \ref{sub-tatetensors} we define a notion of $p$-divisible groups with $(\underline{s},\mu)$-structure (Definition \ref{def-smu}) and prove that these objects form an \'etale stack on $\textup{Nilp}_{W(k_0)}$ (Lemma \ref{lem-smudescent}). In \textsection \ref{sub-functor}, we define a functor from Tannakian $(G,\mu)$-displays which are nilpotent with respect to $\eta$ to formal $p$-divisible groups with $(\underline{s},\mu)$-structure over $R$ in $\textup{Nilp}_{W(k_0)}$ (Lemma \ref{lem-tensors}). In \textsection \ref{sub-mainthm} we prove the functor is an equivalence if $R/pR$ has a $p$-basis \'etale locally, (Theorem \ref{mainthm}). In sections \ref{sub-rzspaces} and \ref{sub-deformations} we establish corollaries of the main theorem. In particular, in \textsection \ref{sub-rzspaces}, using Theorem \ref{mainthm}, we prove that the RZ-functors of Hodge type defined in \cite{Kim2018} and in \cite{BP2017} are naturally equivalent, and in \textsection \ref{sub-deformations} we study the deformation theory of $p$-divisible groups with $(\underline{s},\mu)$-structure.
\subsection{Crystalline Tate tensors}\label{sub-tatetensors}
Let $R$ be a $p$-nilpotent $W(k_0)$-algebra, and let $\underline{\mathbb{D}} = (\mathbb{D},\mathbb{F},\mathbb{V})$ be a Dieudonn\'e crystal on $\Spec R$. Suppose $\mathbb{D}_{R/R}$ is equipped with a filtration by finite projective $R$-modules
\begin{align}\label{eq-filt}
\textup{Fil}^0(\mathbb{D}) = \mathbb{D}_{R/R} \supset \textup{Fil}^1(\mathbb{D}) \supset \textup{Fil}^2(\mathbb{D}) = 0.
\end{align}
Extending the notation of the previous section, let us denote by $\mathbb{D}^\otimes$ the total tensor algebra of $\mathbb{D} \oplus \mathbb{D}^\vee$. This is a crystal of finite locally free $\mathcal{O}_{\Spec R/W(k_0)}$-modules, and the filtration (\ref{eq-filt}) naturally extends to a filtration for $\mathbb{D}^\otimes_{R/R}$. Further, the Frobenius for $\mathbb{D}$ endows the associated isocrystal $\mathbb{D}^\otimes[1/p]$ with the structure of an $F$-isocrystal as in (\ref{eq-tensorfrob}).

\begin{Def}
	A \textit{crystalline Tate tensor} for $\underline{\mathbb{D}}$ over $\Spec{R}$ is a morphism $t: \mathbbm{1} \to \mathbb{D}^\otimes$ of locally free crystals of $\mathcal{O}_{\Spec R/W(k_0)}$-modules such that $t_{R}(R) \subset \textup{Fil}^0(\mathbb{D}^\otimes)$	and such that the induced morphism of isocrystals $\mathbbm{1} \to \mathbb{D}^\otimes[1/p]$ is Frobenius equivariant.
\end{Def}

Let $\underline{G} = (G,\mu,\Lambda, \eta, \underline{s})$ be a local Hodge embedding datum in the sense of Definition \ref{def-lochodge}. As in the previous section, we have $s_i \in \Lambda^{\otimes m_i} \otimes (\Lambda^\vee)^{\otimes n_i} = \Lambda(i)$ for every $i$. More generally, throughout this section, we fix the pair $(m_i, n_i)$ associated with each $i$, and for any object $N$ in a rigid tensor category we define $N(i) := N^{\otimes m_i} \otimes (N^\vee)^{\otimes n_i}$. If $\psi$ is a morphism $N \to N'$, write $\psi(i)$ for the induced morphism $N(i) \to N'(i)$. 

\begin{Def}\label{def-smu}
	Let $R$ be a $p$-nilpotent $W(k_0)$-algebra, and let $\underline{\mathbb{D}}$ be a Dieudonn\'e crystal over $R$ whose $R$-sections are equipped with a filtration (\ref{eq-filt}). An \textit{$(\underline{s},\mu)$-structure on $\underline{\mathbb{D}}$} over $\Spec R$ is a finite collection of crystalline Tate tensors $\underline{t} = (t_1, \dots, t_r)$ satisfying the following conditions:
	\begin{enumerate}[(i)]		
		\item For every PD-thickening $B \to A$ over $R$, there is an extension $B \to B'$ which is faithfully flat and of finite presentation such that there is an isomorphism 
		\begin{align*}
		(\Lambda \otimes_{\zz_p} B', (\underline{s} \otimes 1)_{B'}) \xrightarrow{\sim} (\mathbb{D}_{B'/A'}, \underline{t}_{B'}),
		\end{align*}
		where $A' = A \otimes_B B'$. 
		\item For some faithfully flat \'etale extension $R \to R'$, there is an isomorphism
		\begin{align*}
		(\Lambda \otimes_{\zz_p} R', (\underline{s} \otimes 1)_{R'}) \xrightarrow{\sim} (\mathbb{D}_{R/R} \otimes_R R', \underline{t}_{R'})
		\end{align*}
		respecting the tensors, such that the filtration $\textup{Fil}^1(\mathbb{D})\otimes_R R' \subset \mathbb{D}_{R/R} \otimes_R R'  \xrightarrow{\sim} \Lambda \otimes_{\zz_p} R'$ is induced by $\mu$.
	\end{enumerate}
\end{Def}

\begin{rmk}\label{rmk-consequences}
	Let us derive a few consequences of Definition \ref{def-smu}, compare \cite[Def. 2.3.3 and Rmk. 2.3.5(b)]{HP2017}. Suppose $\underline{t} = (t_1, \dots, t_r)$ is an $(\underline{s},\mu)$-structure on a Dieudonn\'e crystal $\underline{\mathbb{D}}$ over $R$. Let $B \to A$ be a PD-thickening over $R$, and let  $T_{\mathbb{D}_{B/A}}$ denote the $B$-scheme of $B$-module isomorphisms $\Lambda \otimes_{\zz_p} B \xrightarrow{\sim} \mathbb{D}_{B/A}$ which respect the tensors. That is, for an $B$-algebra $B'$, 
	\begin{align*}
		T_{\mathbb{D}_{B/A}}(B') = \textup{Isom}((\Lambda\otimes_{\zz_p} B', (\underline{s} \otimes 1)_{B'}), (\mathbb{D}_{B/A} \otimes_A B', \underline{t}_B \otimes 1)).
	\end{align*}
	Definition \ref{def-smu} (i) implies that $T_{\mathbb{D}_{B/A}}$ is an fppf-locally trivial $G_A$-torsor. 

	Moreover, denote by $T_{\mathbb{D}_{R/R},\mu}$ the subscheme of $T_{\mathbb{D}_{R/R}}$ classifying $R$-module isomorphisms which identify the canonical filtration $\textup{Fil}^1(\Lambda \otimes_{\zz_p} R)$ defined by $\mu$ with the Hodge filtration $\textup{Fil}^1(\mathbb{D})$ of $\mathbb{D}$. Then Conditions (i) and (ii) of Definition \ref{def-smu} together imply that $T_{\mathbb{D}_{R/R},\mu}$ is a $P_{\mu,R}$-torsor, where $P_\mu$ is the parabolic subgroup of $G$ defined by $\mu$ as in (\ref{eq-parabolic}).
\end{rmk}

\begin{Def} Let $R$ be a $p$-nilpotent $W(k_0)$-algebra.
	\begin{enumerate}[(i)]
		\item A \textit{ $p$-divisible group with $(\underline{s},\mu)$-structure} over $R$ is a pair $(X,\underline{t})$ consisting of a $p$-divisible group $X$ over $R$ and an $(\underline{s},\mu)$-structure $\underline{t}$ on $\mathbb{D}(X)$ over $\Spec R$.
		\item A \textit{nilpotent Zink display with $(\underline{s},\mu)$-structure} over $R$ is a pair $(\underline{P}, \underline{t})$ consisting of a nilpotent Zink display $\underline{P}$ over $R$ and an $(\underline{s},\mu)$-structure $\underline{t}$ on $\mathbb{D}(\underline{P})$ over $\Spec R$.
	\end{enumerate}
\end{Def}

Denote by $\textup{fpdiv}_{\underline{s},\mu}(R)$ the category whose objects are formal $p$-divisible groups with $(\underline{s},\mu)$-structure and whose morphisms $(X,\underline{t}) \to (X',\underline{t'})$ are isomorphisms of $p$-divisible groups $X \to X'$ such that the composition of the tensor $t_i$ with the induced morphism $\mathbb{D}(X)^\otimes \to \mathbb{D}(X')^\otimes$ is the tensor $t_i'$ for every $i$. Similarly, let $\textup{nZink}_{\underline{s},\mu}(R)$ denote the category of nilpotent Zink displays with $(\underline{s},\mu)$-structure over $R$. As $R$ varies in $\textup{Nilp}_{W(k_0)}$, these determine fibered categories $\textup{fpdiv}_{\underline{s},\mu}$ and $\textup{nZink}_{\underline{s},\mu}$.

\begin{lemma}\label{lem-smudescent}
	The fibered categories $\textup{\textup{fpdiv}}_{\underline{s},\mu}$ and $\textup{\textup{nZink}}_{\underline{s},\mu}$ form stacks for the \'etale topology on $\textup{\textup{Nilp}}_{W(k_0)}$.
\end{lemma}
\begin{proof}
	It is well known that $p$-divisible groups form an fpqc stack on $\textup{Nilp}_{W(k_0)}$ (see e.g., \cite[Rmk. 2.4.2]{Messing1972}), and formal $p$-divisible groups form a substack because the property of being a formal $p$-divisible group is fpqc local on the base. Further, nilpotent Zink displays form an fpqc stack by \cite[Thm. 37]{Zink2002}. For the remainder of the proof, the same arguments work for both $\textup{fpdiv}_{\underline{s},\mu}$ and $\textup{nZink}_{\underline{s},\mu}$, so we give the proof only for the former.
	
	Let $R \to R'$ be a faithfully flat \'etale homomorphism of $p$-nilpotent $W(k_0)$-algebras. Denote by $\textup{fpdiv}_{\underline{s},\mu}(R'/R)$ the category of formal $p$-divisible groups with $(\underline{s},\mu)$-structure equipped with descent data from $\Spec R'$ down to $\Spec R$. We want to show the natural functor $\textup{fpdiv}_{\underline{s},\mu}(R) \to \textup{fpdiv}_{\underline{s},\mu}(R'/R)$ is an equivalence. That the functor is faithful is immediate from the corresponding property for $p$-divisible groups. Moreover, morphisms in $\textup{fpdiv}_{\underline{s},\mu}(R'/R)$ automatically descend to isomorphisms of $p$-divisible groups over $R$, and these isomorphisms must be compatible with the tensors by Lemma \ref{lem-etalelocal}. 
	
	It remains to prove that objects descend. Let $(X',\underline{t'})$ be a formal $p$-divisible group with $(\underline{s},\mu)$-structure over $R'$, equipped with a descent datum. We obtain an object $(X,\underline{t})$ over $R$ by descent for $p$-divisible groups and Lemma \ref{lem-etalelocal}. Frobenius equivariance of each $t_i$ follows from another application of Lemma \ref{lem-etalelocal}, and \'etale descent for $R$-modules implies that each $t_i$ preserves the filtrations. Condition (ii) of Definition \ref{def-smu} holds for $(X,\underline{t})$ because \'etale covers are stable under composition. To finish the proof we need only check that the first condition of Definition \ref{def-smu} holds for $(X,\underline{t})$. If $B \to A$ is a PD-thickening over $R$ then $A' = A \otimes_R R'$ is faithfully flat \'etale over $A$, and we can lift $B \to A$ to $B' \to A'$ with $B'$ faithfully flat \'etale over $B$. By the flatness of $B \to B'$, the divided powers extend to divided powers on the kernel of $B' \to A'$. Hence $B' \to A'$ is a PD-thickening over $\Spec R'$, so by condition (ii) for $(X', \underline{t'})$, there is an fppf cover $\Spec B'' \to \Spec B'$ trivializing $(\mathbb{D}(X')_{B'/A'}, \underline{t'})$. Then the composition $\Spec B'' \to \Spec B' \to \Spec B$ provides an fppf cover which trivializes $(\mathbb{D}(X)_{B/A}, \underline{t})$.
\end{proof}

\begin{rmk} \label{rmk-equiv}
	It is a consequence of the theorem of Zink and Lau (see \cite[Thm. 1.1.]{Lau2008}) and the compatibility of crystals (see Lemma \ref{lem-zinkdieudonne}) that the natural functor $(\underline{P},\underline{t}) \mapsto (\textup{BT}_R(\underline{P}), \underline{t})$ defines an equivalence between the stacks $\textup{nZink}_{\underline{s},\mu}$ and $\textup{fpdiv}_{\underline{s},\mu}$. 
\end{rmk}

Let us now study Grothendieck-Messing deformation theory in this setting. Let $R$ be a $p$-nilpotent $W(k_0)$-algebra and let $R_0 = R/pR$. If $X$ is a $p$-divisible group over $R$, define
\begin{align}\label{eq-TX}
	T(X)_{R/R_0} := T_{\mathbb{D}(X)_{R/R_0}}, \text{ and } T(X)_{R/R,\mu} := T_{\mathbb{D}(X)_{R/R},\mu},
\end{align}
where $T_{\mathbb{D}(X)_{R/R_0}}$ and $T_{\mathbb{D}(X)_{R/R},\mu}$ are the $G_R$- and $P_{\mu,R}$-torsors respectively defined in Remark \ref{rmk-consequences}. 

Suppose now $(X,\underline{t})$ is a formal $p$-divisible group with $(\underline{s},\mu)$-structure over $R$, and let $(X_{R_0},\underline{t}_{R_0})$ denote the formal $p$-divisible group with $(\underline{s},\mu)$-structure over $R_0$ obtained by base change. Then we have a canonical identification $T(X)_{R/R} \xrightarrow{\sim} T(X_{R_0})_{R/R_0}$ induced by the isomorphism $\mathbb{D}(X_{R_0})_{R/R_0} \xrightarrow{\sim} \mathbb{D}(X)_{R/R}$. It follows that the $P_{\mu,R}$-torsor $T(X)_{R/R,\mu}$ associated with $(X,\underline{t})$ determines a lift of $T(X_{R_0})_{R_0/R_0,\mu}$ inside of $T(X)_{R/R_0}$. The same constructions can all be carried out for nilpotent Zink displays with $(\underline{s},\mu)$-structure, and we denote the resulting $G_R$ and $P_{\mu,R}$-torsors by $T(\underline{P})_{R/R_0}$ and $T(\underline{P})_{R/R_0,\mu}$, respectively.

Define a groupoid $\textup{fpdiv}_{\underline{s},\mu}(R/R_0)$ as follows. For objects take pairs consisting of a formal $p$-divisible group with $(\underline{s},\mu)$-structure $(X,\underline{t})$ over $R_0$ and a lift $T_\mu \subset T(X)_{R/R_0}$ of $T(X)_{R_0/R_0,\mu} \subset T(X)_{R_0/R_0}$, and for morphisms take pairs of isomorphisms $\alpha: (X, \underline{t}) \xrightarrow{\sim} (X', \underline{t}')$ and $\beta: T(X)_{R/R_0,\mu} \xrightarrow{\sim} T(X')_{R/R_0,\mu}$ such that the isomorphism $T(X)_{R/R_0} \xrightarrow{\sim} T(X')_{R/R_0}$ induced by $\alpha$ restricts to $\beta$. Define similarly the category $\textup{nZink}_{\underline{s},\mu}(R/R_0)$.

\begin{lemma}\label{lem-liftspdiv}
	The functor 
	\begin{align}\label{eq-pdivBA}
		\textup{fpdiv}_{\underline{s},\mu}(R) \to \textup{fpdiv}_{\underline{s},\mu}(R/R_0)
	\end{align}
	defined by assigning to a $p$-divisible group with $(\underline{s},\mu)$-structure $(X,\underline{t})$ its reduction $(X_{R_0},\underline{t}_{R_0})$ mod $p$ along with the $P_{\mu,R}$-torsor $T(X)_{R/R,\mu}$ inside of $T(X)_{R/R}$ is an equivalence of categories. Moreover, the analogous result holds for nilpotent Zink displays with $(\underline{s},\mu)$-structure. 

\end{lemma}
\begin{proof}
	By Grothendieck-Messing theory, the functor $X \mapsto (X_{R_0}, \textup{Fil}^1(\mathbb{D}(X)))$ determines an equivalence of categories between formal $p$-divisible groups $X$ over $R$ and pairs $(X_0, E)$ consisting of a formal $p$-divisible group $X_0$ over $R_0$ and a lift of the Hodge filtration of $X_0$ to a direct summand $E \subset \mathbb{D}(X_0)_{R/R_0}$. Indeed, for $p \ge 3$, the divided powers for $R \to R_0$ are nilpotent so this follows from \cite[V. Thm. 1.6]{Messing1972}. For $p=2$ it holds because we are restricting our attention to formal $p$-divisible groups, see \cite[Rmk. 2.6]{Lau2013}. The analogous result holds for nilpotent Zink displays as well, see \cite[Thm. 48]{Zink2002}. We will give the remainder of the proof for formal $p$-divisible groups; the case of nilpotent Zink displays follows from the same arguments. 
	
	Grothendieck-Messing theory implies that the functor (\ref{eq-pdivBA}) is faithful. Let us prove it is full, so suppose $(X, \underline{t})$ and $(X', \underline{t}')$ are formal $p$-divisible groups over $R$ with $(\underline{s},\mu)$-structure, and suppose we have isomorphisms
	\begin{align*}
		\alpha: (X_{R_0},\underline{t}_{R_0}) \xrightarrow{\sim} (X_{R_0}', \underline{t}_{R_0}') \text{ and } \beta: T(X_{R_0})_{R/R_0,\mu} \xrightarrow{\sim} T(X'_{R_0})_{R/R_0,\mu}
	\end{align*}
	such that the induced isomorphism $\alpha_\ast: T(X_{R_0})_{R/R_0} \xrightarrow{\sim} T(X'_{R_0})_{R/R_0}$ restricts to $\beta$. By definition of $T(X_{R_0})_{R/R_0,\mu}$ and $T(X'_{R_0})_{R/R_0,\mu}$, it follows that $\alpha(\textup{Fil}^1(\mathbb{D}(X))) = \textup{Fil}^1(\mathbb{D}(X'))$, \'etale locally on $\Spec R$. By \'etale descent for finite projective $R$-modules the Hodge filtration will be preserved over $R$ as well, so $\alpha$ lifts to a morphism $X \to X'$ by Grothendieck-Messing theory. The tensors are preserved by the lift because of the equivalence (\ref{eq-crystalequiv}) between $\textup{LFCrys}(R_0/W(k_0))$ and $\textup{LFCrys}(R/W(k_0))$. 
	
	Let $(X_0, \underline{t}_0)$ be a formal $p$-divisible group with $(\underline{s},\mu)$-structure over $R_0$ with a lift $T_\mu$ of $T(X_0)_{R_0/R_0,\mu}$. By \'etale descent, it is enough to prove essential surjectivity \'etale locally, so we may assume $T_\mu$ is a trivial $P_{\mu,R}$-torsor. Then any $\psi \in T_\mu \subset T(X_0)_{R/R_0}$ induces an isomorphism $(\Lambda \otimes_{\zz_p} R, (\underline{s} \otimes 1)_R) \xrightarrow{\sim} (\mathbb{D}(X_0)_{R/R_0}, \underline{t}_R)$ such that the base change $\psi_0$ of $\psi$ along $R \to R_0$ identifies the Hodge filtrations, i.e.,
	\begin{align}\label{eq-hodgefilt0}
		\psi_0(\textup{Fil}^1(\Lambda \otimes_{\zz_p} R_0)) = \textup{Fil}^1(\mathbb{D}(X_0)),
	\end{align}
	where $\textup{Fil}^1(\Lambda \otimes_{\zz_p} R_0)$ is the filtration defined by $\mu$. Define $E = \psi(\textup{Fil}^1(\Lambda \otimes_{\zz_p} R)) \subset \mathbb{D}(X_0)_{R/R_0}$. By (\ref{eq-hodgefilt0}), $E$ is a lift of the Hodge filtration for $X_0$, and therefore the pair $(X_0, E)$ lifts to a formal $p$-divisible group $X$ over $R$ by Grothendieck-Messing theory. It is immediate from (\ref{eq-crystalequiv}) that the tensors $\underline{t}_0$ lift to a set of tensors $\underline{t}$ for $X$, so it remains only to show that conditions (i) and (ii) of Definition \ref{def-smu} are satisfied.
	
	For condition (i), let $B \to A$ be a PD-thickening over $R$. Then $B \to A \to A/p = A_0$ is a PD-thickening over $R_0$, so there exists a homomorphism $B \to B'$ which is faithfully flat and of finite presentation such that there exists an isomorphism
	\begin{align*}
		(\Lambda \otimes_{\zz_p} B, (\underline{s} \otimes 1)_B) \xrightarrow{\sim} (\mathbb{D}(X_0)_{B/A_0}, \underline{t_0}_B).
	\end{align*}
	Then condition (i) follows from the identification $(\mathbb{D}(X_0)_{B/A_0}, \underline{t_0}_B) \xrightarrow{\sim} (\mathbb{D}(X)_{B/A}, \underline{t}_B)$. Condition (ii) is satisfied because the isomorphism $\psi$ respects the tensors, and it respects the Hodge filtration by definition of $X$. 
\end{proof}

\subsection{From $G$-displays to $p$-divisible groups}\label{sub-functor}
Let $\underline{G} = (G,\mu,\Lambda,\eta,\underline{s})$ be a local Hodge embedding datum in the sense of Definition \ref{def-lochodge}. Let $\mathscr{P}$ be a Tannakian $(G,\mu)$-display over $\underline{W}(R)$ which is nilpotent with respect to $\eta$, let $\underline{P} = \Z_{\eta,R}(\mathscr{P})$ be the associated Zink display, and let $X = \textup{BT}_R(\underline{P})$ be the associated formal $p$-divisible group. As in the previous section, the tensors $s_i$, viewed as morphisms $\zz_p \to \Lambda(i)$, induce morphisms of crystals
\begin{align}\label{eq-ti}
t_i := \mathbb{D}(\mathscr{P})(s_i): \mathbbm{1} \to \mathbb{D}(\mathscr{P})^{\eta(i)}.
\end{align}
Following the notation of the previous section, we write $\mathbb{D}(X)(i)= \mathbb{D}(X)^{\otimes m_i} \otimes (\mathbb{D}(X)^\vee)^{\otimes n_i}$. If $B \to A$ is a $p$-adic PD-thickening, then $\mathbb{D}(X)_{B/A}$ is $p$-adically complete and separated, since same holds for any finite projective $B$-module. The same is true of $\left(\mathbb{D}(X)_{B/A}\right)(i)$, and hence the natural map
\begin{align}\label{eq-mini}
\left(\mathbb{D}(X)_{B/A}\right)(i) \to \left(\mathbb{D}(X)(i)\right)_{B/A}
\end{align}
is an isomorphism.

By combining (\ref{eq-eta}) with Lemma \ref{lem-zinkdieudonne} and applying the compatibility of $\mathbb{D}(\mathscr{P})$ with tensor products, we have $\mathbb{D}(\mathscr{P})^{\eta(i)} \cong \mathbb{D}(X)(i)$, and hence we obtain morphisms of crystals
\begin{align}\label{eq-tensors}
t_i: \mathbbm{1} \to \mathbb{D}(X)^\otimes
\end{align}
for each $i$. By Lemma \ref{lem-zinkdieudonne}, it is equivalent to view $t_i$ as a morphism $\mathbbm{1} \to \mathbb{D}(\underline{P})^\otimes$.

\begin{lemma}\label{lem-tensors}
	The pair $(X,\underline{t})$ $($resp. $(\underline{P}, \underline{t})$$)$ defines a formal $p$-divisible group $($resp. nilpotent Zink display$)$ with $(\underline{s},\mu)$-structure.
\end{lemma}
\begin{proof}
	It is enough to prove that $(X, \underline{t})$ is a $p$-divisible group with $(\underline{s},\mu)$-structure. Let us write $\underline{M}^\pi$ for the evaluation of $\mathscr{P}$ on a representation $(V,\pi)$. We have isomorphisms $\mathbb{D}(X) \cong \mathbb{D}(\Z_{\eta,R}(\mathscr{P})) \cong \mathbb{D}(\mathscr{P})^\eta$, which all preserve the respective filtrations (see (\ref{eq-samehodge}) and Lemma \ref{lem-zinkdieudonne}), and since the Hodge filtrations of displays are compatible with tensor products (see Remark \ref{rmk-hodge}), we can conclude that the filtration on $\mathbb{D}(X)(i)_{R/R}$ induced from the filtration on $\mathbb{D}(X)_{R/R}$ agrees with the Hodge filtration of $\underline{M}^{\eta(i)}$. Similarly, the filtration on $\mathbbm{1}$ agrees with the one on the unit display $\underline{S} = (S,\sigma)$, so it is enough to show the map
	\begin{align*}
	(t_i)_{R}: R \to \tau^\ast M^{\eta(i)} \otimes_{W(R)} R
	\end{align*}
	preserves the filtrations of the corresponding displays. But the map $(t_i)_{R/R}$ is defined as the reduction of the map $\underline{S} \to \underline{M}^{\eta(i)}$ induced by $s_i$, so this is automatic (see again Remark \ref{rmk-hodge}). Frobenius equivariance follows from Proposition \ref{prop-frobeq} and the comparison of crystals, so we can conclude $\underline{t}$ is a collection of crystalline Tate tensors on $\underline{\mathbb{D}(X)}$ over $\Spec R$.
	
	The lift $\mathscr{P}_{B/A}$ of $\mathscr{P}_{\underline{W}(A)}$ is \'etale locally banal for any PD-thickening $B \to A$ over $R$. Thus for some \'etale faithfully flat extension $B \to B'$, there is an isomorphism of tensor functors
	\begin{align*}
		\mathbb{D}(\mathscr{P})_{B'/A'} \xrightarrow{\sim} \omega_{B'},
	\end{align*}
	where $\omega_{B'}$ is the usual fiber functor, see (\ref{eq-usualfiberfunctor}). Condition (i) follows. 
	
	For condition (ii), by (\ref{eq-eta}) and Lemma \ref{lem-zinkdieudonne}, we have a canonical isomorphism
	\begin{align}\label{eq-isomhodgefilt}
		\mathbb{D}(X)_{R/R} \xrightarrow{\sim} \mathbb{D}(\mathscr{P})_{R/R}^\eta.
	\end{align}
	Write $\mathscr{P}(\Lambda, \eta) = \underline{M}^\eta$, and endow $\mathbb{D}(\mathscr{P})^\eta_{R/R} = \tau^\ast M^\eta \otimes_{W{R}} R$ with the Hodge filtration as in (\ref{eq-hodgedisp}). Then (\ref{eq-isomhodgefilt}) preserves the respective Hodge filtrations (see (\ref{eq-samehodge})). Now choose a faithfully flat \'etale extension $R \to R'$ such that $\mathscr{P}_{\underline{W}(R')}$ is banal, with a trivialization $\psi: \mathscr{P}_U \xrightarrow{\sim} \mathscr{P}_{\underline{W}(R')}$ for some $U \in L^+G(R')$. Then $\psi$ induces an isomorphism
	\begin{align}\label{eq-condii}
		 \Lambda \otimes_{\zz_p} R' = \mathbb{D}(\mathscr{P}_U)^\eta_{R'/R'} \xrightarrow{\sim} \mathbb{D}(\mathscr{P})^\eta_{R'/R'}.
	\end{align}
	Thus by (\ref{eq-isomhodgefilt}) and (\ref{eq-condii}), it is enough to show $\textup{Fil}^1(\mathbb{D}(\mathscr{P}_U)^\eta) = \Lambda^1 \otimes_{W(k_0)} R'$. But if $\mathscr{P}_U(\Lambda, \eta) = \underline{M}^\eta$, then
	\begin{align*}
	\textup{Fil}^1(\mathbb{D}(\mathscr{P}_U)^\eta) = \textup{im}(\bar{\theta_1}),
	\end{align*}
	where $\bar{\theta}_1$ is the map $M_1^\eta \to \tau^\ast M^\eta \to \tau^\ast M^\eta \otimes_{W(R')} R'$. Since $\mathscr{P}_U$ is banal, we have 
	\begin{align*}
	M_1^\eta = (\Lambda^0 \otimes_{W(k_0)} I_{R'}) \oplus (\Lambda^1 \otimes_{W(k_0)} W(R')),
	\end{align*}	
	and $\bar{\theta}_1$ is reduction modulo $I_{R'}$, so the result follows.
\end{proof}

If $\mathscr{P} \to \mathscr{P}'$ is a morphism Tannakian $(G,\mu)$-displays over $\underline{W}(R)$ which are nilpotent with respect to $\eta$, then it follows from the natural transformation property that the resulting morphisms $\underline{P} \to \underline{P}'$ and $X \to X'$ are compatible with the $(\underline{s},\mu)$-structure. Hence we obtain functors
\begin{align}\label{eq-BTG}
\textup{BT}_{\underline{G},R}: G\textup{-}\textup{Disp}_{\underline{W},\mu}^{\otimes, \eta}(R) \to \textup{fpdiv}_{\underline{s},\mu}(R), \ \mathscr{P} \mapsto (\textup{BT}_R(\Z_{\eta,R}(\mathscr{P})), \underline{t}),
\end{align}
and
\begin{align}\label{eq-ZG}
\Z_{\underline{G},R}: G\textup{-}\textup{Disp}_{\underline{W},\mu}^{\otimes, \eta}(R) \to \textup{nZink}_{\underline{s},\mu}(R), \ \mathscr{P} \mapsto (\Z_{\eta,R}(\mathscr{P}), \underline{t}).
\end{align}

The following lemmas will be useful in the proofs of Theorem \ref{thm-1} and Corollary \ref{introcor1}. Following \cite[\textsection 2.2]{Zink2002}, if $\mathbb{D}$ is a crystal of $\mathcal{O}_{\Spec R/ W(k_0)}$-modules, then we define $\mathbb{D}_{W(R)/R}$ by
\begin{align*}
	\mathbb{D}_{W(R)/R} = \varprojlim \mathbb{D}_{W_n(R)/R}.
\end{align*}
By \cite[Prop. 53]{Zink2002}, if $\underline{P}$ is a nilpotent Zink display over $\underline{W}(R)$, there is a canonical isomorphism
\begin{align}\label{eq-zeta}
	\zeta: \mathbb{D}(\underline{P})_{W(R)/R} \xrightarrow{\sim} P.
\end{align}
Explicitly, the isomorphism is defined as follows. The Cartier homomorphism $\Delta: W(R) \to W(W(R))$ (see \cite[(90)]{Zink2002}) defines a morphism of 1-frames $\mathcal{W}(R) \to \mathcal{W}(W(R)/R)$, and the base change $\Delta^\ast \underline{P}$ is lift of $\underline{P}$ to a $\mathcal{W}(W(R)/R)$-window (note that by \cite[Lem. 2.12]{Lau2010} we can freely pass between $\mathcal{W}(W(R)/R)$-windows and compatible systems of $\mathcal{W}(W_n(R)/R)$-windows for varying $n$ as defined in \cite{Zink2002}). Such a lift is unique up to unique isomorphism lifting $\id_{\underline{P}}$ by \cite[Thm. 44]{Zink2002}, so we have an isomorphism of $\mathcal{W}(W(R)/R)$-windows
\begin{align}\label{eq-tildezeta}
	\tilde{\zeta}: \underline{P}_{W(R)/R} \xrightarrow{\sim} \Delta^\ast \underline{P},
\end{align}
which reduces to the identity after base change along $\mathcal{W}(W(R)/R) \to \mathcal{W}(R)$. Here $\underline{P}_{W(R)/R}$ is the lift of $\underline{P}$ used to define $\mathbb{D}(\underline{P})$. Then (\ref{eq-zeta}) is obtained by tensoring (\ref{eq-tildezeta}) along $\hat{w}_0: W(W(R)) \to W(R)$, where $\hat{w}_0$ denotes the zeroth ghost coordinate for $W(W(R))$. It is clear from this description and uniqueness of lifts to $\mathcal{W}(W(R)/R)$ that $\zeta$ is functorial in $\underline{P}$. In other words, if $\underline{P}'$ is another nilpotent Zink display with corresponding homomorphism $\zeta'$ as in (\ref{eq-zeta}), and $\beta: \underline{P} \to \underline{P}'$ is a morphism of displays, then
\begin{align}\label{eq-zetafunctorial}
	\beta \circ \zeta = \zeta' \circ \mathbb{D}(\beta).
\end{align}

If $\mathscr{P}$ is a Tannakian $(G,\mu)$-display over $\underline{W}(R)$ which is nilpotent with respect to $\eta$, and $\mathscr{P}(\Lambda, \eta) = (M^\eta, F^\eta)$, then by Lemma \ref{lem-comparecrystals} there is an isomorphism $\mathbb{D}(\mathscr{P})^\eta \xrightarrow{\sim} \mathbb{D}(\Z_\eta(\mathscr{P}))$. Combining this with (\ref{eq-zeta}), we obtain an isomorphism $\mathbb{D}(\mathscr{P})^\eta_{W(R)/R} \xrightarrow{\sim} \tau^\ast M$, which we also denote by $\zeta$. Using compatibility of $\mathscr{P}$ and $\mathbb{D}(\mathscr{P})$ with tensor products, $\zeta$ extends to 
\begin{align}\label{eq-zetai}
	\zeta(i): \mathbb{D}(\mathscr{P})^{\eta(i)}_{W(R)/R} \xrightarrow{\sim} \tau^\ast M^{\eta(i)}.
\end{align}

\begin{lemma}\label{lem-Dpsi}
	Let $\mathscr{P}$ be a Tannakian $(G,\mu)$-display over $\underline{W}(R)$ which is nilpotent with respect to $\eta$. Suppose $\mathscr{P}$ is banal, with a trivialization given by $\psi: \mathscr{P}_U \xrightarrow{\sim} \mathscr{P}$ for some $U \in L^+G(R)$. Then there exists a unique isomorphism of tensor functors
	\begin{align*}
		\Psi: \omega_{W(R)} \xrightarrow{\sim} \mathbb{D}(\mathscr{P})_{W(R)/R}
	\end{align*}
	such that $\zeta \circ \Psi^\eta = \tau^\ast \psi^\eta$. 
\end{lemma}
\begin{proof}
	Uniqueness follows immediately from the identity $\zeta \circ \Psi^\eta = \tau^\ast \psi^\eta$ because the representation $(\Lambda, \eta)$ is a tensor generator for the category $\textup{Rep}_{\zz_p}(G)$ (see for example \cite[Thm. 2.2.8]{Wilson}), and any two morphisms of tensor functors which agree after evaluation on a tensor generator will agree in general. 
	
	Next we prove existence. For every $n \ge 1$, denote by $r_n$ the natural quotient $W(R) \to W_n(R)$, and let $U_n = (W(r_n)\circ \Delta)(U) \in L^+G(W_n(R))$. For each $n$, the trivialization $\psi$ lifts to a trivialization $\psi_n: \mathscr{P}_{U_n} \xrightarrow{\sim} \mathscr{P}_n$, where $\mathscr{P}_n$ is the unique lift of $\mathscr{P}$ to an adjoint nilpotent Tannakian $(G,\mu)$-display over $\underline{W}(W_n(R)/R)$-display, see Remark \ref{lem-banalcrystal}. Hence we obtain isomorphisms $V \otimes_{\zz_p} W_n(R) \xrightarrow{\sim} \mathbb{D}(\mathscr{P})^\rho_{W_n(R)/R}$ for every representation $(V,\rho)$. Moreover, these are compatible with the natural maps $V \otimes_{\zz_p} W_n(R) \to V \otimes_{\zz_p} W_{n-1}(R)$ induced by $r_{n-1}$ because $U_n$ is a compatible system of lifts. In this way we obtain an isomorphism of tensor functors 
	\begin{align*}
		\Psi: \omega_{W(R)} \xrightarrow{\sim} \mathbb{D}(\mathscr{P})_{W(R)/R}.
	\end{align*}

	It remains to show 
	\begin{align}\label{eq-want}
		\zeta \circ \Psi^\eta = \tau^\ast \psi^\eta
	\end{align}
	Let $\Z_\eta(\mathscr{P})_{W(R)/R}$ and $\Z_\eta(\mathscr{P}_U)_{W(R)/R}$ be the unique lifts of $\Z_\eta(\mathscr{P})$ and $\Z_\eta(\mathscr{P}_U)$, respectively, to windows over $\mathcal{W}(W(R)/R)$ By \cite[Lem 2.12]{Lau2010}, $\Z_\eta(\mathscr{P})_{W(R)/R}$ is the inverse limit of the compatible system of lifts $\Z_\eta(\mathscr{P})_{W_n(R)/R} = \Z_\eta(\mathscr{P}_n)$, and likewise for $\Z_\eta(\mathscr{P}_U)_{W(R)/R}$. We claim $\Psi^\eta$ is the reduction modulo $\hat{w}_0$ of the isomorphism of $\mathcal{W}(W(R)/R)$-windows
	\begin{align*}
		\Z_\eta(\psi)_{W(R)/R}: \Z_\eta(\mathscr{P}_U)_{W(R)/R} = \varprojlim \Z_\eta(\mathscr{P}_{U_n}) \xrightarrow{\lim \tau^\ast(\psi_n)^\eta} \varprojlim \Z_\eta(\mathscr{P}_n) = \Z_\eta(\mathscr{P})_{W(R)/R}.
	\end{align*}
	Indeed, this can be checked after applying $-\otimes_{W(R)} W_n(R)$ to the underlying $W(R)$-modules for every $n$, and therefore the result follows from the identity $r_n \circ \hat{w}_0 = \hat{w}_0 \circ W(r_n)$. Moreover, since $\hat{w}_0 \circ \Delta = \id_{W(R)}$, $\tau^\ast \psi^\eta$ is the reduction modulo $\hat{w}_0$ of the morphism of $\mathcal{W}(W(R)/R)$-windows $\Delta^\ast \tau^\ast \psi^\eta: \Z_\eta(\mathscr{P}_U)_{W(R)/R} = \Delta^\ast \Z_\eta(\mathscr{P}_U) \to \Delta^\ast \Z_\eta(\mathscr{P})$, and $\zeta$ is the reduction of $\tilde{\zeta}$ (see (\ref{eq-tildezeta})). Thus to show (\ref{eq-want}), it is enough to show the identity \
	\begin{align*}
		\Delta^\ast \tau^\ast \psi^\eta = \tilde{\zeta} \circ \Z_\eta(\psi)_{W(R)/R}
	\end{align*}
	of morphisms of $\mathcal{W}(W(R)/R)$-windows. This can be checked after base change to $\mathcal{W}(R)$. But $W(w_0)^\ast \Delta^\ast \tau^\ast \psi^\eta = \tau^\ast \psi^\eta$ because $W(w_0) \circ \Delta = \id_{W(R)}$, and $W(w_0)^\ast \Z_\eta(\psi)_{W(R)/R} = \tau^\ast(\psi_1)^\eta = \tau^\ast \psi^\eta$. The result follows because $\tilde{\zeta}$ lifts the identity of $\Z_\eta(\mathscr{P})$.
\end{proof}

\begin{lemma}\label{lem-tensormatch}
	Let $\mathscr{P}$ be a Tannakian $(G,\mu)$-display over $\underline{W}(R)$ which is nilpotent with respect to $\eta$, and let $t_i = \mathbb{D}(\mathscr{P})(s_i)$ as in $($\ref{eq-ti}$)$. Then $\zeta(i) \circ (t_i)_{W(R)} = \tau^\ast\mathscr{P}(s_i)$. 
\end{lemma}
\begin{proof}
	By Zink's Witt vector descent \cite[Prop. 33]{Zink2002}, the question is fpqc-local on $\Spec R$, so we may assume $\mathscr{P}$ is banal, with a trivialization $\psi: \mathscr{P}_U \to \mathscr{P}$ for some $U \in L^+G(R)$. Then by Lemma \ref{lem-Dpsi} there is an isomorphism $\Psi: \omega_{W(R)/R} \xrightarrow{\sim} \mathbb{D}(\mathscr{P})_{W(R)/R}$ such that $\zeta(\eta) \circ \Psi^\eta = \tau^\ast \psi^\eta$. 
	
	Because $\psi$ is a morphism of tensor functors, we have $\mathscr{P}(s_i) = \tau^\ast \psi^{\eta(i)} \circ (s_i \otimes 1)_{W(R)}$. Likewise $(t_i)_{W(R)} = \Psi^{\eta(i)} \circ (s_i \otimes 1)_{W(R)}$. Then
	\begin{align*}
		\zeta(i) \circ (t_i)_{W(R)} = \zeta(i) \circ \Psi^{\eta(i)} \circ (s_i \otimes 1)_{W(R)} = \tau^\ast\psi^{\eta(i)} \circ (s_i \otimes 1)_{W(R)} = \mathscr{P}(s_i). 
	\end{align*}
\end{proof}

\subsection{Proof of Theorem \ref{thm-1}} \label{sub-mainthm}

In this section we prove Theorem \ref{thm-1}. Our strategy is as follows: We first prove the theorem in the case where $pR = 0$ and $R$ admits a $p$-basis \'etale locally by following the strategy in the proof of \cite[Thm. 5.15]{Daniels2019}. That is, we first prove full-faithfulness of the functor, and then we reduce essential surjectivity to the banal case using descent, see Proposition \ref{prop-pR=0} below. The case of general $R$ is then reduced to the case where $pR = 0$ using the analogs of Grothendieck-Messing theory in the two settings; this is Theorem \ref{mainthm}. Let $\underline{G} = (G,\mu,\Lambda,\eta,\underline{s})$ be a local Hodge embedding datum as in the previous section.

Let $R$ be a $p$-nilpotent $W(k_0)$-algebra, and let $M$ be a finite projective graded $W(R)^\oplus$-module. Suppose we are given a collection $\underline{u} = (u_1, \dots, u_r)$ of $W(R)$-module homomorphisms $u_i: W(R) \to (\tau^\ast M)(i)$. Define
\begin{align*}
Q_{M,\underline{u}} = \underline{\textup{Isom}}^0((\Lambda_{W(R)^\oplus},(\underline{s} \otimes 1)_{W(R)}), (M, \underline{u}))
\end{align*}
to be the fpqc sheaf on $\Spec R$ of isomorphisms of graded $W(R)^\oplus$-modules $\psi: \Lambda_{W(R)^\oplus} \xrightarrow{\sim} M$ which respect the tensors after pulling back by $\tau$, in the sense that $(\tau^\ast \psi)(i) \circ (s_i \otimes 1)_{W(R)} = u_i$ for every $i$. We will denote such an isomorphism by $(\Lambda_{W(R)^\oplus}, (\underline{s}\otimes 1)_{W(R)}) \xrightarrow{\sim} (M,\underline{u})$. We write $\underline{\textup{Aut}}^0(\Lambda_{W(R)^\oplus},(\underline{s} \otimes 1)_{W(R)})$ for the sheaf $Q_{\Lambda_{W(R)^\oplus},(\underline{s}\otimes 1)_{W(R)}}$. When the set of tensors is empty, we denote the corresponding sheaf simply by $\underline{\textup{Aut}}^0(\Lambda_{W(R)^\oplus})$. By the arguments of \cite[Lem. 3.9]{Daniels2019}, we have
\begin{align} \label{eq-gl}
	\underline{\textup{Aut}}^0(\Lambda_{W(R)^\oplus}) \cong L^+_{\eta\circ\mu} \textup{GL}(\Lambda). 
\end{align}
It follows from (\ref{eq-gl}) and Lemma \ref{lem-checkhodge} below that we have an identification
\begin{align}\label{eq-aut} 	
\underline{\textup{Aut}}^0(\Lambda_{W(R)^\oplus},(\underline{s}\otimes 1)_{W(R)}) = L^+_\mu G.
\end{align}

\begin{lemma}\label{lem-checkhodge}
	Let $g \in L^+_{\eta\circ\mu}\textup{GL}(\Lambda)(R)$. Then $g \in L^+_\mu G(R)$ if and only if $\tau(g) \in L^+G(R)$. 
\end{lemma}
\begin{proof}
	For any reductive group scheme $H$ over $\zz_p$ with cocharacter $\lambda: \mathbb{G}_m \to H$, let $P_\lambda \subset H$ be the parabolic subgroup defined by $\lambda$, see (\ref{eq-parabolic}). Following \cite{BP2017}, define also a closed subgroup scheme $H^\lambda \subset L^+H$  by
	\begin{align*}
		H^\lambda(R) = \{h \in H \mid h_0 \in P_\lambda(R)\}
	\end{align*}
	for any $\zz_p$-algebra $R$, where $h_0$ denotes the image of $h$ under $w_0: H(W(R)) \to H(R)$. 
	
	By \cite[Rmk. 6.3.3]{Lau2018}, $\tau$ induces isomorphisms
	\begin{align*}
		L^+_{\eta\circ\mu} \textup{GL}(\Lambda) \xrightarrow{\sim} H^{\eta\circ \mu} \text{ and } L^+_\mu G \xrightarrow{\sim} H^\mu.
	\end{align*}
	Thus we reduce to showing that $H^{\eta\circ\mu} \cap L^+G = H^\mu$, which follows from the identity $P_{\eta\circ\mu} \cap G = P_\mu$, see \cite[Prop. 4.1.10, 1.]{Conrad2014}.
\end{proof}

Suppose $(\underline{P}, \underline{t})$ is a nilpotent Zink display with $(\underline{s},\mu)$-structure over $R$, and let $\underline{M} = M_{\underline{W}(R)}(\underline{P})$ be the $1$-display associated with $\underline{P}$ as in Lemma \ref{lem-windows} (here we use notation as in (\ref{eq-displaytowindow})). Recall the isomorphism $\zeta: \mathbb{D}(\underline{P})_{W(R)/R} \xrightarrow{\sim} P = \tau^\ast M$ (see (\ref{eq-zeta})), which extends to an isomorphism $\zeta(i): \mathbb{D}(\underline{P})_{W(R)/R}(i) \xrightarrow{\sim} \tau^\ast M(i)$. For each $i$, then, we obtain a $W(R)$-module homomorphism 
\begin{align*}
\zeta(i) \circ (t_i)_{W(R)}: W(R) \to (\tau^\ast M)(i),
\end{align*}
which we denote by $u_i$. Notice here that we are using the natural isomorphism (\ref{eq-mini}) to identify
\begin{align*} 
	\left(\mathbb{D}(\underline{P})_{W(R)/R}\right)(i) \xrightarrow{\sim} \left(\mathbb{D}(\underline{P})(i)\right)_{W(R)/R}.
\end{align*}

\begin{lemma}\label{lem-torsor}
	Let $(\underline{P},\underline{t})$ be a nilpotent Zink display with $(\underline{s},\mu)$-structure, and let $\underline{M} = M_{\underline{W}(R)}(\underline{P})$ be the 1-display associated with $\underline{P}$. Let $u_i =\zeta(i) \circ (t_i)_{W(R)}$. Then the fpqc sheaf
	\begin{align*}
	Q_{M,\underline{u}} = \underline{\textup{Isom}}^0((\Lambda_{W(R)^\oplus}, (\underline{s}\otimes 1)_{W(R)}), (M, \underline{u}))
	\end{align*}
	is an \'etale locally trivial $L^+_\mu G$-torsor over $\Spec R$.
\end{lemma}
\begin{proof}
	By (\ref{eq-aut}) it is enough to show that, \'etale locally, there is an isomorphism $\psi: (\Lambda_{W(R)^\oplus},(\underline{s} \otimes 1)_{W(R)}) \xrightarrow{\sim} (M,\underline{t}_{W(R)})$. Moreover, letting
	\begin{align*}
	\Fil \Lambda_{W(R)} := I(R) (\Lambda^0 \otimes_{W(k_0)} W(R)) \oplus (\Lambda^1 \otimes_{W(k_0)} W(R)),
	\end{align*}
	we see that it is enough to show that, \'etale locally, there is an isomorphism $\overline{\psi}: \Lambda_{W(R)} \xrightarrow{\sim} P$ which sends $\Fil \Lambda_{W(R)}$ into $\Fil P$ and which respects the tensors.
	
	Condition (ii) in Definition \ref{def-smu} implies that, after replacing $R$ by some faithfully flat \'etale extension, we have an isomorphism $\Lambda_R \xrightarrow{\sim} \mathbb{D}(\underline{P})_{R/R}$ which sends $\Lambda^1_R$ into $\textup{Fil}^1(\mathbb{D}(\underline{P}))$ and which respects the tensors. Recalling the identifications
	\begin{align*}
	\mathbb{D}(\underline{P})_{R/R} = P / I(R) P, \  \textup{Fil}^1(\mathbb{D}(\underline{P})) = \Fil P / I(R) P, \text{ and } \zeta: \mathbb{D}(\underline{P})_{W(R)/R} \xrightarrow{\sim} P,
	\end{align*}
	we reduce the proof to showing that any such isomorphism lifts to an isomorphism $\Lambda_{W(R)} \xrightarrow{\sim} \mathbb{D}(\underline{P})_{W(R)/R}$ which respects the tensors, since any lift will automatically preserve the filtrations. 
	
	Define $Y$ to be the $W(R)$-scheme whose points in a $W(R)$-algebra $R'$ are isomorphisms $\Lambda_{R'} \xrightarrow{\sim} \mathbb{D}(\underline{P})_{W(R)/R} \otimes_{W(R)} R'$ which respect the tensors, i.e.
	\begin{align*}
	Y(R') = \textup{Isom}((\Lambda_{R'}, (\underline{s}\otimes 1)_{R'}), (\mathbb{D}(\underline{P})_{W(R)/R} \otimes_{W(R)} R', \underline{t}_{W(R)} \otimes \id_{R'})).
	\end{align*}
	We need to show that the natural map $Y(W(R)) \to Y(R)$ is surjective. For every $n$, define the analogous $W_n(R)$-scheme $Y_n$, so for any $W_n(R)$-algebra $R'$ we have
	\begin{align*}
	Y_n(R') = \textup{Isom}((\Lambda_{R'}, (\underline{s}\otimes 1)_{R'}), (\mathbb{D}(X)_{W_n(R)/R} \otimes_{W_n(R)} R', \underline{t}_{W_n(R)} \otimes \id_{R'})).
	\end{align*}
	Then, in particular, $Y_n(R') = Y(R')$ for all $W_n(R)$-algebras, and condition (i) of Definition \ref{def-smu} implies that $Y_n$ is an fppf locally trivial $G_{W_n(R)}$-torsor. In particular, $Y_n$ is formally smooth over $W_n(R)$. Since $W_n(R) \to W_{n-1}(R)$ has nilpotent kernel for all $p$-nilpotent $W(k_0)$-algebras $R$, it follows that the natural map $Y_n(W_n(R)) \to Y_n(W_{n-1}(R))$ is surjective for all $n$. Hence $Y(W_n(R)) \to Y(W_{n-1}(R))$ is surjective for all $n$, and therefore so too is $Y(W(R)) \to Y(R)$. 
\end{proof}	

Continuing with the notation of Lemma \ref{lem-torsor}, so $(\underline{P}, \underline{t})$ is a nilpotent Zink display with $(\underline{s},\mu)$-structure over $R$, and $\underline{M}$ is the corresponding 1-display over $\underline{W}(R)$. Thus we have an identity $P = \tau^\ast M$. Suppose $\beta \in Q_{M,\underline{t}_{W(R)}}(R)$. Then $\tau^\ast \beta$ defines an isomorphism $\Lambda_{W(R)} = \tau^\ast(\Lambda_{W(R)^\oplus}) \xrightarrow{\sim} \tau^\ast M$, and the composition
\begin{align}\label{eq-Ubeta}
	\Lambda_{W(R)} \xrightarrow{\sim} \sigma^\ast \Lambda_{W(R)^\oplus} \xrightarrow{\sigma^\ast \beta} \sigma^\ast M \xrightarrow{F^\sharp} \tau^\ast M \xrightarrow{\tau^\ast \beta^{-1}} \Lambda_{W(R)}
\end{align}
determines an element $U_\beta \in \textup{GL}(\Lambda_{W(R)})$. 

Let $L = \beta(\Lambda \otimes_{\zz_p} W(R)) \subset M$ viewed as a graded $W(R)$-module. Then multiplication induces an isomorphism of graded $W(R)^\oplus$-modules $L\otimes_{W(R)} W(R)^\oplus \xrightarrow{\sim} M$. This gives us an identification of $W(R)$-modules
\begin{align}\label{eq-tauast}
	L \xrightarrow{\sim} \tau^\ast(L\otimes_{W(R)} W(R)^\oplus) \xrightarrow{\sim} \tau^\ast M = P
\end{align}
such that the composition $\Lambda \otimes_{\zz_p} W(R) \xrightarrow{\sim} L \to P$ is equal to $\tau^\ast \beta$. 
Denote by $L_0$ and $L_1$ the images of $\beta(\Lambda_{W(R)}^0)$ and $\beta(\Lambda_{W(R)}^1)$ respectively inside $P$ under (\ref{eq-tauast}), so $P = L_0 \oplus L_1$. If we define $\Psi = F_0 \res_{L_0} \oplus F_1 \res_{L_1}$, then $\Psi$ is an $f$-linear automorphism, and $(L_0, L_1, \Psi)$ is a normal representation for $\underline{P}$. Moreover, we have an isomorphism
\begin{align}\label{eq-sigmaast}
	f^\ast P \xrightarrow{\sim} f^\ast L \xrightarrow{\sim} \sigma^\ast(L \otimes_{W(R)} W(R)^\oplus) \xrightarrow{\sim} \sigma^\ast M,
\end{align}
where the first arrow comes from applying $f^\ast$ to the inverse of (\ref{eq-tauast}). From the definition of the equivalence between 1-displays and Zink displays (see Lemma \ref{lem-windows}), the identification (\ref{eq-sigmaast}) has the property that the composition $f^\ast P \xrightarrow{(\ref{eq-sigmaast})} \sigma^\ast M \xrightarrow{F^\sharp} \tau^\ast M=P$ is equal to $\Psi^\sharp: f^\ast P \xrightarrow{\sim} P$. Hence (\ref{eq-Ubeta}) can be identified with the composition
\begin{align}\label{eq-displaybasis}
	\Lambda_{W(R)} \xrightarrow{\sim} f^\ast \Lambda_{W(R)} \xrightarrow{f^\ast \tau^\ast \beta} f^\ast P \xrightarrow{\Psi^\sharp} P \xrightarrow{\tau^\ast\beta^{-1}} \Lambda_{W(R)}.
\end{align}
By definition of $L_0$ and $L_1$, the isomorphism $\tau^\ast \beta: \Lambda \otimes_{\zz_p} W(R) \xrightarrow{\sim} P$ sends $\Lambda^0_{W(R)}$ to $L_0$ and $\Lambda^1_{W(R)}$ to $L_1$. Thus (\ref{eq-displaybasis}) implies that we have an isomorphism between $P$ and the display $\underline{P}_\beta$ given by $P_{\beta} = \Lambda \otimes_{\zz_p} W(R)$ with normal representation $(\Lambda^0_{W(R)}, \Lambda^1_{W(R)}, U_\beta \circ (\id \otimes f))$. 

Using the isomorphism $\tau^\ast \beta: \underline{P}_\beta \xrightarrow{\sim} \underline{P}$, we can extend the tensors $\underline{t}$ to $\underline{P}_\beta$. Explicitly, let $t_i'$ be the composition
\begin{align}\label{eq-ti'}
	\mathbbm{1} \xrightarrow{t_i} \mathbb{D}(\underline{P})(i) \xrightarrow{(\tau^\ast \beta)^{-1}} \mathbb{D}(\underline{P}_\beta)(i).
\end{align}
 Because $\beta \in Q_{M, \underline{u}}$, we know $\tau^\ast \beta \circ (s_i \otimes 1)_{W(R)} = u_i$. Write $\zeta_\beta$ for the isomorphism $\mathbb{D}(\underline{P}_\beta)_{W(R)/R} \xrightarrow{\sim} P_\beta$ given by (\ref{eq-zeta}). Then $\zeta_\beta(i) \circ (s_i \otimes 1)_{W(R)} = (s_i \otimes 1)_{W(R)}$. Hence by functoriality of $\zeta$ (see (\ref{eq-zetafunctorial})), we see 
 \begin{align}\label{eq-ti'si}
 	(t_i')_{W(R)} = (s_i \otimes 1)_{W(R)}.
 \end{align}
	
Let $\pi$ denote the representation $\textup{GL}(\Lambda) \to \textup{GL}(\Lambda(i))$. 
	
\begin{lemma}\label{lem-Ubeta}
	Suppose $pR=0$ and that $R$ admits a $p$-basis \'etale locally. Then in the situation described above, $\pi(U_\beta)(s_i \otimes 1)_{W(R)} = (s_i \otimes 1)_{W(R)}$. 
\end{lemma}
\begin{proof}
	By descent for Witt vectors \cite[Prop. 33]{Zink2002} we may assume $R$ admits a $p$-basis. In that case the map from $R$ to its perfect closure $R^\textup{perf}$ is faithfully flat by Lemma \ref{lem-pbasisuseful}, so we may further assume that $R$ is perfect. 
	
	Let $\underline{P}_\beta$ be defined as above, and write $\underline{P}_\beta = (P_\beta, \Fil P_\beta, F_{\beta,0}, F_{\beta,1})$. By \cite[Prop. 57]{Zink2002}, the identification $\zeta:\mathbb{D}(\underline{P}_\beta)_{W(R)/R} \xrightarrow{\sim} P_\beta$ (see (\ref{eq-zeta})) is compatible with the Frobenius. Thus the composition
	\begin{align*}
		f^\ast \mathbb{D}(\underline{P}_\beta)_{W(R)/R} \xrightarrow{\sim} \phi^\ast \mathbb{D}(\underline{P}_\beta)_{W(R)/R} \xrightarrow{\mathbb{F}_{W(R)/R}} \mathbb{D}(\underline{P}_\beta)_{W(R)/R}
	\end{align*}
	is given by $F_{\beta,0}^\sharp: f^\ast P_\beta \to P_\beta$. Moreover, $P_\beta = \Lambda \otimes_{\zz_p} W(R)$ and $F_{\beta,0}^\sharp = U_\beta \circ (\id_{\Lambda^0_{W(R)}} \oplus p\cdot \id_{\Lambda^1_{W(R)}})\circ (\id \otimes f)^\sharp$ by definition of $\underline{P}_\beta$. Thus the composition
	\begin{align*}
		\Lambda \otimes_{\zz_p} W(R) \xrightarrow{\sim} f^\ast(\Lambda \otimes_{\zz_p} W(R)) = f^\ast P_\beta \xrightarrow{F_{\beta,0}^\sharp} P = \Lambda \otimes_{\zz_p} W(R)
	\end{align*}
	is given by $U_\beta\circ(\id_{\Lambda^0_{W(R)}} \oplus p\cdot \id_{\Lambda^1_{W(R)}})$. Similarly the evaluation of the Verschiebung on $W(R)\to R$ is identified with $(p\cdot \id_{\Lambda^0_{W(R)}} \oplus \Lambda^1_{W(R)})\circ U_\beta^{-1}$ under the above identifications. 
	
	Since $s_i$ is fixed by $G$ for all $i$, we have in particular $s_i \otimes 1 \subset (\Lambda(i))^0$. Hence we can compute exactly as at the end of Proposition \ref{prop-frobeq} to obtain
	\begin{align*}
		p^{n_i}\mathbb{F}_i(s_i \otimes 1)_{W(R)} = p^{n_i} \pi(U)(s_i \otimes 1)_{W(R)}. 
	\end{align*}
	By (\ref{eq-ti'si}), $(s_i \otimes 1)_{W(R)} = (t_i')_{W(R)}$. Since $t_i'$ is Frobenius invariant, it follows that $p^{n_i}(s_i \otimes 1)_{W(R)} = p^{n_i} \pi(U)(s_i \otimes 1)_{W(R)}$. Thus $\pi(U)(s_i \otimes 1)_{W(R)} = (s_i \otimes 1)_{W(R)}$ because $W(R)$ is $p$-torsion free when $R$ is perfect.
\end{proof}

In the following lemma, we associate a $G$-display of type $\mu$ over $\underline{W}(R)$ to any nilpotent Zink display with $(\underline{s},\mu)$-structure $(\underline{P},\underline{t})$. Continue the notation  of Lemma \ref{lem-torsor}, and denote by $\alpha_{\underline{M},\underline{u}}$ the map $Q_{M,\underline{u}} \to L^+G, \ \beta \mapsto U_\beta$ defined by Lemma \ref{lem-Ubeta}.

\begin{lemma}\label{lem-Gdispmatching}
	Suppose $pR = 0$ and that $R$ admits a $p$-basis \'etale locally. Then the pair $(Q_{M,\underline{u}}, \alpha_{\underline{M},\underline{u}})$ determines a $G$-display of type $\mu$ over $\underline{W}(R)$. Moreover, if $\underline{M} = \mathscr{P}(\Lambda, \eta)$ for some Tannakian $(G,\mu)$-display $\mathscr{P}$ over $W(R)$, then evaluation on $(\Lambda, \eta)$ induces an isomorphism of $G$-displays of type $\mu$
	\begin{align*}
	(Q_{\mathscr{P}},\alpha_{\mathscr{P}}) \xrightarrow{\sim} (Q_{M,\underline{u}}, \alpha_{\underline{M},\underline{u}}).
	\end{align*}
\end{lemma}
\begin{rmk}
	Here $(Q_{\mathscr{P}},\alpha_{\mathscr{P}})$ is the $G$-display of type $\mu$ over $\underline{W}(R)$ associated with $\mathscr{P}$ as in (\ref{eq-morphism}) (see also \cite[Constr. 3.15]{Daniels2019}).
\end{rmk}
\begin{proof}
	For the first assertion, we note that if $h \in L^+_\mu G(R)$, then $U_{\beta \cdot h}$ is the composition $\tau^\ast h^{-1} \circ \tau^\ast \beta^{-1} \circ \Phi^\sharp \circ \sigma^\ast h \circ \sigma^\ast \beta$, which is equal to $\tau(h)^{-1} \cdot U_{\beta} \cdot \sigma(h)$. The second assertion follows from the proof of \cite[Lem. 5.14]{Daniels2019}.
\end{proof}

We can now prove Theorem \ref{thm-1} in the case where $pR = 0$. 

\begin{prop}\label{prop-pR=0}
	Suppose $pR=0$ and $R$ admits a $p$-basis \'etale locally. Then the functor $\textup{BT}_{\underline{G},R}$ is an equivalence.
\end{prop}
\begin{proof}
	By Remark \ref{rmk-equiv} it is enough to show the functor $\Z_{\underline{G},R}$ is an equivalence. The proof in this case is formally very similar to the proof of \cite[Thm. 5.15]{Daniels2019}.  Indeed, faithfulness of $\Z_{\underline{G},R}$ follows exactly as in \textit{loc. cit.}. Namely, the problem reduces by descent to faithfulness of the representation $\eta$. For fullness, if $\mathscr{P}$ and $\mathscr{P}'$ are Tannakian $(G,\mu)$-displays over $\underline{W}(R)$ which are nilpotent with respect to $\eta$, and $\varphi: \Z_{\underline{G},R}(\mathscr{P}) \to \Z_{\underline{G},R}(\mathscr{P}')$ is a morphism of $p$-divisible groups with $(\underline{s},\mu)$-structure, one uses Lemma \ref{lem-Gdispmatching} to obtain a morphism $(Q_{\mathscr{P}}, \alpha_{\mathscr{P}}) \to (Q_{\mathscr{P}'},\alpha_{\mathscr{P}'})$ of $G$-displays of type $\mu$ over $\underline{W}(R)$, which is induced from a unique morphism $\xi: \mathscr{P} \to \mathscr{P}'$. As in the proof of \cite[Thm. 5.15]{Daniels2019}, we have $\Z_{\eta,R}(\xi) = \xi^\eta = \varphi$.
	
	Let us now show essential surjectivity. Let $(\underline{P},\underline{s})$ be a nilpotent Zink display with $(\underline{s},\mu)$-structure over $R$. Since $\Z_{\underline{G},R}$ is fully faithful, by descent it is enough to show that $(\underline{P}, \underline{t})$ is \'etale locally in the essential image of $\Z_{\underline{G},R}$. Let $\underline{M} = M_{\underline{W}(R)}(\underline{P})$ be the 1-display corresponding to $\underline{P}$. By Lemma \ref{lem-torsor}, $Q_{M,\underline{u}}(R')$ has a section $\beta$ for some \'etale faithfully flat extension $R'$ of $R$. The composition  $U_\beta = \tau^\ast\beta^{-1} \circ F_{R'}^\sharp \circ \sigma^\ast \beta$ is an element of $L^+G(R')$ by Lemma \ref{lem-Ubeta}, and $\Z_{\eta,R'}(\mathscr{P}_{U_\beta}) = \underline{P}_\beta$, where $\underline{P}_\beta$ is the Zink display with normal representation $(\Lambda_{W(R)}^0, \Lambda_{W(R)}^1, U_\beta \circ (\id \otimes f))$ defined before Lemma \ref{lem-Ubeta}. It follows that $\beta$ determines an isomorphism $\Z_{\eta, R'}(\mathscr{P}_{U_\beta}) = \underline{P}_\beta \xrightarrow{\sim} \underline{P}_{R'}$. It remains to show the induced isomorphism of crystals
	\begin{align}\label{eq-isomofcrystals}
		\mathbb{D}(\underline{P}_\beta) \xrightarrow{\sim} \mathbb{D}(\underline{P}_{R'})
	\end{align}
	sends $t_i$ to $t_{\beta,i}:=\mathbb{D}(\mathscr{P}_{U_\beta})(s_i)$ for all $i$. Let $t_i'$ be the tensor for $\underline{P}_\beta$ induced by (\ref{eq-isomofcrystals}), see also (\ref{eq-ti'}). Since $\beta \in Q_{M,\underline{u}}$, we know that $(t_i')_{W(R)} = (s_i \otimes 1)_{W(R)}$, see (\ref{eq-ti'si}). On the other hand, it follows from Lemma \ref{lem-tensormatch} that $(t_{\beta,i})_{W(R)} = (s_i \otimes 1)_{W(R)}$. Thus we have the equality $t_i' = t_{\beta,i}$ by Lemma \ref{lem-tensorsagree}, so $\Z_{\eta,R'}(\mathscr{P}_{U_\beta}) \cong (\underline{P}, \underline{t})$.
\end{proof}

\begin{thm}\label{mainthm}
	Suppose $R$ is a $p$-nilpotent $W(k_0)$-algebra such that $R/pR$ has a $p$-basis \'etale locally. Then the functor $\textup{BT}_{\underline{G},R}$ is an equivalence. 
\end{thm}
\begin{proof}
	Let $R_0 = R/pR$. We have a commutative diagram of functors
	\begin{center}
		\begin{tikzcd}
			G\textup{-Disp}^{\otimes,\eta}_{\underline{W},\mu}(R) 
			\arrow[r, "\textup{BT}_{\underline{G},R}"] \arrow[d]
			& \textup{fpdiv}_{\underline{s},\mu}(R)
			\arrow[d]
			\\ G\textup{-Disp}^{\otimes,\eta}_{\underline{W},\mu}(R_0)
			\arrow[r, "\textup{BT}_{\underline{G},R_0}"]
			& \textup{fpdiv}_{\underline{s},\mu}(R_0).
		\end{tikzcd}
	\end{center}
	By Proposition \ref{prop-pR=0}, the bottom horizontal arrow is an equivalence. 
	
	By Proposition \ref{prop-lifting} and Proposition \ref{prop-Gdisplift}, lifts of a Tannakian $(G,\mu)$-display $\mathscr{P}$ over $\underline{W}(R_0)$ along the left-hand vertical arrow correspond to lifts of the Hodge filtration $\textup{Fil}_\mathscr{P}$ of the fiber functor $\omega_{\mathscr{P}}$ to a filtration of $\omega_{\mathscr{P},R}$. Equivalently, by the discussion at the end of \textsection \ref{sub-gdispdaniels}, lifts of $\mathscr{P}$ correspond to lifts of the Hodge filtration $Q_{\textup{Fil}_{\mathscr{P}}}$ of the corresponding $G$-display of type $\mu$ to a $P_{\mu,R}$-torsor $Q_\mu$ inside the $G_R$-torsor $Q_{\omega_{\mathscr{P}.R}}$. 
	
	On the other hand, by Lemma \ref{lem-liftspdiv}, we have an analogous description of lifts along the right-hand arrow: Lifts of a formal $p$-divisible group with $(\underline{s},\mu)$-structure $(X, \underline{t})$ over $\mu$ correspond to lifts of the $P_{\mu,R_0}$-torsor $T(X)_{R_0/R_0,\mu}$ associated with $(X,\underline{t})$ to a $P_{\mu,R}$-torsor $T_\mu$ inside the $G_R$-torsor $T(X)_{R/R_0}$. Thus $\textup{BT}_{\underline{G},R}$ will be an equivalence if we can show that the respective torsors correspond under $\textup{BT}_{\underline{G},R}$.
	
	Let $\mathscr{P}$ be a Tannakian $(G,\mu)$-display over $R$ which is nilpotent with respect to $\eta$, and write $\mathscr{P}(V,\rho) = (M^\rho, F^\rho)$ for every representation $(V,\rho)$ of $G$. Set $(X, \underline{t}) = \textup{BT}_{\underline{G},R}(\mathscr{P})$. Then the $G_R$- and $P_{\mu,R}$ torsors associated to $\mathscr{P}$ are
	\begin{align*}
		Q_{\omega_{\mathscr{P}}} = \textup{Isom}^\otimes(\omega_R, \omega_{\mathscr{P}}) \text{ and } \textup{Isom}^\otimes(\textup{Fil}_\mu, \textup{Fil}_{\mathscr{P}}),
	\end{align*}
	respectively, see (\ref{eq-GRtors}) and (\ref{eq-Pmutors}). Any $\beta \in Q_{\omega_{\mathscr{P}}}$ is an isomorphism of tensor functors, so $\beta(s_i \otimes 1)_R = \bar{\tau}^\ast\mathscr{P}(s_i)$ for every $i$. Therefore evaluation on $\eta$ induces an isomorphism of $G_R$-torsors
	\begin{align*}
		Q_{\omega_{\mathscr{P}}} \xrightarrow{\sim} \textup{Isom}((\Lambda \otimes_{\zz_p} R, (\underline{s}\otimes 1)_R), (\tau^\ast M^\eta \otimes_{W(R)} R, \bar{\tau}^\ast\mathscr{P}(\underline{s}))).
	\end{align*}
	But we have an isomorphism $\tau^\ast M^\eta \otimes_{W(R)} R = \mathbb{D}(\mathscr{P})^\eta_{R/R} \xrightarrow{\sim} \mathbb{D}(X)_{R/R}$, and under this isomorphism $\bar{\tau}^\ast(\mathscr{P}(s_i))$ is identified with $(t_i)_{R}$ by definition of $t_i$ (see (\ref{eq-ti})). Thus we have an isomorphism of $G_R$-torsors $Q_{\omega_{\mathscr{P}}} \xrightarrow{\sim} T(X)_{R/R}$.	Similarly, $Q_{\textup{Fil}_{\mathscr{P}}} \xrightarrow{\sim} T(X)_{R/R,\mu}$, since if $\beta \in Q_{\omega_{\mathscr{P}}}$ preserves the Hodge filtration of the Tannakian $(G,\mu)$-display, then its evaluation on $\eta$ will preserve the Hodge filtration of the corresponding formal $p$-divisible group. 
	
	Finally we check that $\mathbb{D}(\mathscr{P})^\eta \xrightarrow{\sim} \mathbb{D}(X)$ induces an isomorphism $Q_{\omega_{\mathscr{P},R}} \xrightarrow{\sim} T(X)_{R/R_0}$. This is similar to the case of $Q_{\omega_\mathscr{P}}$, except here evaluation on $(\Lambda, \eta)$ sends an isomorphism $\beta \in Q_{\omega_{\mathscr{P},R}}$ to an isomorphism
	\begin{align}\label{eq-R/R_0}
		(\Lambda \otimes_{\zz_p} R, (\underline{s} \otimes 1)_R) \xrightarrow{\sim} (\tilde{\tau}^\ast M_{R/R_0}^\eta \otimes_{W(R)} R, \tilde{\tau}^\ast(\mathscr{P}_{R/R_0}(\underline{s}))),
	\end{align}
	where $M_{R/R_0}^\eta$ is the evaluation of the unique lift $\mathscr{P}_{R/R_0}$  of $\mathscr{P}$ to a Tannakian $(G,\mu)$-display over $\underline{W}(R/R_0)$. By definition of $\mathbb{D}(\mathscr{P})$, the right-hand side of (\ref{eq-R/R_0}) is identified with $(\mathbb{D}(\mathscr{P})^\eta_{R/R_0}, \mathbb{D}(\mathscr{P})_{R/R_0}(\underline{s}))$, and hence with $(\mathbb{D}(X)_{R/R_0}, \underline{t}_R)$. Therefore we obtain an isomorphism of $G_R$-torsors $Q_{\omega_{\mathscr{P},R}} \xrightarrow{\sim} T(X)_{R/R_0}$, and the theorem follows.
\end{proof}

\subsection{RZ spaces of Hodge type}\label{sub-rzspaces}

In this section we give an explicit isomorphism between the Rapoport-Zink functor of Hodge type defined using Tannakian $(G,\mu)$-displays as in \cite{BP2017} and \cite{Daniels2019}, and the one defined using crystalline Tate tensors as in \cite{Kim2018} and \cite{HP2017}. We begin by recalling the definition of $G$-quasi-isogenies as in \cite{Daniels2019}, which are used to define the Rapoport-Zink functor in terms of Tannakian $(G,\mu)$-displays. 

If $R$ is a $\zz_p$-algebra, the Frobenius for $W(R)$ naturally extends to $W(R)[1/p]$. An \textit{isodisplay} over $R$ is a pair $(N,\varphi)$ where $N$ is a finitely generated projective $W(R)[1/p]$-module and $\varphi:N \to N$ is an $f$-linear isomorphism. If $\underline{M}$ is a display over $\underline{W}(R)$, then we can associate to $\underline{M}$ an isodisplay $(N,\varphi)$ using the process explained in \cite[\textsection 3.4]{Daniels2019}. Let us review the construction.

If $\underline{M} = (M, F)$ is a display over $\underline{W}(R)$ with standard datum $(L,\Phi)$, then the \textit{depth} of $\underline{M}$ is the smallest integer $d$ such that $L_d$, the $d^\textup{th}$ graded piece of $L = \bigoplus L_i$, is nonzero. By \cite[Lem. 2.7]{Daniels2019}, $d$ does not depend on the choice of normal decomposition. Moreover, by \cite[Lem. 2.8]{Daniels2019}, the natural map $\theta_n: M_n \to \tau^\ast M$ (see \textsection \ref{sub-frames}) is an isomorphism of $W(R)$-modules for all $n \le d$. 

Suppose $\underline{M}$ is a display of depth $d$. Define $\varphi = p^d \circ F_d \circ \theta_d^{-1}$. Then $N = (\tau^\ast M, \varphi)$ is an isodisplay, and the assignment $\underline{M} \mapsto (\tau^\ast M, \varphi)$
determines an exact tensor functor from displays over $\underline{W}(R)$ to isodisplays over $R$. A \textit{quasi-isogeny} of displays over $\underline{W}(R)$ is an isomorphism of their corresponding isodisplays, and a quasi-isogeny is an isogeny if it is induced from a morphism of displays. These notions naturally extend to $G$-displays. Indeed, a \textit{$G$-isodisplay} over $R$ is an exact tensor functor $\textup{Rep}_{\zz_p}(G) \to \textup{Isodisp}(R)$. Any $G$-display $\mathscr{P}$ naturally determines a $G$-isodisplay $\mathscr{P}[1/p]$ by composition of functors, and a \textit{$G$-quasi-isogeny} between two $G$-displays is an isomorphism of their corresponding $G$-isodisplays. See \cite[\textsection 3.4]{Daniels2019} for more details.

Let us now recall the definition of local Shimura data of Hodge type as in \cite{HP2017} and \cite{BP2017} (see also \cite{Daniels2019}). Let $k$ be an algebraic closure of $\mathbb{F}_p$, and let $W(k)$ be the Witt vectors over $k$. Write $K = W(k)[1/p]$, and let $\bar{K}$ be an algebraic closure of $K$. We will write $\sigma$ for the extension of the Frobenius of $W(k)$ to $K$ (hopefully this causes no confusion with the previous definition of $\sigma$).  

Assume that $G$ is a connected reductive group scheme over $\zz_p$, and let $(\{\mu\},[b])$ be a pair such that 
\begin{itemize}
	\item $\{\mu\}$ is a $G(\bar{K})$-conjugacy class of cocharacters ${\mathbb{G}_m}_{\bar{K}} \to G_{\bar{K}}$;
	\item $[b]$ is a $\sigma$-conjugacy class of elements $b \in G(K)$.
\end{itemize}
The local reflex field is the field of definition $E$ of the conjugacy class $\{\mu\}$. Because $G_{\mathbb{Q}_p}$ splits over an unramified extension of $\mathbb{Q}_p$, $E$ is a subfield of $K$ (a priori, $E \subset \bar{K}$), and by \cite[Lem. 1.1.3]{Kottwitz1984}, there is a cocharacter $\mu \in \{\mu\}$ which is defined over $E$. Moreover, we an find a representative $\mu$ which extends to an integral cocharacter defined over the valuation ring $\mathcal{O}_E$ of $E$. Note that if $k_E$ is the residue field of $\mathcal{O}_E$, then $k_E$ is finite, $\mathcal{O}_E = W(k_E)$, and $E = W(k_E)[1/p]$. 

We say the triple $(G,\{\mu\},[b])$ is a \textit{local unramified Shimura datum} if $\{\mu\}$ is minuscule and for some (or equivalently, any) integral representative $\mu$ of $\{\mu\}$, the $\sigma$-conjugacy class $[b]$ has a representative 
\begin{align*}
b \in G(W(k))\sigma(\mu)(p)G(W(k)).
\end{align*}
If these assumptions are satisfied, then we can find an integral representative $\mu$ of $\{\mu\}$ defined over $\mathcal{O}_E$ and a representative $b$ of $[b]$ such that $b = u \sigma(\mu)(p)$ for some $u \in L^+G(k)$. Such a pair $(\mu,b)$ will be called a \textit{framing pair}. 

If $(\mu,b)$ is a framing pair, then we associate to $(\mu,b)$ the \textit{framing object} $\mathscr{P}_0 := \mathscr{P}_u$ where $u \in L^+G(k)$ is the unique element such that $b = u\sigma(\mu)(p)$, and $\mathscr{P}_u$ is defined as in Proposition \ref{prop-banal}. 

\begin{Def}
	Fix a framing pair $(\mu,b)$ for $(G,\{\mu\}, [b])$, and let $\mathscr{P}_0$ be the associated framing object. The \textit{display RZ-functor} associated with $(G,\mu,b)$ is the functor on \textup{Nilp}$_{W(k)}$ which assigns to a $p$-nilpotent $W(k)$-algebra $R$ the set of isomorphism classes of pairs $(\mathscr{P}, \rho)$, where
	\begin{itemize}
		\item $\mathscr{P}$ is a Tannakian $(G,\mu)$-display over $R$,
		\item $\rho: \mathscr{P}_{R/pR} \dashrightarrow (\mathscr{P}_0)_{R/pR}$ is a $G$-quasi-isogeny.
	\end{itemize}
	Denote the display RZ-functor associated with $(G,\mu,b)$ by \textup{RZ}$^{\text{disp}}_{G,\mu,b}$. 
\end{Def}

Let $\textup{Nilp}^\text{fsm}_{W(k)}$ denote the category of adic $W(k)$-algebras in which $p$ is nilpotent, and which are formally finitely generated and formally smooth over $W(k)/p^nW(k)$ for some $n \ge 1$. We extend $\textup{RZ}_{G,\mu,b}^\textup{disp}$ to a functor $\textup{RZ}_{G,\mu,b}^\textup{disp,fsm}$ on $\textup{Nilp}^\text{fsm}_{W(k)}$ by defining 
\begin{align*}
\textup{RZ}_{G,\mu,b}^\textup{disp,fsm}(A) := \varprojlim_n \textup{RZ}_{G,\mu,b}^\textup{disp}(A /I^n),
\end{align*}
where $I$ is an ideal of definition of $A$. 

\begin{rmk}\label{rmk-RZdispfsm}
	Let $A \in \textup{Nilp}_{W(k)}^\textup{fsm}$, and suppose $I$ is an ideal of definition for $A$. Define a Tannakian $(G,\mu)$-display over $\Spf A$ to be a compatible system $(\mathscr{P}_n)_n$ of Tannakian $(G,\mu)$-displays $\mathscr{P}_n$ over $\underline{W}(A/I^n)$. Likewise, a $G$-quasi-isogeny over $\Spf A$ is a compatible system $(\rho_n)_n$ of $G$-quasi-isogenies over $A/I^n$. With these definitions, we see that $\textup{RZ}_{G,\mu,b}^\textup{disp,fsm}(A)$ is the set of isomorphism classes of pairs $((\mathscr{P}_n)_n,(\rho_n)_n)$, with $(\mathscr{P})_n$ a Tannakian $(G,\mu)$-display over $\Spf A$, and $(\rho_n)_n$ a $G$-quasi-isogeny $\mathscr{P}_{A/pA} \dashrightarrow (\mathscr{P}_0)_{A/pA}$ defined over $\Spf A/pA$. In fact, by \cite[Prop. 3.2.11]{BP2017} and \cite[Cor. 3.17]{Daniels2019}, the categories of Tannakian $(G,\mu)$-displays over $\underline{W}(A)$ and over $\Spf A$ are equivalent, so it is equivalent to consider pairs $(\mathscr{P},(\rho_n)_n)$, where $\mathscr{P}$ is a Tannakian $(G,\mu)$-display over $\underline{W}(A)$ and $(\rho_n)_n$ is a $G$-quasi-isogeny over $\Spf A/pA$.
\end{rmk}

Let $(G,\mu)$ be of Hodge type as in Definition \ref{def-hodgetype}, with Hodge embedding $\eta: G \hookrightarrow \textup{GL}(\Lambda)$. Suppose $(\mu,b)$ is a framing pair, and let $\mathscr{P}_0$ be the framing object given by $u \in L^+G(k)$, so $b = u\sigma(\mu)(p)$. Then $G$ is cut out by some collection of tensors $\underline{s}$, and $\underline{G} = (G,\mu,\Lambda,\eta,\underline{s})$ is a local Hodge embedding datum. For the remainder of this section we will assume $\mathscr{P}_0$ is nilpotent with respect to $\eta$. Then by \cite[Thm. 5.1.3]{BP2017}, the restriction of RZ$^{\textup{disp}}_{G,\mu,b}$ to Noetherian algebras in $\textup{Nilp}_{W(k)}$ is representable by a formal scheme $\textup{RZ}^\textup{BP}_{G,\mu,b}$ which is formally smooth and formally locally of finite type over $W(k)$. Applying $\textup{BT}_{\underline{G},k}$ to $\mathscr{P}_0$, we obtain a formal $p$-divisible group with $(\underline{s},\mu)$-structure 
\begin{align*}
(X_0, \underline{t_0}) = (\textup{BT}_k(\Z_k(\mathscr{P}_0)), \mathbb{D}(\mathscr{P}_0)(\underline{s})).
\end{align*}

\begin{Def}\label{def-RZpdiv}
	Let $\underline{G} = (G,\mu,\Lambda,\eta,\underline{s})$, $b$ and $(X_0,\underline{t_0})$ be as above. The \textit{$p$-divisible group RZ-functor} associated with the data $(\underline{G},b)$ is the functor on \textup{Nilp}$_{W(k)}$ which assigns to a $p$-nilpotent $W(k)$-algebra the set of isomorphism classes of triples $(X,\underline{t},\iota)$, where
	\begin{itemize}
		\item $(X,\underline{t})$ is a $p$-divisible group with $(\underline{s},\mu)$-structure,
		\item $\iota: X\otimes_R R/pR \dashrightarrow X_0 \otimes_k R/pR$ is a quasi-isogeny such that, for some nilpotent ideal $J \subset R$ with $p \in J$, the composition of $t_i$ with 
		\begin{align*}
		\mathbb{D}(\iota_{R/J}): \mathbb{D}(X\otimes_R R/J)^\otimes[1/p] \xrightarrow{\sim} \mathbb{D}(X \otimes_k R/J)^\otimes[1/p]	
		\end{align*}
		is equal to $t_{0,i}$ for every $i$.
	\end{itemize}
	Denote the $p$-divisible group RZ-functor associated with $(\underline{G},b)$ by $\textup{RZ}^{p\text{-div}}_{\underline{G},b}$.
\end{Def}

We also extend $\textup{RZ}_{\underline{G},b}^{p\textup{-div}}$ to a functor $\textup{RZ}_{\underline{G},b}^{p\textup{-div,fsm}}$ on $\textup{Nilp}^\text{fsm}_{W(k)}$ by defining 
\begin{align*}
\textup{RZ}_{\underline{G},b}^{p\textup{-div,fsm}}(A) := \varprojlim_n \textup{RZ}_{\underline{G},b}^{p\textup{-div}}(A /I^n),
\end{align*}
where once again $I$ is an ideal of definition of $A$. 
\begin{rmk}
	As in the case of the display RZ-functor, the extension to $\textup{Nilp}_{W(k)}^\textup{fsm}$ can be thought of as classifying objects over $\Spf A$. More precisely, $\textup{RZ}_{\underline{G},b}^{p\textup{-div,fsm}}(A)$ is the set of isomorphism classes of triples $(X,\underline{t},\iota)$, with $(X,\underline{t})$ a $p$-divisible group with $(\underline{s},\mu)$-structure over $\Spf A$, and $\iota = (\iota_n)_n$ a quasi-isogeny over $\Spf A/pA$ such that for every $n$, $\mathbb{D}((\iota_n)_{A/I}) \circ t_i = t_{0,i}$, for all $i$. Since $(\iota_n)_n$ is a compatible system, it is equivalent to assume $\mathbb{D}(\iota_1) \circ t_i = t_{0,i}$ for all $i$. If $I$ is chosen with $p \in I$, then by rigidity of quasi-isogenies along with \cite[Lem. 2.4.4]{deJong1996} and the proof of \cite[Prop. 2.4.8]{deJong1996}, elements of $\textup{RZ}_{\underline{G},b}^{p\textup{-div,fsm}}(A)$ correspond to triples $(X, \underline{t}, \iota)$, with $(X,\underline{t})$ a $p$-divisible group with $(\underline{s},\mu)$-structure over $\Spec A$ and 
	\begin{align*}
	\iota: X \otimes_A A/I \dashrightarrow X_0 \otimes_k A/I
	\end{align*}	
	a quasi-isogeny such that $\mathbb{D}(\iota) \circ t_i = t_{0,i}$ for all $i$ (see \cite[\textsection 2.3.6]{HP2017} for details).
\end{rmk}

Suppose $(\mathscr{P},\rho) \in \textup{RZ}^{\text{disp}}_G(R)$ for $R \in \textup{Nilp}_{W(k)}$. Let us write 
\begin{align*}
\mathscr{P}_{R/pR}(\Lambda, \eta) = \underline{M}^\eta \text{ and } (\mathscr{P}_0)_{R/pR}(\Lambda, \eta) = \underline{M_0}^\eta.
\end{align*}
By evaluating $\rho$ on $(\Lambda, \eta)$, we obtain a quasi-isogeny of 1-displays 
\begin{align*}
\underline{M}^\eta \dashrightarrow \underline{M_0}^\eta.
\end{align*}
By \cite[Prop. 66]{Zink2002}, such a quasi-isogeny is equivalent to an invertible section of \begin{align*}
\textup{Hom}_{\text{Disp}(\underline{W}(R/pR))}(\underline{M}^\eta,\underline{M_0}^\eta)[1/p],
\end{align*}
so the functor $\textup{BT}$ induces a quasi-isogeny of $p$-divisible groups 
\begin{align*}
\iota_\rho: \textup{BT}(\Z_{\eta}(\mathscr{P}_{R/pR})) \dashrightarrow \textup{BT}(\Z_{\eta}((\mathscr{P}_0)_{R/pR})).
\end{align*}
If $(\rho_n)_n$ is a $G$-quasi-isogeny defined over $\Spf A/pA$ for $A \in \textup{Nilp}_{W(k)}^\textup{fsm}$ with ideal of definition $I$ containing $p$, then by taking $R = A/I^n$ as $n$ varies we obtain a quasi-isogeny of $p$-divisible groups $(\iota_{\rho_n})_n$ defined over $\Spf A/pA$.
\begin{lemma}\label{lem-RZnattrans}
	The assignment
	\begin{align*}
	(\mathscr{P}, (\rho_n)_n) \mapsto (\textup{BT}_{\underline{G},R}(\mathscr{P}), (\iota_{\rho_n})_n)
	\end{align*}
	determines a natural transformation $\Psi:\textup{RZ}^{\textup{disp,fsm}}_{G,\mu,b} \to \textup{RZ}^{p\textup{-div,fsm}}_{\underline{G},b}$ of functors on \textup{Nilp}$_{W(k)}^\textup{fsm}$. 
\end{lemma} 
\begin{proof}
	Let $\mathscr{P}$ be a Tannakian $(G,\mu)$-display over $\underline{W}(A)$, and let $(\rho_n)_n$ be a $G$-quasi-isogeny over $\Spf A$ (cf. Remark \ref{rmk-RZdispfsm}). Suppose $I$ is an ideal of definition for $A$ with $p \in I$, and let $R = A/I$. Write $\rho = \rho_1$, and $\iota = \iota_{\rho}$, so $\iota = (\iota_{\rho_n})_{R}$ for every $n$. We need to show $\mathbb{D}(\iota) \circ t_i = t_{0,i}$ for every $i$. We claim it is enough to show this after evaluation on $W(R) \to R$. Indeed, by \cite[Lem. 3.2.8 and its proof]{HP2017}, since $R$ is finitely generated over $k$, it is enough to check the identity holds at a closed point in each connected component of $\Spec R$ (see also \cite[Rmk. 2.3.5 (d)]{HP2017}). But any field $k'$ of characteristic $p$ has a $p$-basis, so if the identity holds after evaluation on $W(k') \to k'$, then it holds over $\Spec k'$ by Lemma \ref{lem-tensorsagree}. 	
 
	Since $\rho$ is a natural transformation of functors $\rho: \mathscr{P}[1/p] \to (\mathscr{P}_0)_{R}[1/p]$, we have an identification
	\begin{align}\label{eq-identity0}
	\rho^{\eta(i)} \circ \mathscr{P}[1/p](s_i) =  (\mathscr{P}_0)_{R}[1/p](s_i).
	\end{align}
	Moreover, by definition, we have $\mathscr{P}[1/p](s_i) = \tau^\ast \mathscr{P}(s_i)$, so (\ref{eq-identity0}) can be rewritten as
	\begin{align}\label{eq-identity}
	\rho^{\eta(i)} \circ \tau^\ast \mathscr{P}(s_i) = \tau^\ast(\mathscr{P}_0)_{R}(s_i).
	\end{align}
	Recall the isomorphism $\zeta(i): \mathbb{D}(\mathscr{P})_{W(R)/R}^{\eta(i)} \xrightarrow{\sim}  \tau^\ast M^{\eta(i)}$ (see (\ref{eq-zetai})). We will write $\zeta_0(i)$ for the analogous isomorphism defined for $\mathscr{P}_0$. By Lemma \ref{lem-tensormatch}, $\tau^\ast \mathscr{P}(s_i) = \zeta(i) \circ (t_i)_{W(R)}$, and $\tau^\ast (\mathscr{P}_0)_R(s_i) = \zeta_0(i) \circ (t_{0,i})_{W(R)}$. Moreover, the isomorphism $\mathbb{D}(\underline{P})\xrightarrow{\sim} \mathbb{D}(\textup{BT}(\underline{P}))$ from Lemma \ref{lem-zinkdieudonne} is functorial in $\underline{P}$, so we can identify $\mathbb{D}(\iota)$ with $\mathbb{D}(\rho^\eta)$. Thus it is enough to show $\zeta_0 \circ \mathbb{D}(\rho^\eta) = \rho^{\eta} \circ \zeta$. But this follows immediately from the functoriality of $\zeta$, see (\ref{eq-zetafunctorial}).
\end{proof}

\begin{thm}\label{thm-RZfunctors}
	The natural transformation $\Psi:\textup{RZ}_{G,\mu,b}^{\textup{disp,fsm}} \to \textup{RZ}_{\underline{G},b}^{p\textup{-div,fsm}}$ defined in Lemma \ref{lem-RZnattrans} is an isomorphism of functors on $\textup{Nilp}_{W(k)}^{\textup{fsm}}$.
\end{thm}
\begin{proof}
	This is formally similar to \cite[Thm. 5.15]{Daniels2019}. If $A$ is in $\textup{Nilp}_{W(k)}^\textup{fsm}$, then the $\mathbb{F}_p$-algebra $A/pA$ satisfies condition (1.3.1.1) of \cite{deJong1996}, so in particular it is Noetherian, $F$-finite, and formally smooth over $\mathbb{F}_p$. Hence by \cite[Lem. 2.1]{Lau2018b} $A/pA$ has a $p$-basis \'etale (even Zariski) locally.  Thus for any $A$ in $\textup{Nilp}_{W(k)}^\textup{fsm}$, $\Psi_A$ is injective by full-faithfulness of $\textup{BT}_{\underline{G},A}$. For surjectivity, suppose $(X, \underline{t}, (\iota_n)_n) \in \textup{RZ}_{\underline{G},b}^{p\textup{-div,fsm}}(A)$ for $A \in \textup{Nilp}^{\textup{fsm}}_{W(k)}$. Then by Theorem \ref{mainthm} there exists a Tannakian $(G,\mu)$-display $\mathscr{P}$ over $\underline{W}(A)$ such that $\textup{BT}_{\underline{G},A}(\mathscr{P}) \cong (X,\underline{t})$. It remains to define a $G$-quasi-isogeny $(\rho_n)_n$ over $\Spf A$. 
	
	Choose an ideal of definition $I$ with $p \in I$, fix $n$, and let $A_n = A/((p)+ I^n)$. Consider the $G$-quasi-isogeny $\iota_n: X\otimes_A A_n \dashrightarrow X_0 \otimes_k A_n$.
	By the second condition in Definition \ref{def-RZpdiv}, we have $\mathbb{D}((\iota_n)_{A/I}) \circ t_i = t_{0,i}$. Moreover, by \cite[Lem. 4.6.3]{Kim2018}, such an identity lifts along a quotient by a nilpotent ideal, so we obtain
	\begin{align}\label{eq-RZthm}
	\mathbb{D}(\iota_n)\circ t_i = t_{0,i}.
	\end{align}
	By descent it is enough to define the $G$-quasi-isogeny \'etale locally. After an \'etale faithfully flat extension, $\mathscr{P}$ is banal, with trivialization $\psi:\mathscr{P}_U \xrightarrow{\sim} \mathscr{P}$ for some $U \in L^+G(A)$. For every $n$ denote by $\mathscr{P}_n$ the base change of $\mathscr{P}$ to $A_n$, let $U_n$ denote the image of $U$ in $L^+G(A_n)$, and let $\psi_n: \mathscr{P}_{U_n} \xrightarrow{\sim} \mathscr{P}_n$ be the trivialization obtained by base change. By \cite[Thm. 4.7]{Daniels2019}, a $G$-quasi-isogeny $(\mathscr{P}_U)_{A_n/pA_n} = \mathscr{P}_{U_n} \dashrightarrow (\mathscr{P}_0)_{A_n / pA_n}$ is given by $g_n \in G(W(A_n))[1/p]$ such that $U_n \sigma(\mu(p)) = g_n^{-1} b f(g_n)$ in $G(W(A_n)[1/p])$.

	By Lemma \ref{lem-Dpsi}, the trivialization $\psi$ lifts to an isomorphism of tensor functors $\Psi : \omega_{W(A)} \xrightarrow{\sim} \mathbb{D}(\mathscr{P})_{W(A)/A}$ such that $\zeta \circ \Psi^\eta = \tau^\ast \psi^\eta$. Since $W(A) \to W(A_n)$ preserves the divided powers, there is an isomorphism 
	\begin{align*}
		\mathbb{D}(\mathscr{P}_n)_{W(A_n)/A_n} = \mathbb{D}(\mathscr{P})_{W(A)/A} \otimes_{W(A)} W(A_n)
	\end{align*}
	given by the crystal property for $\mathbb{D}(\mathscr{P})$. Thus base change along $A \to A_n$ induces $\Psi_n: \omega_{W(A_n)} \xrightarrow{\sim}  \mathbb{D}(\mathscr{P}_n)_{W(A_n)/A_n}$ such that $\zeta_n \circ \Psi_n^\eta= \tau^\ast \psi_n$. By uniqueness, $\Psi_n$ is the isomorphism associated to $\psi_n$ by Lemma \ref{lem-Dpsi}. 
	
	Let $g_n$ denote the composition
	\begin{align*}
		\Lambda \otimes_{\zz_p} W(A_n) [1/p] \xrightarrow{\Psi_n^\eta} \mathbb{D}(\mathscr{P})^\eta_{W(A_n)/A_n}[1/p] \xrightarrow{\mathbb{D}(\iota_n)_{W(A_n)}} \mathbb{D}(\mathscr{P}_0)^\eta_{W(A_n)/A_n}[1/p] = \Lambda \otimes_{\zz_p} W(A_n) [1/p].
	\end{align*}
	Then $g_n \in \textup{GL}(\Lambda \otimes_{\zz_p} W(A_n)[1/p])$. We claim $g_n \in G(W(A_n)[1/p])$ and $U\sigma(\mu(p)) = g_n^{-1}bf(g_n)$. 
	
	For the first claim, we note that by Lemma \ref{lem-tensormatch}, $\Psi_n^\eta \circ (s_i \otimes 1)_{W(A_n)} = \mathbb{D}(\mathscr{P_n})(s_i)$. Moreover, since $\textup{BT}_{\underline{G},A}(\mathscr{P}) = (X, \underline{t})$, we know $\mathbb{D}(\mathscr{P})(s_i) = t_i$. Hence it follows from the identity (\ref{eq-RZthm}) and the definition of $g_n$ that $g_n(s_i \otimes 1)_{W(A_n)[1/p]} = (s_i \otimes 1)_{W(A_n)[1/p]}$. 
	
	Now let $\mathscr{P}_n(\Lambda, \eta) = (M_n^\eta, F_n^\eta)$, and let $\mathbb{F}_{W(A_n)}'$ be the homomorphism $f^\ast (\Lambda \otimes_{\zz_p} W(A_n)[1/p]) \to \Lambda \otimes_{\zz_p} W(A_n)[1/p]$ induced by the Frobenius $\mathbb{F}_{W(A_n)}$ for $\mathbb{D}(X)$ via $\Psi_n^\eta$. By \cite[Prop. 57]{Zink2002}, the isomorphism $\zeta_n: \mathbb{D}(\mathscr{P}_n)^{\eta}_{W(A_n)/A_n} \xrightarrow{\sim} \tau^\ast M_n^\eta$ is compatible with Frobenius, so $\mathbb{F}_{W(A_n)}'$  is identified with the Frobenius for $\mathscr{P}_{U_n}$ on $(\Lambda, \eta)$, which is given by $U_n \sigma(\mu(p)) \circ (\id \otimes f)^\sharp$ by \cite[Lem. 3.27]{Daniels2019}. Similarly, the Frobenius for $\mathbb{D}(\mathscr{P}_0)^\eta_{W(A_n)/A_n}$ is identified with $b \circ (\id \otimes f)^\sharp$. Thus the identity $U_n\sigma(\mu(p)) = g_n^{-1}bf(g_n)$ follows from the fact that $\mathbb{D}(\iota_n)$ is a morphism of $F$-isocrystals.
	
	The collection $(g_n)_n$ is compatible as $n$ varies because the same is true for $\mathbb{D}(\iota_n)$, and because $\Psi_n$ is induced by base change of $\Psi$ along $A \to A_n$. Let $\rho_n$ be the isogeny induced by $g_n$ for each $n$; thus $(\mathscr{P},(\rho_n)_n) \in \textup{RZ}_{G,\mu,b}^{\textup{disp, fsm}}$. It remains to show that $\rho_n$ induces $\iota_n$ for each $n$. For this it is enough to show $\rho_n^\eta$ and $\Phi(\iota_n)$ define the same quasi-isogeny of $1$-displays (here $\Phi$ is Lau's functor (\ref{eq-laufunctor})). By definition of $\rho_n$, we have $\rho_n^\eta \circ \zeta = \zeta_0 \circ \mathbb{D}(\iota_n)_{W(A_n)}$, where $\zeta_0$ is the analog of $\zeta$ for $\Z_\eta(\mathscr{P}_0)$. On the other hand, by functoriality of $\zeta$ (see (\ref{eq-zetafunctorial})), we have $\Phi(\iota_n) \circ \zeta = \zeta_0 \circ \mathbb{D}(\iota_n)_{W(A_n)}$. Thus $\rho_n = \Phi(\iota_n)$ for all $n$, and the result follows. 
\end{proof}

\begin{rmk}\label{rmk-contravariant}
	The functor in Definition \ref{def-RZpdiv} is formulated using covariant Dieudonn\'e theory, hence it differs slightly from those of \cite{Kim2018} and \cite{HP2017} which are formulated using contravariant Dieudonn\'e theory. In fact, the difference is purely aesthetic, and the functors are isomorphic. Indeed, if $(G,\mu,\Lambda,\eta,\underline{s})$ is a local Hodge embedding datum in our sense, then the embedding $\eta^\vee: G \hookrightarrow \textup{GL}(\Lambda^\vee)$ determines a local Hodge embedding datum for $(G,\{\mu\},[b])$ in the sense of \cite[Def. 2.2.3]{HP2017}. It follows that $(G,b,\mu,\Lambda^\vee)$ is a local unramified Shimura-Hodge datum in the sense of \textit{loc. cit.}, and $X_0^D$ is the unique $p$-divisible group over $k$ associated with this datum by \cite[Lem. 2.2.5]{HP2017}. Moreover, the contravariant Dieudonn\'e crystal of a $p$-divisible group $X$ is given by the covariant Dieudonn\'e crystal of the Serre dual $X^D$ of $X$, and under this relationship the respective Hodge filtrations are identified. Hence the assignment $(X,\iota) \mapsto (X^D, \iota^D)$ provides the isomorphism between our $p$-divisible group RZ-functor and that of \cite{Kim2018} and \cite{HP2017}. 
\end{rmk}

The main theorem of \cite{Kim2018} states that there is a formal scheme $\textup{RZ}_{\underline{G},b}$ over Spf $W(k)$ which is formally smooth and formally locally of finite type which represents $\textup{RZ}_{\underline{G},b}^\textup{fsm}$ in the sense that
\begin{align}\label{eq-rep}
\textup{RZ}_{\underline{G},b}^\textup{fsm}(A) = \textup{Hom}_{\textup{Spf }W(k)}(\textup{Spf }A, \textup{RZ}_{\underline{G},b})
\end{align}
for $A \in \textup{Nilp}^\text{fsm}_{W(k)}$.
\begin{cor}\label{cor}
	The formal schemes $\textup{RZ}^{\textup{BP}}_{G,\mu,b}$ and $\textup{RZ}_{\underline{G},b}$ are isomorphic. 
\end{cor}
\begin{proof}
	By the results of \cite{Kim2018}, $\textup{RZ}_{\underline{G},b}$ is the unique formally smooth and locally formally of finite type formal scheme over $\textup{Spf }W(k)$ representing the functor $\textup{RZ}_{\underline{G},b}^\textup{fsm}$ on $\textup{Nilp}^\text{fsm}_{W(k)}$ in the sense of (\ref{eq-rep}). But by Theorem \ref{thm-RZfunctors} the same is true of $\textup{RZ}^{\textup{BP}}_{G,\mu,b}$.
\end{proof}

\begin{rmk}
	Corollary \ref{cor} is also known by \cite[Rmk. 5.2.7]{BP2017}. However, in \textit{loc. cit.} no explicit isomorphism is given between the respective RZ-functors.
\end{rmk}

\subsection{Deformations}\label{sub-deformations}
Let $G$ be a reductive group scheme over $\zz_p$, and let $\mu$ be a minuscule cocharacter of $G$ defined over $W(k_0)$. In this section we want to study the infinitesimal deformation theory of $p$-divisible groups with $G$-structure over $k$. We begin by reviewing the deformation theory of adjoint nilpotent Tannakian $(G,\mu)$-displays as in \cite[\textsection 3.5]{BP2017}, so fix a Tannakian $(G,\mu)$-display $\mathscr{P}_0$ which is adjoint nilpotent over $k$. Let $\textup{Art}_{W(k)}$ denote the category of augmented local Artin $W(k)$-algebras, i.e., the category of local artin $W(k)$-algebras $(R,\mathfrak{m})$ together with a fixed isomorphism $R/\mathfrak{m} \xrightarrow{\sim} k$. Such a ring is necessarily a $p$-nilpotent $W(k)$-algebra.  

Let $\mathfrak{Def}(\mathscr{P}_0)$ denote the functor on $\textup{Art}_{W(k)}$ which assigns to $R \in \textup{Art}_{W(k)}$ the set of isomorphism classes of pairs $(\mathscr{P}, \delta)$ where $\mathscr{P}$ is a Tannakian $(G,\mu)$-display over $R$ and $\delta: \mathscr{P}_k \xrightarrow{\sim} \mathscr{P}_0$ is an isomorphism of Tannakian $(G,\mu)$-displays over $\underline{W}(k)$. An isomorphism between pairs $(\mathscr{P},\delta)$ and $(\mathscr{P}', \delta')$ is an isomorphism $\psi: \mathscr{P} \xrightarrow{\sim} \mathscr{P}'$ such that $\delta' \circ \psi_k = \delta$. We will usually omit the fixed isomorphism $\delta$ and refer to the pair $(\mathscr{P},\delta)$ simply as $\mathscr{P}$.

By \cite[3.5.9]{BP2017}, $\mathfrak{Def}(\mathscr{P}_0)$ is prorepresentable by a power series ring over $W(k)$. Let us summarize the theory and describe the universal deformation. Denote by $U_G^\circ$ the opposite unipotent subgroup of $G$ defined by $\mu$. By \cite[Lem. 6.3.2]{Lau2018} (see also \cite[Lem. A.0.5]{BP2017}), there exists a unique $\mathbb{G}_m$-equivariant isomorphism of schemes
\begin{align*}
\textup{log}: U_G^\circ \xrightarrow{\sim} V(\textup{Lie }U_G^\circ)
\end{align*}
which induces the identity on Lie algebras. Moreover, log is an isomorphism of $W(k_0)$-group schemes. Since $U_G^\circ$ is smooth (see e.g. \cite[Thm. 4.1.17]{Conrad2014}), $\textup{Lie }U_G^\circ$ is finite and free as a $W(k_0)$-module, so after a choice of basis log induces an isomorphism of $W(k_0)$-group schemes
\begin{align*}
\textup{log}: U_G^\circ \xrightarrow{\sim} \mathbb{G}_a^\ell,
\end{align*}
where $\ell$ is the dimension of $U_G^\circ$. Let $\textup{Spf}(R_G)$ be the formal completion of $U_G^\circ \otimes_{W(k_0)} W(k)$ at the origin, and note that we have a (non-canonical) isomorphism $R_G \cong W(k)[[t_1, \dots, t_\ell]]$. If $w \in R_G$, denote by $[w]$ the Teichm\"uller lift of $w$ in $W(R_G)$, so $[w] = (w,0,\dots)$ in the usual Witt vector coordinates. Define the element $h_G^\text{univ} \in U_G^\circ(W(R_G))$ to be the unique element such that
\begin{align*}
\textup{log}(h_G^{\text{univ}}) = ([t_1], \dots, [t_\ell]) \in \mathbb{G}_a^\ell(W(R_G)).
\end{align*}
Since $k$ is algebraically closed, $\mathscr{P}_0$ is banal, given by some $u_0 \in L^+G(k)$, and the inclusion $W(k) \hookrightarrow R_G$ allows us to view $u_0$ as an element of $L^+G(R_G)$. Define
\begin{align*}
u_G^\textup{univ} := (h_G^{\textup{univ}})^{-1} u_0 \in L^+G(R_G),
\end{align*}
and let $\mathscr{P}^\textup{univ}$ denote the Tannakian $(G,\mu)$-display over $\underline{W}(R_G)$ defined by $u_G^\textup{univ}$. By the results of \cite[3.5.9]{BP2017}, the ring $R_G$ prorepresents $\mathfrak{Def}(\mathscr{P}_0)$, and $\mathscr{P}^\textup{univ}$ defines the universal deformation of $\mathscr{P}_0$ over $R_G$.

If $G = \textup{GL}_h$, $\mu = \mu_{d,h}$, and $\mathscr{P}_0$ corresponds to a nilpotent Zink display $\underline{P_0}$, then this recovers the deformation theory of \cite[\textsection 2.2]{Zink2002} because any lift of $\underline{P_0}$ to a Zink display over a local Artin $W(k)$-algebra is nilpotent by \cite[Lem. 21]{Zink2002}. 

Suppose now $\underline{G} = (G,\mu,\Lambda,\eta,\underline{s})$ is a local Hodge embedding datum, and suppose that we can choose a basis for $\Lambda_{W(k_0)}$ such that $\eta \circ \mu = \mu_{d,h}$. Let $R_{\textup{GL}} := R_{\textup{GL}(\Lambda)}$, so that $\textup{Spf}(R_{\textup{GL}})$ is the formal completion of $U^\circ_{\textup{GL}(\Lambda)} \otimes_{W(k_0)} W(k)$ at the origin, where $U^\circ_{\textup{GL}(\Lambda)}$ is the opposite unipotent subgroup of $\textup{GL}(\Lambda)$ defined by $\eta\circ\mu$. Then $U^\circ_G \hookrightarrow U^\circ_{\textup{GL}(\Lambda)}$ induces a surjection $R_\textup{GL} \to R_G$, which we denote by $\pi$. Notice that $R_\textup{GL}$ is non-canonically isomorphic to the power series ring $W(k)[[t_1, \dots, t_{d(h-d)}]]$. We choose coordinates for $R_G$ so that $R_G \cong R/(t_{r+1}, \dots, t_{d(h-d)})$.

Let $\mathscr{P}_0$ be a Tannakian $(G,\mu)$-display over $k$ which is nilpotent with respect to $\eta$, and write $\underline{P_0} = \mathscr{P}_0(\Lambda,\eta)$ for the associated Zink display over $k$. Write $\mathfrak{Def}(\underline{P_0})$ for the deformation functor of $(\textup{GL}(\Lambda),\eta \circ \mu)$-displays for $\underline{P}_0$. Then by the above paragraph $\mathfrak{Def}(\underline{P_0})$ is prorepresentable by $R_\textup{GL}$, with universal deformation $\underline{P}^\textup{univ}$ having standard representation $(\Lambda^0\otimes_{W(k_0)} W(R_\textup{GL}), \Lambda^1\otimes_{W(k_0)} W(R_\textup{GL}), \Phi^\textup{univ}_\textup{GL})$, where
\begin{align*}
\Phi^\textup{univ}_{\textup{GL}} = (h^\textup{univ}_\textup{GL})^{-1} \eta(u_0) \circ (\id_\Lambda \otimes f).
\end{align*}	 
\begin{lemma}\label{lem-univcompat}
	If $\mathscr{P}^\textup{univ}$ is the universal deformation of $\mathscr{P}_0$ as a Tannakian $(G,\mu)$-display over $\underline{W}(R_G)$, then 
	\begin{align*}
	\mathscr{P}^\textup{univ}(\Lambda, \eta) = (\underline{P}^\textup{univ})_{R_G}.
	\end{align*}
\end{lemma}
\begin{proof}
	Given the explicit descriptions of the universal deformations above, it is enough to show 
	\begin{align*}
	W(\pi)(h_\textup{GL}^\textup{univ}) = \eta(h_G^\textup{univ}).
	\end{align*}
	The embedding $\eta$ induces an embedding $U_G^\circ \hookrightarrow U_{\textup{GL}}^\circ$, which we also denote by $\eta$, as well as a map $d\eta:\textup{Lie }U_G^\circ \to \textup{Lie }U_\textup{GL}^\circ$. 
	With the above choice of coordinates, 
	\begin{align*}
	W(\pi)([t_1],\dots,[t_{d(h-d)}]) = ([t_1], \dots, [t_r], 0, \dots 0) = d\eta([t_1], \dots, [t_r]).
	\end{align*}
	From the explicit description of the log map given in \cite[Lem. 6.3.2]{Lau2018} we see that $\textup{log} \circ \eta = d\eta \circ \textup{log}$ as maps $U_G^\circ \to V(\textup{Lie }U_{\textup{GL}}^\circ)$. 
	Hence $W(\pi)(h_\textup{GL}^\textup{univ})$ and $\eta(h_G^\textup{univ})$ agree after applying log, so the result follows because log is an isomorphism.	
\end{proof}

Let $(X_0, \underline{t_0})$ be a formal $p$-divisible group with $(\underline{s},\mu)$-structure over $k$, and let $\mathscr{P}_0$ be the Tannakian $(G,\mu)$-display over $\underline{W}(k)$ corresponding to $(X_0, \underline{t_0})$ by Theorem \ref{mainthm}. We will apply our results to the deformation theory of $(X_0, \underline{t_0})$. Denote by $\mathfrak{Def}(X_0)$ the functor of deformations of the $p$-divisible group $X_0$, so for $R \in \textup{Art}_{W(k)}$, $\mathfrak{Def}(X_0)(R)$ is the set of isomorphism classes of $p$-divisible groups $X$ over $R$ together with an isomorphism $X \otimes_R k \cong X_0$. If $\underline{P_0}$ is the nilpotent Zink display corresponding to $X_0$, then by the equivalence of Zink and Lau (or by \cite[Cor. 4.8(i)]{Illusie1985}) it follows that $R_{\textup{GL}}$ prorepresents $\mathfrak{Def}(X_0)$ with universal deformation given by $\textup{BT}_{R_\textup{GL}}(\underline{P}^\textup{univ})$ over $R_\textup{GL}$.

\begin{cor}\label{thm-def}
	Let $R \in \textup{Art}_{W(k)}$, such that $R/pR$ admits a $p$-basis \'etale locally, and choose a $p$-divisible group $X$ over $R$ which lifts $X_0$. Let $\varpi: R_{\textup{GL}} \to R$ be the homomorphism induced by $X$. Then $\varpi$ factors through $R_G$ if and only if there exists an $(\underline{s},\mu)$-structure on $\mathbb{D}(X)$ lifting the one on $\mathbb{D}(X_0)$. 
\end{cor}
\begin{proof}
	First note that $X$ is infinitesimal since the same is true for $X_0$, and the property can be checked at geometric points in characteristic $p$ (see \cite[II Prop. 4.4]{Messing1972}). Let $\underline{P}$ be the nilpotent Zink display associated with $X$, so $\underline{P} = \varpi^\ast \underline{P}^\textup{univ}$. 
	
	The result will follow from Theorem \ref{mainthm} if we can show that $\varpi$ factors through $R_G$ if and only if $\underline{P} \cong \mathscr{P}(\Lambda, \eta)$ for some Tannakian $(G,\mu)$-display $\mathscr{P}$ over $\underline{W}(R)$. If $\varpi$ factors as $\nu \circ \pi$ for some $\nu: R_G \to R$, then $\underline{P} = \varpi^\ast \underline{P}^\textup{univ} \cong \nu^\ast (\pi^\ast \underline{P}^\textup{univ})$. But then by Lemma \ref{lem-univcompat} we have $\underline{P}\cong(\nu^\ast \mathscr{P}^\textup{univ})(\Lambda, \eta)$. Conversely, if $\underline{P} \cong \mathscr{P}(\Lambda, \eta)$ for some Tannakian $(G,\mu)$-display $\mathscr{P}$, then $\mathscr{P}$ is a deformation of $\mathscr{P}_0$, so there is some $\nu: R_G \to R$ such that $\mathscr{P} = \nu^\ast \mathscr{P}^\textup{univ}$. Then again Lemma \ref{lem-univcompat} implies that $\underline{P} \cong \nu^\ast\pi^\ast \underline{P}^\textup{univ}$, so $\nu \circ \pi = \varpi$ by prorepresentability of $R_\textup{GL}$ and universality of $\underline{P}^\textup{univ}$. 
\end{proof}

\appendix
\section{Descent}\label{section-appendix}
\subsection{Semi-frames and Witt vectors}
For our purposes we find it useful to develop a slightly weaker notion than that of a frame, which we call a semi-frame.
\begin{Def} \label{def-quasiframe}
	A \textit{semi-frame} is a pair $\underline{S} = (S,\tau)$, where $S$ is a $\zz$-graded ring
	\begin{align*}
	S = \bigoplus_{n \in \zz} S_n
	\end{align*}
	and $\tau: S \to S_0$ is a ring homomorphism, such that the following conditions hold:
	\begin{itemize}
		\item The endomorphism $\tau_0$ of $S_0$ is the identity, and $\tau_{-n}: S_{-n} \to S_0$ is a bijection for all $n \ge 1$.
		\item The image of $S_1$ under $\tau$ is contained in the Jacobson radical of $S_0$, \text{Rad}$(S_0)$.
	\end{itemize}
	We say $(S,\tau)$ is a semi-frame for $R = S_0 / \tau(S_1)$. 
\end{Def}

As in \textsection \ref{sub-frames}, we write $\tau(S_1) = tS_1$ since $\tau$ acts on $S_1$ as multiplication by $t$. 

\begin{rmk}\label{rmk-semiframe}
	As in \cite[\textsection 2.1]{Daniels2019} we note that a semi-frame is equivalent to a pair $(\bigoplus_{n\ge 0} S_n, (t_n)_{n\ge 0})$ where $S_{\ge 0}$ is a $\zz_{\ge 0}$-graded ring and $(t_n)_{n\ge 0}$ is a collection of $S_{\ge 0}$-linear maps $t_n: S_{n+1} \to S_n$ such that $t_0(S_1) \subseteq \text{Rad}(S_0)$. 
\end{rmk}

\begin{lemma}\label{lem-frame}
	Let $\underline{S} = (S,\sigma,\tau)$ be a frame. Then $tS_1 \subseteq \textup{Rad}(S_0)$. 
\end{lemma}
\begin{proof}
	This is proved as part of \cite[Lemma 3.1.1]{Lau2018}. Let us repeat the proof here. Let $a \in tS_1$. Then $\sigma(a) \in pS_0$. Since $\sigma$ lifts the $p$-power Frobenius of $S_0/pS_0$, it follows that $a^p \in pS_0.$ But $p \in \textup{Rad}(S_0)$ by assumption, so $a \in \textup{Rad}(S_0)$ as well. 
\end{proof}

It follows from Lemma \ref{lem-frame} that the assignment $(S,\sigma, \tau) \mapsto (S,\tau)$ defines a forgetful functor from the category of frames to the category of semi-frames. The following lemma provides a way to check that certain quotients of frames are semi-frames.

\begin{lemma} \label{lem-qfcheck}
	Let $S'$ be a $\zz$-graded ring $S' = \bigoplus_{n \in \zz} S_n'$ and $\tau'$ be a ring homomorphism $\tau': S' \to S_0'$ such that the pair $(S',\tau')$ satisfies the first bullet in Definition \ref{def-quasiframe}. If there exists a frame $(S,\sigma,\tau)$ and a surjective homomorphism of graded rings $\varphi: S\to S'$ such that $\tau' \circ \varphi = \varphi \circ \tau$, then $\tau'(S_1') \subseteq \textup{Rad}(S_0')$, i.e., $(S',\tau')$ is a semi-frame.
\end{lemma}
\begin{proof}
	Since $S \to S'$ is surjective, the image of $\textup{Rad}(S_0)$ is contained in $\textup{Rad}(S_0')$. Let $t'$ be the unique element in $S_{-1}'$ with $\tau'(t') = 1$. Then $\varphi(t) = t'$, and surjectivity of $\varphi$ implies that $\varphi(tS_1) = t'S_1'$. Therefore, by Lemma \ref{lem-frame},
	\begin{align*}
	t'S_1' = \varphi(tS_1) \subseteq \varphi(\textup{Rad}(S_0)) \subseteq \textup{Rad}(S_0').
	\end{align*}
\end{proof}

Let $R$ be a ring. Then for every $m \ge 1$ we attach to $R$ the ring of $m$-truncated $p$-typical Witt vectors $W_m(R)$. These rings are equipped with a Frobenius $f_{m,R}: W_{m+1}(R) \to W_{m}(R)$ which is a ring homomorphism, and a Verschiebung $v_{m,R}: W_{m}(R) \to W_{m+1}(R)$ which is additive. We will suppress the subscripts on the Frobenius and the Verschiebung when $m$ and $R$ are clear from context. 

Let $I_m(R)= v(W_{m-1}(R)) = \ker(W_m(R)\to R)$, and let $I(R)= v(W(R)) \subseteq W(R)$. For every finite $m$, the truncation map $r_m: W(R)\to W_m(R)$ induces an isomorphism $W(R)/ v^m(W(R)) \cong W_m(R),$ and these combine to give an isomorphism 
\begin{align*}
W(R)\cong \varprojlim W_m(R).
\end{align*}
Hence $W(R)$ is complete and separated with respect to the topology defined by the ideals $v^m(W(R))$. We will refer to this as the $v$-adic topology. 

For every non-negative integer $m$ we have the following truncated variant of the Witt frame. 

\begin{ex}[Truncated Witt semi-frames]
	For a $p$-adic ring $R$ and a non-negative integer $m$, let $W_m(R)^\oplus$ be the quotient of $W(R)^\oplus$ by the graded ideal
	\begin{align*}
	V_m(R) = \bigoplus_{n \ge 0} (v^m(W(R)) \cdot t^{-n}) \oplus \bigoplus_{n \ge 1} v^m(W(R)).
	\end{align*}
	To be precise, for $n \ge 1$, $V_m(R)_n = v^m(W(R))$ is viewed as a $W(R)$-submodule of $v(W(R)) = I(R) = W(R)^\oplus_n$. The map $\tau: W(R)^\oplus \to W(R)$ extends to a map $\tau^m: W_m(R)^\oplus \to W_m(R)$, so by Lemma \ref{lem-qfcheck} the pair $(W_m(R)^\oplus, \tau)$ constitutes a semi-frame, called the \textit{$m$-truncated Witt semi-frame} for $R$. 
\end{ex}

\begin{rmk}
	The truncated Witt semi-frames are not associated with frames in general. Indeed, the Frobenius $f$ on $W_m(R)$ has image in the smaller ring $W_{m-1}(R)$, and does not determine an endomophism of $W_m(R)$ unless $pR= 0$. In the latter case the semi-frame $(W_m(R)^\oplus, \tau)$ is associated with a frame, but this frame differs slightly from the truncated Witt frame given in \cite[Example 2.1.6]{Lau2018}, which uses $I_{m+1}(R)$ for each graded piece above zero.
\end{rmk}

\begin{ex}[Truncated relative Witt semi-frames] \label{ex-truncated}
	Let $m$ be a non-negative integer and let $B \to A$ be a PD-thickening of $p$-adic rings. Let $W_m(B/A)^\oplus$ be the quotient of $W(B/A)^\oplus$ by the graded ideal
	\begin{align*}
	V_m(B/A)= \bigoplus_{n \ge 1} v^m(W(B))\cdot t^{-n} \oplus \bigoplus_{n \ge 0} v^m(W(B)),
	\end{align*}
	where $v^m(W(B))$ is embedded into $I(B) \oplus J$ via the first factor. Define maps $t_n: (W_m(B/A)^\oplus)_{n+1} \to (W_m(B/A)^\oplus)_n$ as follows: for $n \ge 1$, $t_n$ is multiplication by $p$ on the first component and the identity on $J$, and $t_0$ is the map
	\begin{align*}
	t_0: I_m(B) \oplus J \to W_m(B), \ (v(a), x) \mapsto r_m(v(a)) + r_m(\log^{-1}[x,0,\dots]),
	\end{align*} 
	where $r_m: W(B) \to W_m(B)$ is the truncation homomorphism. These maps determine a pair $\underline{W_m}(B/A) = (W_m(B/A)^\oplus, \tau)$ by Remark \ref{rmk-semiframe}, which constitutes a semi-frame by Lemma \ref{lem-qfcheck}. Since we have an isomorphism
	\begin{align*}
	W_m(B) / t(I_m(B)\oplus J) \cong A,
	\end{align*}
	$\underline{W_m}(B/A)$ defines a semi-frame over $A$, which we call the \textit{$m$-truncated relative Witt semi-frame} for $B\to A$. 
\end{ex}

We close this section by giving a Nakayama lemma for finite graded modules over the graded ring associated with a semi-frame, following \cite[Lemma 3.1.1, Corollary 3.1.2]{Lau2018}. Let $(S, \tau)$ be a semi-frame over a ring $R$. Denote by $\nu: S \to R$ the ring homomorphism which extends the natural projection $S_0 \to R$ by zero on all graded pieces away from $S_0$. 

\begin{lemma}\label{lem-nak}
	Let $(S,\tau)$ be a semi-frame, and let $M$ be a finite graded $S$-module.
	\begin{enumerate}[\textup{(}i\textup{)}]
		\item If $M \otimes_{S,\nu} R = 0$, then $M = 0$.
		\item Let $N$ be another finite graded $S$-module, and suppose $M$ is projective. Then a homomorphism $f: N \to M$ of graded $S$-modules is bijective if and only if its reduction $\bar{f}: N \otimes_{S,\nu}R \to M \otimes_{S, \nu} R$ is bijective.
	\end{enumerate}
\end{lemma}
\begin{proof}
	The proof of (i) is identical to the proof of \cite[Lemma 3.1.1]{Lau2018} since $tS_1 \subseteq \textup{Rad}(S_0)$. Part (ii) is an immediate consequence of (i), as in \cite[Corollary 3.1.2]{Lau2018}.
\end{proof}
\subsection{Complete semi-frames}

In this section we develop a technical framework for frames which arise as the limit of a sequence of semi-frames, in a sense which we will make precise. This will be used in the next section to prove descent for displays over relative Witt frames.

For this section, let $S$ be a $\zz$-graded ring, and let $V^\bullet$ be a sequence of graded ideals 
\begin{align*}
V^m = \bigoplus_{n \in \zz} V^m_n
\end{align*}
in $S$ such that $V^{m+1}_n \subseteq V^m_n$ for all $m, n$. For every $m$, denote by $S^m$ the quotient $S / V^m$. Explicitly, 
\begin{align*}
S^m = \bigoplus_{n \in \zz} (S_n / V^m_n).
\end{align*}
If $M$ is a finite projective graded $S$-module, then for every $m$, the quotient $M / V^mM$ is a finite projective graded $S^m$-module, with graded pieces 
\begin{align*}
(M/V^mM)_n = M_n / (M_n \cap V^mM).
\end{align*}

\begin{Def} Let $S$ and $V^\bullet$ be as above, and let $M$ be a graded $S$-module.
	\begin{enumerate}[(a)]
		\item The \textit{graded completion of $S$ with respect to $V^\bullet$} is 
		\begin{align*}
		S^\wedge := \bigoplus_{n \in \zz} \left(\varprojlim_m S_n / V^m_n\right).
		\end{align*}	
		The graded ring $S$ is \textit{$V^\bullet$-adic} if the natural homomorphism of graded rings $S \to S^\wedge$ is an isomorphism.
		\item The \textit{graded completion of $M$ with respect to $V^\bullet$} is
		\begin{align*}
		M^\wedge := \bigoplus_{n \in \zz} \left( \varprojlim_m M_n / (M_n \cap V^mM) \right).
		\end{align*}
		We say $M$ is \textit{$V^\bullet$-adic} if the natural graded $S$-module homomorphism $\varphi_M: M \to M^\wedge$ is an isomorphism.
	\end{enumerate}
\end{Def}
\begin{lemma}\label{lem-complete}
	Let $S$ be $V^\bullet$-adic, and suppose that $M$ is a finite projective graded $S$-module. Then $M$ is $V^\bullet$-adic.
\end{lemma}
\begin{proof}
	The proof reduces to the case where $M$ is a finite free graded $S$-module, which is immediate.
\end{proof}

\begin{Def}\label{def-systems}
	Let $S$ and $V^\bullet$ be as above. Define $\textup{PGrMod}((S^m)_m)$ to be the category whose objects are systems $(M^m)_{m \in \mathbb{N}}$ of finite projective graded $S^m$-modules equipped with graded $S^{m+1}$-module homomorphisms
	\begin{align*}
	\theta^m: M^{m+1} \to M^m
	\end{align*}
	which induce isomorphisms $M^{m+1} \otimes_{S^{m+1}} S^m \xrightarrow{\sim} M^m$. If $(M^m)_{m \in \mathbb{N}}$ and $(N^m)_{m \in \mathbb{N}}$ are two objects in $\textup{PGrMod}((S^m)_m)$, then a morphism between them is a collection of graded $S^m$-module homomorphisms $M^m \to N^m$ which are compatible with the $\theta^m$-maps.
\end{Def}

If $M$ is an object in $\textup{PGrMod}(S)$, then for every $m$, $M/V^mM$ is an object in $\textup{PGrMod}(S^m)$. 
This assignment determines a functor
\begin{align}\label{eq-fun}
\textup{PGrMod}(S) \to \textup{PGrMod}((S^m)_m).
\end{align}

\begin{prop}\label{prop-truncate}
	Let $(S,\tau)$ be a semi-frame for $R$, and let $V^\bullet = (V^m)_{m \in \mathbb{N}}$ be a sequence of graded ideals in $S$ such that $S$ is $V^\bullet$-adic. Suppose
	\begin{enumerate}[$($i$)$]
		\item For each $m$, there exists a ring homomorphism $\tau^m: S^m \to S^m_0$ such that $(S^m,\tau^m)$ is a semi-frame and such that the natural homomorphism of graded rings $S \to S^m$ induces a morphism of semi-frames $(S,\tau) \to (S^m,\tau^m)$.
		\item For every $m$, $S_0^m / tS_1^m = R$, and the homomorphisms $S_0 \to S_0^m$ lift the identity on $R$.
		\item For every finite projective $R$-module $M$ there exists a finite projective $S_0$-module $M'$ along with an isomorphism of $R$-modules $M'\otimes_{S_0} R \cong M$.
	\end{enumerate}
	Then the functor \textup{(}\ref{eq-fun}\textup{)} is an equivalence of categories.
\end{prop}
\begin{proof}
	We define an quasi-inverse functor as follows: Let $(M^m)_{m \in \mathbb{N}}$ be an object in $\textup{PGrMod}((S^m)_m)$, so $M^m = \bigoplus_{n \in \zz} M_n^m$ for every $m$, and define $M = \bigoplus M_n$, where $M_n = \varprojlim_m M^m_n$. Here the transition maps $M^{m+1}_n \to M^m_n$ are given by $\theta^m_n$, i.e. by the $n$th graded piece of $\theta^m$. We claim $M$ is a finite projective graded $S$-module.	
	
	Consider the finite projective graded $R$-module $\overline{L}:= M^1 \otimes_{S^1, \nu^1} R$. By (iii), there is a finite projective graded $S_0$-module $L$ such that $L \otimes_{S_0} R \cong \overline{L}$. Define $N:= L \otimes_{S_0} S$. Then $N$ is a finite projective graded $S$-module. Because the maps $\theta^m: M^{m+1} \to M^m$ are surjective for every $m$, we see that the induced map $M \to M^k$ sending $(a^m)_m \in M_n = \varprojlim_m M^m_n$ to $a^k \in M^k_n$ is also surjective, and therefore so too is the homomorphism $M \to \overline{L} = M^1 \otimes_{S^1, \nu^1} R$. Then the identity of $\overline{L}$ lifts to a homomorphism of graded $S$-modules $\psi: N \to M.$ Note that conditions (i) and (ii) imply that $\nu: S \to R$ factors through $S^m$ for every $m$, so, in particular, $V^m \subseteq \ker(\nu:S \to R)$ for every $m$. Let us denote by $\nu^m$ the induced map $S^m \to R$. 
	
	The composition $N \to M \to M^m$ factors through $N / V^mN$, inducing a graded $S^m$-module homomorphism $\psi_m: N / V^mN \to M^m$ for every $m$. Further, $N \to \overline{L}$ factors through $N / V^m N$ since $V^m \subseteq \ker(\nu: S \to R)$, and $M \to \overline{L}$ factors through $M^m$, so $\psi_m$ also lifts the identity of $\overline{L}$. Therefore 
	$(N/V^mN) \otimes_{S^m,\nu^m} R \to M^m \otimes_{S^m,\nu^m} R$ is an isomorphism, and by Lemma \ref{lem-nak} (ii), $\psi_m$ is an isomorphism of graded $S^m$-modules. By definition these isomorphisms satisfy $\theta^m \circ \psi_{m+1} = \psi_m \circ (\theta')^m$, where $(\theta')^m$ is the natural surjection $N / V^{m+1}N \to N/ V^m N$. Altogether we see 
	\begin{align*}
	\varprojlim_m N_n / (V^mN\cap N_n) \cong \varprojlim_m M^m_n
	\end{align*}
	for every $n$, so $M \cong N$ as graded $S$-modules by Lemma \ref{lem-complete}, and $M$ is indeed a finite projective graded $S$-module. Also, Lemma \ref{lem-complete} and the isomorphism
	\begin{align*}
	M / V^mM \cong N/ V^mN \cong M^m
	\end{align*}
	show that these functors are quasi-inverse to one another.
\end{proof}

Let $A$, $B$ and $R$ be $p$-adic rings, and let $B \to A$ be a PD-thickening with kernel $J$. Recall the graded ideals $V^\bullet_R = (V_m(R))_{m \ge 1}$ and $V^\bullet_{B/A} = (V_m(B/A))_{m \ge 1}$ defined in Example \ref{ex-truncated}.
\begin{lemma}
	The frames $\underline{W}(R)$ and $\underline{W}(B/A)$ defined in the previous section satisfy conditions $($i$)$ - $($iii$)$ in Proposition \ref{prop-truncate}. 
\end{lemma}
\begin{proof}
	Let us first prove that condition (iii) is satisfied. Since $R$ is $p$-adic, it follows from \cite[Proposition 3]{Zink2002} that $W(R)$ is complete and separated with respect to $I(R)$. Then every finite projective $R$-module lifts to a finite projective $W(R)$-module. Similarly every finite projective $B$-module lifts to $W(B)$, and every finite projective $A$-module lifts to $B$ since $J$ is a locally nilpotent ideal.
	
	Now, $W(R)^\oplus$ is graded complete with respect to the ideals $V_m(A)$ because both $W(R)$ and $I(R)$ are complete with respect to the ideals $v^m(W(R))$. Similarly $W(B/A)^\oplus$ is graded complete with respect to $V_m(B/A)$. For the semi-frames $(S^m, \tau^m)$ we take $(W_m(R)^\oplus, \tau^m)$ in the case of $(W(R)^\oplus, \tau)$, and we take $(W_m(B/A)^\oplus, \tau^m)$ in the case of $(W(B/A)^\oplus, \tau)$. Conditions (i) and (ii) of Proposition \ref{prop-truncate} are easily verified in each of these cases.
\end{proof}

\subsection{Descent for the relative Witt frame} \label{sub-descent}

As $R$ varies, the frame $\underline{W}(R)$ is naturally a functor of $R$. In fact, this association determines an fpqc sheaf in frames because the functors $R \mapsto W(R)$ and $R \mapsto I(R)$ both determine fpqc sheaves of abelian groups. Denote by $\textup{PGrMod}_W$ the fibered category over $\textup{Nilp}_{\zz_p}$ whose fiber over $R$ in $\textup{Nilp}_{\zz_p}$ is the category $\textup{PGrMod}(\underline{W}(R))$. By \cite[Lemma 4.3.2]{Lau2018}, \textup{PGrMod}$_W$ is an fpqc stack over $\textup{Nilp}_{\zz_p}$. The goal of this section is to prove the analog of this statement, replacing $W(R)^\oplus$ with $W(B/A)^\oplus$. We need to be a little careful here, because the behavior of the relative Witt frame differs from that of the Witt frame. In particular, we must replace the fpqc topology with the \'etale topology.

Let us begin by checking some \'etale-local properties of finite projective graded modules over semi-frames.  The following lemma is analogous to \cite[Lemmas 2.10 - 2.12]{Daniels2019}. 
\begin{lemma}\label{lem-localproperties}
	Suppose $\underline{S}$ is an \'etale sheaf of semi-frames on $\text{Spec }R$, with the property that $\uS(R') = (S(R'), \tau(R'))$ is a semi-frame for $R'$ for all \'etale $R$-algebras $R'$. If $R \to R'$ is a faithfully flat \'etale ring homomorphism, then the following hold:
	\begin{enumerate}[$($i$)$]
		\item \label{lem-descentseq} If $M$ is a finite projective graded $S(R)$-module, then there is an exact sequence 
		\begin{center}
			\begin{tikzcd}[column sep = small]
			0 
			\arrow[r]
			&M 	
			\arrow[r] 
			&M \otimes_{S(R)} S(R')
			\arrow[r, yshift=2pt] \arrow[r, yshift=-2pt] 
			&M \otimes_{S(R)} S(R' \otimes_R R')
			\end{tikzcd}
		\end{center}
		where the arrows are induced by applying $S$ to the usual exact sequence
		\begin{center}
			\begin{tikzcd}[column sep = small]
			0 
			\arrow[r]
			&R
			\arrow[r] 
			&R'
			\arrow[r, yshift=2pt] \arrow[r, yshift=-2pt] 
			&R'\otimes_R R'
			\end{tikzcd}
		\end{center}
		\item A finite projective graded $S(R)$-module $M$ is of type $I = (i_1, \dots, i_n)$ if and only if $M_{S(R')}$ is of type $I$.
		\item A sequence $ 0 \to M \to N \to P \to 0$ of finite projective graded $S(R)$-modules is exact if and only if it is exact after base change to $S(R')$.
		\item \label{lem-dispmorphism} Suppose additionally that $\uS$ is a sheaf of frames on $\Spec R$. If $\underline{M} = (M,F)$ and $\underline{M}'= (M', F')$ are displays over $\uS$, then a homomorphism $\psi: M \to M'$ of graded $S(R)$-modules is a morphism of displays if and only if $\psi_{R'}$ is a morphism of displays.
	\end{enumerate}
\end{lemma}
\begin{proof}
	For (i), since $M$ is finite projective we can reduce to the case where $M$ is finite free as a graded $S(R)$-module. This in turn reduces to the case $M = S(R)$, for which the result holds because $S$ is an \'etale sheaf of $\zz$-graded rings. 
	
	The proof for (ii) follows from the fact that the rank of a finite projective module is invariant under base change.
	
	The proof of (iii) is formally the same as that of \cite[Lemma 2.12]{Daniels2019}. The only nontrivial assertion is that if the sequence is exact after base change, then $M \to P$ is surjective. But since $\uS$ is a sheaf of semi-frames, Nakayama's lemma (Lemma \ref{lem-nak}) applies, so it is enough to check $M \otimes_{S(R), \nu} R \to P \otimes_{S(R),\nu} R$ is surjective. But this follows from surjectivity of $M_{S(R')} \to P_{S(R')}$ and faithful flatness of $R \to R'$.
	
	Let us prove (iv). If $\psi$ is a morphism of displays then $\psi_{R'}$ is as well. For the converse, we need to prove $(F')^\sharp \circ \sigma^\ast \psi$ and $\tau^\ast \psi \circ F^\sharp$ agree as homomorphism of finite projective $S(R)_0$-modules. We know this holds after base change to $S(R')_0$, so it is enough to prove the base change functor from the category of finite projective $S(R)_0$-modules to the category of finite projective $S(R')_0$-modules is faithful. But this is easy to see because by (i) the homomorphism $M \to M \otimes_{S(R)_0} S(R')_0$ is injective.
\end{proof}

Now we narrow our focus to the relative Witt frame. Let $A$ be a ring in $\textup{Nilp}_{\zz_p}$, and let $B \to A$ be a PD-thickening. In order to treat the finite and infinite cases uniformly, denote by $\underline{W_\infty}(B/A)$ the frame $\underline{W}(B/A)$. If $A'$ is any \'etale $A$-algebra, then there exists a unique \'etale $B$-algebra $B'$ along with a isomorphism of $A$-algebras $B'\otimes_B A \cong A'$ (see \cite[\href{https://stacks.math.columbia.edu/tag/039R}{Tag 039R}]{stacks-project}, for example). Moreover, if $J = \ker(B \to A)$, then $\ker(B' \to A') = JB'$, and by flatness the divided powers on $B \to A$ extend to $B' \to A'$, see \cite[\href{https://stacks.math.columbia.edu/tag/07H1}{Tag 07H1}]{stacks-project}. In this way the assignment
\begin{align}\label{eq-b/a}
A' \mapsto \underline{W_m}(B'/A')
\end{align}
becomes a functor from the category of \'etale $A$-algebras to the category of semi-frames for any $m \ge 1$ (including $\infty$).  
\begin{lemma}\label{lem-sheafqf}
	Let $1 \le m \le \infty$. The functor $($\ref{eq-b/a}$)$ defines an \'etale sheaf of semi-frames over $\textup{\textup{\'Et}}_A$.
\end{lemma}
\begin{proof}
	Let $A \to A'$ be a faithfully flat \'etale morphism with $B' \to A'$ lifting $B \to A$. Define $A'' = A' \otimes_A A'$, and let $B'' = B' \otimes_B B'$, which is the unique \'etale $B'$-algebra lifting $A''$. Let $J = \ker(B \to A)$ and define $J'$ and $J''$ analogously. Then the proof reduces to showing that
	\begin{align*}
	0 \to I_m(B) \oplus J \to I_m(B') \oplus J' \rightrightarrows I_m(B'') \oplus J''
	\end{align*}
	is exact, which follows from \'etale descent for $B$-modules.
\end{proof}

\begin{rmk}
	The frame $\underline{W}(B/A)$ over $A$ is a $p$-adic frame in the sense of \cite[Def. 4.2.1]{Lau2018}, so by \cite[Lem. 4.2.3]{Lau2018}, we can associate to it an \'etale sheaf of frames $\underline{S}$ such that $\underline{S}(A) = \underline{W}(B/A)$. However, by \cite[Ex. 4.2.7]{Lau2018}, this sheaf will not agree in general with the \'etale sheaf of frames $A' \mapsto \underline{W}(B'/A')$ described above.
\end{rmk}

For $1 \le m \le \infty$, denote by $\textup{PGrMod}^m_{B/A}$ the fibered category over $\textup{\'Et}_A$ whose fiber over an \'etale $A$-algebra $A'$ is $\textup{PGrMod}(W_m(B'/A')^\oplus)$, where $B'$ is the unique \'etale $B$-algebra with $B'\otimes_B A \cong A'$. Before we prove that $\textup{PGrMod}^{m}_{B/A}$ is a stack, let us first prove a useful lemma.

\begin{lemma}\label{lem-basechange}
	Let $1 \le m < \infty$. Suppose $B \to A$ is a PD-thickening in $\textup{\textup{Nilp}}_{\zz_p}$, and suppose $A \to A'$ is \'etale with lift $B \to B'$.
	\begin{enumerate}[\textup{(}i\textup{)}]
		\item The natural graded ring homomorphism $W_m(B/A)^\oplus \to W_m(B'/A')^\oplus$ induces an isomorphism
		\begin{align*}
		W_m(B/A)^\oplus \otimes_{W_m(B)} W_m(B') \xrightarrow{\sim} W_m(B'/A')^\oplus.
		\end{align*}
		\item Let $A'' = A' \otimes_A A'$ and $B'' = B' \otimes_B B'$. Then the natural homomorphism of graded rings
		\begin{align*}
		W_m(B'/A')^\oplus \otimes_{W_m(B/A)^\oplus} W_m(B'/A')^\oplus \to W_m(B''/A'')^\oplus	
		\end{align*}
		is an isomorphism.
		\item Let $A''' = A' \otimes_A A' \otimes_A A'$ and $B''' = B' \otimes_B B' \otimes_B B'$. Then the natural homomorphism of graded rings
		\begin{align*}
			W_m(B'/A')^\oplus \otimes_{W_m(B/A)^\oplus} W_m(B'/A')^\oplus \otimes_{W_m(B/A)^\oplus} W_m(B'/A')^\oplus \to W_m(B'''/A''')^\oplus
		\end{align*}
		is an isomorphism.
		\item If $A \to A'$ is faithfully flat \'etale, then $W_m(B/A)^\oplus \to W_m(B'/A')^\oplus$ is faithfully flat.
	\end{enumerate}
\end{lemma}
\begin{proof}
	For (i), since tensor products commute with direct sums, it is enough to prove this for each graded piece of $W_m(B'/A')^\oplus$. For graded pieces with $n \le 0$ this is clear, so we need only prove
	\begin{align*}
	(I_m(B) \oplus J) \otimes_{W_m(B)} W_m(B') \cong I_m(B') \oplus J'.
	\end{align*}
	This further reduces to proving the statement for $I_m(B')$ and for $J'$. By \cite[Proposition A.8]{LZ2004}, the homomorphism $W_m(B) \to W_m(B')$ is \'etale, and $w_0: W_m(B')\to B'$ induces an isomorphism
	\begin{align}\label{eq-w0}
	B \otimes_{W_m(B)} W_m(B') \cong B'.
	\end{align}
	Then by taking the tensor product of $0 \to I_m(B) \to W_m(B) \to B \to 0$ with $W_m(B')$ and applying the five lemma we obtain $I_m(B) \otimes_{W_m(B)} W_m(B') \cong I_m(B')$. 
	
	Finally, flatness of $W_m(B')$ over $W_m(B)$ implies $J \otimes_{W_m(B)} W_m(B') \cong W_m(B') J$, and by definition of the $W_m(B)$ action on $J$, we have $W_m(B')J = B'J$. Hence
	\begin{align*}
	J \otimes_{W_m(B)} W_m(B') \cong W_m(B')J \cong B'J \cong J \otimes_B B'.
	\end{align*}
	The result follows since $J \otimes_B B'\cong J'$.
	
	To prove (ii) and (iii) we first prove an auxiliary statement. Let $A \to A_1$ and $A \to A_2$ be \'etale ring homomorphisms with lifts $B \to B_1$ and $B \to B_2$. Then $B_1 \otimes_B B_2$ is an \'etale $B$-algebra lifting $A_1 \otimes_A A_2$. We claim the natural homomorphism of graded rings
	\begin{align}\label{eq-aux}
		W_m(B_1/A_1)^\oplus \otimes_{W_m(B/A)^\oplus} W_m(B_2/A_2)^\oplus \to W_m(B_1 \otimes_B B_2 / A_1 \otimes_A A_2)^\oplus
	\end{align}
	is an isomorphism. Granting (\ref{eq-aux}) for the moment, we can prove (ii) and (iii). Indeed, (ii) follows immediately by taking $A_1 = A_2 = A'$, and (iii) follows by combining (ii) with (\ref{eq-aux}) for $A_1 = A''$ and $A_2 = A'$. 
	
	Now let us prove (\ref{eq-aux}). For the sake of brevity let us write $S = W_m(B/A)^\oplus$, $S^1 = W_m(B_1/A_1)^\oplus$, $S^2 = W_m(B_2/A_2)^\oplus$, and $S^{1,2} = W_m(B_1 \otimes_B B_2 / A_1 \otimes_A A_2)^\oplus$. By \cite[Cor. 9.4]{Borger} (take $R = \zz$, and $E = p\zz$ to obtain the $p$-typical Witt vectors in \textit{loc. cit.}), since $B \to B_1$ and $B \to B_2$ are \'etale, the natural map $W_m(B_1) \otimes_{W_m(B)} W_m(B_2) \to W_m(B_1 \otimes_B B_2)$ is an isomorphism. Combining this with part (i), we have a chain of isomorphisms
	\begin{align*}
		S^1 \otimes_S S^2 \xrightarrow{\sim} S \otimes_{S_0} (S^1_0 \otimes_{S_0} S^2_0) \xrightarrow{\sim} S \otimes_{S_0} S^{1,2}_0.
	\end{align*}
	But since $A \to A_1 \otimes_A A_2$ is \'etale with lift $B \to B_1 \otimes_B B_2$, we have $S \otimes_{S_0} S^{1,2}_0 \xrightarrow{\sim} S^{1,2}$ by part (i) again. One checks that the composition $S^1 \otimes_S S^2 \xrightarrow{\sim} S^{1,2}$ is the desired map.
	
	Now, for (iii), to show $W_m(B/A)^\oplus \to W_m(B'/A')^\oplus$ is faithfully flat it is enough to prove the same is true of $W_m(B) \to W_m(B')$, by part (i). But $W_m(B) \to W_m(B')$ is \'etale by \cite[Proposition A.8]{LZ2004} and $\Spec{W_m(B')} \to \Spec{W_m(B)}$ is surjective because $B \to B'$ is faithfully flat and $W_m(B) \to B$ and $W_m(B') \to B'$ are PD-thickenings for $A$ in $\textup{Nilp}_{\zz_p}$. 
\end{proof}

\begin{prop}\label{prop-descent}
	For $1 \le m \le \infty$, the fibered category $\textup{\textup{PGrMod}}^{m}_{B/A}$ is an \'etale stack over $\textup{\textup{\'Et}}_A$. 
\end{prop}
\begin{proof}
	This proof is similar to the proof of \cite[Lemma 4.3.1]{Lau2018}. Indeed, that the morphisms form a sheaf follows from Lemma \ref{lem-localproperties} (i) by the arguments in \textit{loc. cit.}. 
	
	Let us first prove that objects descend in the case where $m$ is finite. Let $A \to A'$ be an \'etale faithfully flat homomorphism, and let $B'$ be the unique \'etale $B$-algebra such that $B'\otimes_B A \cong A'$. Note that $B \to B'$ is then also faithfully flat. Let $A'' = A'\otimes_A A'$, and $B'' = B'\otimes_B B'$. Suppose $M'$ is a finite projective graded $W_m(B'/A')^\oplus$-module equipped with a descent datum over $W_m(B''/A'')^\oplus$. By parts (ii), (iii), and (iv) of Lemma \ref{lem-basechange}, we may apply faithfully flat descent for graded modules over graded rings (see, e.g. \cite[Corollary III.1.4]{CVO1998}) to obtain a graded $W_m(B/A)^\oplus$-module $M$ such that $M \otimes_{W_m(B/A)^\oplus} W_m(B'/A')^\oplus \cong M'$. Additionally, faithful flatness of $W_m(B/A)^\oplus \to W_m(B'/A')^\oplus$ implies that $M$ is finite and projective as an $W_m(B/A)^\oplus$-module, hence it is finite and projective as a graded $W_m(B/A)^\oplus$-module by \cite[Lemma 3.0.1]{Lau2018}. This completes the proof of descent for finite $m$.
	
	Now let $S = W(B/A)^\oplus$, $S'= W(B'/A')^\oplus$, and $S'' = W(B''/A'')^\oplus$. For every finite $m$, let $V^m = V_m(B/A)$, $S^m = S/ V^m$, and define the variants for $S'$ and $S''$ in the obvious way. Let $p_1: S' \to S''$, resp. $p_2: S' \to S''$ be the map induced by $a \mapsto a \otimes 1$ resp. $a \mapsto 1 \otimes a$ from $A'$ to $A'\otimes_A A'$. Define similarly $p_1^m, p_2^m: (S')^m \to (S'')^m$. Denote by \textup{PGrMod}$(S\to S')$ the category of finite projective graded $S'$ modules $M'$ equipped with descent data, i.e., equipped with isomorphisms	$\alpha: p_1^\ast(M') \xrightarrow{\sim} p_2^\ast(M')$ which satisfy the cocycle condition. We must show that the natural functor $\textup{PGrMod}(S) \to \textup{PGrMod}(S\to S')$ is an equivalence.
	
	Define the category \textup{PGrMod}$((S^m \to (S')^m)_m)$ consisting of systems $(M^m)$ of finite projective graded $(S')^m$-modules $M^m$ along with isomorphisms $\theta^m: M^{m+1} \otimes_{(S')^{m+1}} (S')^m \xrightarrow{\sim} M^m$	and descent data $\alpha^m: (M^m)^{p_1^m} \xrightarrow{\sim} (M^m)^{p_2^m}$ such that $\alpha^m \circ \theta^m = \theta^m \circ (\alpha^{m+1} \otimes \id_{(S'')^m})$ for all $m$. Then the first part of the proof implies that the natural functor
	\begin{align}\label{eq-equiv1}
	\textup{PGrMod}((S^m)_m) \to \textup{PGrMod}((S^m \to (S')^m)_m)
	\end{align}
	is an equivalence of categories. Further, it is straightforward to check that the functor (\ref{eq-fun}) respects descent data, and that the equivalence in Proposition \ref{prop-truncate} extends to an equivalence
	\begin{align}\label{eq-equiv2}
	\textup{PGrMod}(S\to S') \xrightarrow{\sim} \textup{PGrMod}((S^m \to (S')^m)_m).
	\end{align}
	The result follows by combining the equivalence $\textup{PGrMod}(S) \xrightarrow{\sim} \textup{PGrMod}((S^m)_m)$ with (\ref{eq-equiv1}) and (\ref{eq-equiv2}).
\end{proof}

\

\bibliographystyle{amsalpha}
\bibliography{Refs}

\end{document}